\documentclass[10pt]{amsart}
\usepackage{amsmath,amsfonts,amssymb,amsthm,amscd,latexsym,euscript,xypic, verbatim}
\usepackage[all]{xy}
\usepackage{graphics}
\usepackage{epsfig}
\usepackage{multirow} 
\swapnumbers
\theoremstyle{plain}

\newtheorem*{thmA}{Theorem A}
\newtheorem*{thmB}{Theorem B}

\newtheorem*{thm1.2}{(1.2) Theorem}
\newtheorem*{thm1.3}{(1.3) Theorem}
\newtheorem*{thm1.4}{(1.4) Theorem}
\newtheorem*{propA*}{Proposition A}
\newtheorem*{propB*}{Proposition B}
\newtheorem*{thmC*}{Theorem C}
\newtheorem*{propD*}{Proposition D}
\newtheorem{prop}{Proposition}[section]
\newtheorem{thm}[prop]{Theorem}
\newtheorem{cor}[prop]{Corollary}
\newtheorem{lemma}[prop]{Lemma}

\theoremstyle{definition}
\newtheorem{point}[prop]{}

\newtheorem{Def}[prop]{Definition}
\newtheorem*{Def*}{Definition}

\newtheorem{notation}[prop]{Notation}
\newtheorem*{notation*}{Notation}
\newtheorem*{question*}{Question}

\def\rmono{\rto|<\hole|<<\ahook}
\def\lmono{\lto|<\hole|<<\bhook}
\def\urmono{\urto|<\hole|<<\ahook}

\newcommand{\cala}{\mathcal{A}}
\newcommand{\calb}{\mathcal{B}}
\newcommand{\calc}{\mathcal{C}}
\newcommand{\cald}{\mathcal{D}}
\newcommand{\cale}{\mathcal{E}}
\newcommand{\calf}{\mathcal{F}}
\newcommand{\calg}{\mathcal{G}}

\newcommand{\calm}{\mathcal{M}}

\newcommand{\calr}{\mathcal{R}}
\newcommand{\cals}{\mathcal{S}}
\newcommand{\calt}{\mathcal{T}}

\newcommand{\ra}{\rightarrow}

\newcommand{\lra}{\longrightarrow}

\xyoption{2cell}
\UseAllTwocells
\title{On the classification of fibrations}
\author[Blomgren, Chach\'{o}lski]
{M.~Blomgren$^1$ \quad  W.~Chach\'{o}lski$^2$}
\address{Martin Blomgren\\
      Department of Mathematics\\
      KTH\\
      S 10044 Stockholm\\
      Sweden}
\email{blomgr@kth.se}

\address{Wojciech Chach\'{o}lski\\
      Department of Mathematics\\
      KTH\\
      S 10044 Stockholm\\
      Sweden}
\email{wojtek@math.kth.se}

\begin{document}
\maketitle
\footnotetext[1]{Supported by grant KAW 2005.0098 from the Knut and Alice Wallenberg Foundation.}
\footnotetext[2]{Partially supported by G\"oran Gustafsson Stiftelse and VR grants.}

%%%%%%%%%%%%%%%%%%%%%%%%%%%%%%%%%%%%%%%
%%%%%%%%%%%%%% Introduction  %%%%%%%%%%%%%%%%%%%
%%%%%%%%%%%%%%%%%%%%%%%%%%%%%%%%%%%%%%%
\section{Introduction}      
Classification questions are often about  understanding components of a category.  
It is not unusual however that with  a category one can associate a unique homotopy  type of  a simplicial set whose set of components coincide with the set of components of the category. Such a space carries  more information about the  category than just the set of its components.  For example:
\begin{Def}\label{def intro FibXF}
Let  $X$ and $F$ be spaces. Define
$\text{Fib}(X,F)$ to be the category whose objects are maps $f\colon A\to B$ where $B$ is weakly equivalent to $X$ and  the homotopy fiber of $f$, over any base point in $B$, is weakly equivalent to $F$. The set of morphisms in  $\text{Fib}(X,F)$ between $f\colon A\to B$ and $f'\colon A'\to B'$, consists of pairs of weak equivalences $\phi\colon A\to A'$ and $\psi\colon B\to B'$ for which $f'\phi=\psi f$\@.  The composition of morphisms is induced by the usual composition of maps.
\end{Def}

A classical result states that the components of  $\text{Fib}(X,F)$ can be enumerated by the set of homotopy classes $ [X,B\text{we}(F,F)]$ where $B\text{we}(F,F)$ is the  classifying space of the topological monoid of weak equivalences of $F$\@. This is a classical theorem proved  by Stasheff in~\cite{MR0154286} and re-proved and generalized by May in~\cite{MR0370579}.

Instead of looking  at the set of components, it is  more desirable to study the entire moduli space of fibrations.
% with a given homotopy type of the base and the fiber. 
One would like to understand the homotopy type of the  category $\text{Fib}(X,F)$ and not just the set of its components.  Naively one can  try to form the nerve of $\text{Fib}(X,F)$ and then identify its homotopy type. However, since  $\text{Fib}(X,F)$ is not equivalent to a small category, this can not be done so directly.  Instead we are going to show that $\text{Fib}(X,F)$ has what we call a {\bf core} (see~\ref{def core}) which is a small category whose nerve  approximates the homotopy type of $\text{Fib}(X,F)$\@.
Our classification statement can be then formulated as follows:

\begin{thmA}
The category\/ 
$\text{\rm Fib}(X,F)$ has a core whose nerve  admits a map to\/ $B\text{\rm we}(X,X)$\@. This map has a  section and its homotopy fiber is weakly equivalent to the mapping space\/ 
$\text{\rm map}(X,B\text{\rm we}(F,F))$.
\end{thmA}

It turns out that the above theorem is a particular case of a much more general statement that holds in an arbitrary model category. The purpose of such a generalization is not only to show that analogous classification statements hold in much broader  context. Statements that hold in an arbitrary model category often have more conceptual proofs in which one  does not need to use the nature of objects considered  but rather basic fundamental facts from  homotopy theory. In this way arguments are becoming more transparent. It was Dwyer and Kan who first realized and proved that such general classification statements are true. In their sequence of papers that includes~\cite{MR581012, MR584566, MR744846, MR777259, MR789792} they develop a strategy and techniques for dealing  with classification questions. An important part of their program was a discovery of continuity in model categories. They showed that an arbitrary model category has mapping spaces whose homotopy type is unique. They also gave a particular model for them through the use of so called hammocks. 

In this paper we follow, in principle, the plan of Dwyer and Kan. Our realization of their strategy is different however.   For example homotopical smallness is an essential ingredient in our work.
Another important difference is that we use a model for mapping spaces developed in~\cite{MR2443229}.
%One   difference is our  use of a model for mapping spaces developed in~\cite{MR2443229}.  Another concept  essential  in our approach  is homotopical smallness. 
%n our view understanding this concept does require non-trivial arguments that at several instances are not the  naive straightforward ones. 
%Another important difference is that we use a model for mapping spaces developed in~\cite{MR2443229}. 
%This model is particularly useful to understand the homotopy type of categories that frequently occur in classification questions.
Our general statement is about the homotopy type of the category of weak equivalences $\calm_{\text{we}}$ of a model category $\calm$\@. Its objects are the objects of $\calm$ and morphisms are all the weak equivalences in $\calm$\@. To understand its homotopy type, we study the components of  $\calm_{\text{we}}$\@. For an object $X$ in $\calm$, we denote by $X_{\text{\rm we}}$ the full subcategory of $\calm_{\text{we}}$ that consists of all the objects in $\calm$ which are weakly equivalent to  $X$. This subcategory is also called a {\bf component} of $\calm_{\text{we}}$\@. Our key result states (see~\ref{thm funintoxwe}):

\begin{thmB}
Let $I$ be a small category and $X$ be an object in a model category $\calm$\@. The category of  functors\/ $\text{\rm Fun}(I,X_{\text{\rm we}})$ has a core which is weakly equivalent to the mapping 
space\/ $\text{\rm map}(N(I),B\text{\rm we}(X,X))$\@.
\end{thmB}
%
%This theorem illustrates well how to recover  continuity  discovered by Dwyer and Kan is encoded by
%by model categories. Let $X$  be a CW complex. Consider the monoid $\text{haut}(X,X)$ of all the continuous maps  $f:X\ra X$ that are weak equivalences. This is just a group like discrete monoid and %so its nerve has the homotopy type of an Eilenberg-MacLane space $K(\pi,1)$. This monoid is a subcategory of the component $X_{\text{we}}$. According to Theorem B  the core of this component has the homotopy type of the classifying space of the {\bf space} of weak equivalences $\text{we}(X,X)$. Thus out of this component $X_{\text{we}}$ the homotopy type of the {\bf space} $\text{we}(X,X)$ can be recovered.

%%%%%%%%%%%%%%%%%%%%%%%%%%%%%%%%%%%%%%%
%%%%%%%%%%%%%% Notation  %%%%%%%%%%%%%%%%%%%
%%%%%%%%%%%%%%%%%%%%%%%%%%%%%%%%%%%%%%%
\section{Categorical constructions and notation}
%In this section we  set some notation and recall  several categorical constructions.
To describe sets we use 
  Zermelo-Fraenkel set theory with the axiom of choice. 
%Note that the constructions recalled in this section do not require   smallness.

\begin{point} 
The term category is used as defined in~\cite[Section 7]{MR0156879}.\  
 The category $\cala^{\text{op}}$ is the opposite category of $\cala$\@.
  A natural transformation between functors
$f,g\colon\calb\to \cala$  is denoted by  $\phi\colon g\to f$\@.  It consists of  
morphisms $\phi_b\colon g(b)\to f(b)$ in $\cala$ for any object $b$ in $\calb$ such that $f(\beta)\phi_{b_1}=\phi_{b_0}g(\beta)$
for any morphism $\beta\colon b_1\ra b_0$ in $\calb$.

The category of sets is denoted by Sets. A category is  small if it has a set of  objects.\
$\text{Cat}$ denotes the category of small categories and $\Delta$
its  full subcategory  whose objects are posets 
$[n]:=\{0<\ldots< n\}$ for $n\geq 0$. 
 %$.\ $\Delta$ is the full subcategory of such posets in Cat. 
 
The symbol $B\subset A$  denotes the fact that 
$B$ is a subset or a subcategory or a subspace, depending on whether  $A$ is  a set or a category or a space.

%Let  $I_0\subset I_1\subset\cdots\subset \calc$ be a sequence of small subcategories of $\calc$. The union
%$\cup_{k\geq 0} I_k\subset \calc$ is the  small subcategory whose objects, respectively  morphisms, are these objects, respectively  morphisms, in $\calc$  which belong to $I_k$ for some $k$. 

\end{point}
\begin{point}\label{pt modelcat}
 The symbol $\calm$  always denotes a model category which  in addition to  the standard axioms {\bf MC1}-{\bf MC5}  (see e.g.\ ~\cite{MR2002k:55026, MR96h:55014, 99h:55031, MR36:6480}),
we require  $\calm$ to be closed under arbitrary colimits and limits, to have  a functorial fibrant replacement, and that  any commutative square in $\calm$ on the left below can be extended {\bf functorially } to a commutative diagram on the right
with the indicated morphisms being cofibrations and acyclic fibrations:
\[\xymatrix@R=14pt@C=15pt{
X\rto^{f}\dto_-{\alpha_1} &Y\dto^-{\alpha_2}\\
X'\rto^{f'} & Y'
} \ \ \ \ \  \ \ \ \ \ \ \ \ 
\xymatrix@R=14pt@C=20pt{
X\dto_-{\alpha_1}\ar@/^17pt/[rr]|{f}\rmono  & P(f)\ar@{->>}[r]^{\simeq} \dto|(.45){P(\alpha_1,\alpha_2)\  \ }& Y\dto^-{\alpha_2}\\
X'\ar@/_14pt/[rr]|{f'}\rmono  &P(f') \ar@{->>}[r]^{\simeq} &Y'
} \]

%\begin{itemize}
%\item  any commutative square in $\calm$ on the left below can be extended {\bf functorially} to a %commutative diagram on the right
%with the indicated morphisms being cofibrations and acyclic fibrations:
%\[\xymatrix@R=14pt@C=15pt{
%X\rto^{f}\dto_-{\alpha_1} &Y\dto^-{\alpha_2}\\
%X'\rto^{f'} & Y'
%} \ \ \ \ \  \ \ \ \ \ \ \ \ 
%\xymatrix@R=14pt@C=20pt{
%X\dto_-{\alpha_1}\ar@/^16pt/[rr]|{f}\rmono  & P(f)\ar@{->>}[r]^{\simeq} \dto|(.45){P(\alpha_1,\alpha_2)}& Y\dto^-{\alpha_2}\\
%X'\ar@/_15pt/[rr]|{f'}\rmono  &P(f') \ar@{->>}[r]^{\simeq} &Y'
%} \]
%\item $\calm$ admits a functorial fibrant replacement;
%\item $\calm$ is closed under arbitrary colimits and limits;
%Functoriality means that   $(\alpha_1,\alpha_2)\mapsto P(\alpha_1,\alpha_2)$  describes a functor  out of the arrow category of  $\calm$ to $\calm$.
%Commutativity of the diagram on the right above means that  the morphisms
%$X\hookrightarrow P(f)$ and $P(f)\twoheadrightarrow Y$
%are
%natural transformations.
 %: a functor
%$R:\calm\ra \calm$ and a natural weak equivalence  $X\ra R(X)$ such that $R(X)$ is  a fibrant object.
%\end{itemize}
The restriction of $P$ to the  the full subcategory of arrows in $\calm$ of the form
 $\emptyset\to  X$ is a  functorial cofibrant replacement in $\calm$.
%We are going to denote by the same symbol the restriction of $P$ to the full subcategory of arrows of the form
% $\emptyset\ra  X$ where $\emptyset$ is the initial object in $\calm$. In this way we obtain a functorial cofibrant replacement $P(X)\ra X$ in $\calm$.

Simplicial sets are also called spaces and their category with the  standard model structure (see for example~\cite{MR2001d:55012}) is denoted by $\text{Spaces}$.  The full subcategory of
$\text{Spaces}$ whose objects are the standard simplices $\Delta[n]$ is isomorphic to $\Delta$. 
%
%The category of simplicial sets with the standard model structure (see for example~\cite{MR2001d:55012}) is denoted by $\text{Spaces}$. Simplicial sets are also called spaces and morphisms between them maps. The standard simplices are denoted by $\Delta[n]$. The full subcategory of $\text{Spaces}$ whose objects are the standard simplices is isomorphic to $\Delta$. 
\end{point}

\begin{point}
\label{point systemandGr}
A  system $\calf$ of categories indexed by a category $\calc$ consists of a category $\calf_c$ for any object
$c$ in $\calc$ and a functor $\calf_{\alpha}\colon \calf_{c_0}\to \calf_{c_1}$ for any morphism $\alpha\colon c_1\to c_0$ in $\calc$ (note contravariancy).
These functors are required to satisfy: $\calf_{\text{id}_c}=\text{id}$
for any object $c$; and $\calf_{\alpha\alpha'}=\calf_{\alpha'}\calf_{\alpha}$
for any  morphisms $\alpha'\colon c_2\to c_1$ and $\alpha\colon c_1\to c_0$.

A subsystem $\calg\subset \calf$    consists of a subcategory $\calg_c\subset \calf_c$ for any object $c$ in $\calc$
such that, for any morphism $\alpha\colon c_1\to c_0$,   $\calf_{\alpha}$ takes $\calg_{c_0}$ to  $\calg_{c_1}$.
%The categories $\calg_c$ with the restrictions of  $\calf_{\alpha}$'s
%form a   system of categories denoted by 
%$\calg$.
\end{point}
\begin{point}\label{pt Grothendieck}
Let $\calf$ be a system of categories indexed by $\calc$\@. Its Grothendieck construction, denoted by $\text{Gr}_{\calc}\calf$,
is  the category whose objects are pairs $(c,x)$ where $c$ is an object in $\calc$ and $x$ in $\calf_c$.
The set of morphisms between $(c_1,x_1)$ and $(c_0,x_0)$  is the set of  pairs    
$(\alpha\colon :c_1\to c_0,\beta\colon  x_1\to \calf_{\alpha}(x_0))$ where $\alpha$ is a morphism in $\calc$ and $\beta$ is a morphism in $\calf_{c_1}$\@. 
The composition of $(\alpha'\colon c_2\to  c_1,\beta'\colon x_2\to \calf_{\alpha'}(x_1))$ 
and 
$(\alpha\colon c_1\to c_0,\beta\colon x_1\to\calf_{\alpha}(x_0) )$ is defined to  be the pair:
\[(c_2\xrightarrow{\alpha'} c_1\xrightarrow{\alpha} c_0, \ 
x_2\xrightarrow{\beta'} \calf_{\alpha'}(x_1) \xrightarrow{\calf_{\alpha'}(\beta)} 
\calf_{\alpha'}(\calf_{\alpha}(x_0))=\calf_{\alpha\alpha'}(x_0))
\]

The  {\bf projection} $\pi\colon\text{Gr}_{\calc}\calf\to \calc$ is  the functor that assigns to  an object 
$(c,x)$ (resp.\ morphism $(\alpha,\beta)$) in $\text{Gr}_{\calc}\calf$ the object $c$ (resp.\ morphism
$\alpha$)  in $\calc$.
For any object $c$ in $\calc$ the functor $\calf_c\to \text{Gr}_{\calc}\calf$ which
assigns  to an object $x$   the  pair $(c,x)$ and to  a morphism $\beta\colon x\to y$ 
the  pair $(\text{id}_c, \beta)$ is called the {\bf standard inclusion}.
\end{point}

\begin{point}\label{pt undercat}
Let $f\colon \calb\to \cala$ be a functor and $a$ be an object in $\cala$\@. The {\bf under category} 
\mbox{$a\!\uparrow\! f$} has 
   pairs $(b,\alpha:a\to  f(b))$ of an object  $b$ in $\calb$ and a morphism $\alpha\colon a\to f(b)$
 in $\cala$ as objects.  The set of morphisms
between   $(b_1, \alpha_1)$ and $(b_0, \alpha_0)$ in \mbox{$a\!\uparrow\! f$} is the set of  
morphisms $\beta\colon b_1\to b_0$ in $\calb$ for which   $f(\beta)\alpha_1=\alpha_0$\@.
%, i.e., the following
%diagram commutes:
%\[
%\xymatrix@R=12pt{
%& a\ar@/_5pt/[dl]_(.4){\alpha_1}\ar@/^5pt/[dr]^(.4){\alpha_0}\\
%f(b_1)\rrto^{f(\beta)} & & f(b_0)
%}
%\]
The category \mbox{$a\!\uparrow \!\text{id}_{\cala}$} is also denoted by \mbox{$a\!\uparrow\!\!\cala$}\@.
By forgetting the second component we obtain a functor $(a\!\uparrow\! f)\to \calb$   called  {\bf forgetful}.

The {\bf over category} $f\!\downarrow\! a$  has pairs $(b, \alpha\colon f(b)\ra a)$ of an object $b$ in $\calb$ and a morphism $\alpha$Ê in $\cala$ as objects. The set of morphisms between $(b_1,\alpha_1)$ and
$(b_0,\alpha_0)$ is the set of morphisms $\beta\colon b_1\to b_0$ in $\calb$ for which $\alpha_0f(\beta)=\alpha_1$.

Consider the following commutative diagram of functors:
\[\xymatrix@R=11pt@C=15pt{
\cald\rto^{g}\dto_{e} & \calc\dto^{h}\\
\calb\rto^{f} & \cala
}\]
Let  $c$ be an object  in $\calc$. The symbol  $(e,h)\colon c\!\uparrow\! g\to h(c)\!\uparrow\! f$ denotes the functor
that maps an object $(d, \alpha)$ to  $(e(d), h(\alpha))$ and 
a morphism $\alpha\colon d_1\to d_0$ to  $e(\alpha)$.

Let $\gamma \colon a_1\to a_0$ be a morphism in $\cala$. The functor
$\gamma\!\uparrow\! f\colon a_0\!\uparrow\! f\to a_1\!\uparrow\! f$  assigns  to
$(b, \alpha)$  the object $(b, \alpha\gamma)$ and to a morphism $\beta$  the same  $\beta$\@.
The assignment
$a\mapsto (a\!\uparrow\! f)$ and $\lambda\mapsto (\gamma\!\uparrow\! f)$
is  a system of categories indexed by $\cala$  denoted by $-\!\uparrow\! f$\@.
Its Grothendieck construction $\text{Gr}_{\cala}(-\!\uparrow\! f)$ is isomorphic to  a category whose   objects are
pairs $(b, \alpha\colon a\to f(b))$  of an object $b$ in $\calb$ and a morphism $\alpha$ in $\cala$. The set of morphisms between 
$(b_1, \alpha_1\colon a_1\to f(b_1))$ and $(b_0, \alpha_0\colon a_0\to f(b_0))$
consists of   pairs  $(\beta\colon b_1\to b_0, \gamma\colon a_1\to a_0)$
where $\gamma$ is a morphism in $\cala$ and  $\beta$  is a morphism in $\calb$  
making the following square commutative:
\[\xymatrix@R=11pt@C=15pt{
a_1\rto^{\gamma}\dto_{\alpha_1} & a_0\dto^{\alpha_0}\\
f(b_1)\rto^{f(\beta)} & f(b_0)
}\]
%The composition of the morphisms is given by the compositions of the components.

The functor
$\hat{\pi}\colon \text{Gr}_{\cala}(-\!\uparrow\! f)\to \calb$  assigns to an object 
$(b, \alpha)$ in $\text{Gr}_{\cala}(-\!\uparrow\! f)$ the object 
$b$ in $\calb$ and to a morphism $(\beta,\gamma)$ in  $\text{Gr}_{\cala}(-\!\uparrow\! f)$ the morphism $\beta$ in $\calb$\@. 
The functor $\hat{f}\colon \calb\to \text{Gr}_{\cala}(-\!\uparrow\! f)$  assigns to an object $b$ in $\calb$ the object $(b, \text{id}_{f(b)})$ in $\text{Gr}_{\cala}(-\!\uparrow\! f)$ and to a morphism $\beta$ in $\calb$ the morphism $(\beta, f(\beta))$ in   $  \text{Gr}_{\cala}(-\!\uparrow\! f)$\@. These functors fit into the following commutative diagram:
\[\xymatrix@R=11pt{
\calb\rto|(.3){\hat{f}}\ar@/_5pt/[dr]_{f}\ar@/^15pt/[rr]|{\text{id}} & \text{Gr}_{\cala}(-\!\uparrow\! f)
\dto^{\pi} \rto_(.65){\hat{\pi}} &
\calb\\
& \cala
}\]
Note  that there is a natural transformation between the functors $\hat{f}\hat{\pi}$  and 
$\text{id}_{\text{Gr}_{\cala}(-\!\uparrow\! f)}$ 
which for an object $(b,\alpha\colon a\to f(b))$ is given by the morphism
$(\text{id}_b, \alpha)$.
\end{point}

\begin{point}\label{pt functors}
$\text{Fun}(I,\calc)$ is  the category of  functors indexed by a {\bf small} category $I$
with values in  a category $\calc$ and natural transformations as morphisms. 
The set of  
natural transformations between  $F\colon I\to \calc$ and $G\colon I\to \calc$ is denoted by $\text{Nat}(F,G)$.

Let $f\colon I\to J$ be a functor of  small categories.\ 
The composition with $f$ functor is denoted by $f^{\ast}\colon \text{Fun}(J,\calc)\to \text{Fun}(I,\calc)$.
It assigns to a natural transformation $\{\psi_j\}_{j\in J}$  the natural transformation $\{\psi_{f(i)}\}_{i\in I}$.
%The functor
%$f^{\ast}\colon \text{Fun}(J,\calc)\to \text{Fun}(I,\calc)$ assigns to a functor $F$  the composition $Ff$ and to a natural transformation $\{\psi_j\}_{j\in J}$  the natural transformation $\{\psi_{f(i)}\}_{i\in I}$. 
 If $\calc $ is closed under colimits, then $f^{\ast}$ has a left adjoint  
$f^{k}\colon  \text{Fun}(I,\calc)\to \text{Fun}(J,\calc)$ called the left Kan extension.
The assignment $I\mapsto \text{Fun}(I,\calc)$ and $f\mapsto f^{\ast}$ is a system of categories indexed by  $\text{Cat}$\@.
% Its Grothendieck construction is isomorphic to a category whose objects are
%functors indexed by  small categories with values in $\calc$\@. The set of morphisms
%between $F_1\colon I_1\to\calc$ and $F_0\colon I_0\to\calc$ consists of  pairs $(f\colon I_1\to I_0,
%\psi\colon F_1\to F_0 f)$ of a functor $f$ and a natural transformation $\psi$.

 A natural transformation  $\phi\colon F\to G$ in
$\text{Fun}(I,\calm)$  is  called a weak equivalence if $\phi_i\colon F(i)\to G(i)$ is   a weak equivalence for any $i$ in $I$\@.\/
 $\text{Ho}(\text{Fun}(I,\calm))$ denotes the localization of 
$\text{Fun}(I,\calm)$ with respect to weak equivalences which  
exists by~\cite[Theorem 11.3]{MR2002k:55026}.
% A functor $F:I\ra \calm$ is called
%{\bf homotopy constant} if it is isomorphic in $\text{Ho}(\text{Fun}(I,\calm))$ 
%to a constant functor.
\end{point}

\begin{point}\label{not component}
A connected component of a category $\calc$ containing an object $c$ is the class of all objects $d$ for which there is a finite  sequence of  morphisms 
$d = x_0 \to x_1\gets \cdots \gets x_n = c$ in $\calc$\@. The symbol $\pi_0(\calc)$ denotes the discrete category
(identities are the only morphisms) whose objects are connected components of $\calc$
and $\pi_0\colon \calc\to \pi_0(\calc)$
denotes the unique functor mapping  an object to its component.  
\end{point}

\noindent
{\bf Part I.   Categories and homotopy.}         
In this part we review two ways of doing homotopy theory on categories. In one the homotopy relation is induced by natural transformations.  This notion however is too strong for us. We need weak equivalences. A standard way of introducing them is to transport  weak notions from simplicial sets using  the nerve. This works for small categories.  Our aim is to extend the weak notions to categories such as $\text{Fib}(X,F)$
% We need the weak notions however for more general categories such as $\text{Fib}(X,F)$.\/ % which are not equivalent to small categories. 
%This will be done for categories 
that can be approximated by small categories. Dwyer and Kan called them homotopically small (\cite{MR584566}). We call them essentially small and their study  %homotopy techniques to study  them 
is the main aim of this part.

%%%%%%%%%%%%%%%%%%%%%%%%%%%%%%%%%%%%%%%
%%%%%%%%%%%%%% Stronghomotopy  %%%%%%%%%%%%%%%%%%%
%%%%%%%%%%%%%%%%%%%%%%%%%%%%%%%%%%%%%%%
\section{Categorical  homotopy}\label{sec basicdictionary}
Here is a standard  dictionary of homotopy notions on arbitrary categories:
% (without any smallness assumption):
\begin{itemize}
\item 
Functors $f,g\colon \calb\to \cala$ are {\bf homotopic} if there are
functors $\{h_k\colon \calb\to \cala\}_{0\leq k\leq n}$ and natural transformations 
$f=h_0\to h_1\gets \cdots\gets h_n= g$.%, connecting $f$ and $g$. 
\item 
$f\colon\calb\to \cala$ is  a {\bf homotopy equivalence} if it has a homotopy inverse, i.e.,  a functor 
$g\colon\cala\to \calb$  for which $fg$ and $gf$ are homotopic to $\text{id}_\cala$ and $\text{id}_\calb$\@.
%If such $f$ exists, then  $\cala$ and $\calb$ are said to be {\bf homotopy equivalent}.
%
%
%there is a functor
%$g:\cala\ra \calb$ such that $gf$ is homotopic to $\text{id}_{\calb}$ and $fg$ is homotopic to $\text{id}_{\cala}$. Such a $g$ is called a homotopy inverse to $f$.
%\item   A category $\cala$ is called  a {\bf homotopy retract} of a category  $\calb$ if there are functors
%$f:\cala\ra \calb$ and $r:\calb\ra \cala$ for which $rf$ is homotopic to $\text{id}_{\cala}$.
\item  $f\colon \calb\to \cala$ is   a {\bf strong fibration} if, for any morphism $\gamma\colon a_1\to a_0$ in
$\cala$,  $\gamma\!\uparrow\! f\colon a_0\!\uparrow\! f\to a_1\!\uparrow\! f$ (see~\ref{pt undercat}) is a homotopy equivalence. 
\item  A {\bf strong homotopy pull-back} is a commutative square of functors:
\[\xymatrix@R=11pt@C=15pt{
\cald\rto^{g}\dto_-{e} & \calc\dto^-{h}\\
\calb\rto^{f} & \cala
}\]
 such that $f$ is a strong fibration and 
  $(e,h)\colon c\!\uparrow\! g\to h(c)\!\uparrow\! f$
(see~\ref{pt undercat})  is a homotopy
equivalence  for any object $c$ in $\calc$.
\end{itemize}

For example $\hat{f}\colon\calb\to \text{Gr}_{\cala}(-\!\uparrow\! f)$
(see~\ref{pt undercat}) is  a homotopy equivalence. 

%%%%%%%%%%%%%%%%%%%%%%%%%%%%%%%%%%%%%%%
%%%%%%%%%%%%%% Small  %%%%%%%%%%%%%%%%%%%
%%%%%%%%%%%%%%%%%%%%%%%%%%%%%%%%%%%%%%%
\section{Small categories}
\label{sec smallcat}
It is the  nerve construction  used to translate weak  homotopy notions from $\text{Spaces}$ into $\text{Cat}$. 
It is a functor $N\colon\text{Cat}\to \text{Spaces}$ 
that assigns to a small category $I$ a simplicial set $N(I)$  whose set of
$n$-dimensional simplices is,  for $n>0$,  the set of $n$-composable morphisms in $I$ and,  if $n=0$, it is the set of objects in $I$. 
%If $[n]$ denotes the poset 
%$\{0<\cdots< n\}$, then $N(I)=\text{Fun}([-],I)$.   
% For example $N([n])$ is isomorphic to $\Delta[n]$. 

Here is a basic dictionary (compare with notions recalled  in Section~\ref{sec basicdictionary}):
\begin{itemize}
\item A functor $f\colon J\to I$ of small categories is called a {\bf weak equivalence}, if $N(f)\colon N(J)\to N(I)$ is a weak equivalence of
spaces.
\item  A functor $f\colon J \to I$ of small categories is called a {\bf  quasi-fibration} if
 $\alpha\!\uparrow\! f\colon i_0\!\uparrow\! f \to i_1\!\uparrow\! f$ is a weak equivalence
 for any   $\alpha\colon i_1\to i_0$ in $I$. 
\item 
A commutative square  of small categories
is called a {\bf homotopy pull-back}  if after applying the nerve
we obtain a homotopy pull-back  of spaces.
\end{itemize}

It is well known  that if   functors of small categories are homotopic as functors, then
their nerves are homotopic as  maps.
Consequently  a homotopy equivalence, resp.\ strong fibration, of small categories is a weak equivalence, resp.\ quasi fibration. To prove that for small categories strong homotopy pull-backs are homotopy pull-backs we need 
the so called Thomason and Puppe's theorems:

\begin{prop}\label{prop ThomasonPuppe}
\stepcounter{subsection}
Let $I$ be a small category and
$F,G\colon I^{\text{\rm op}}\to \text{\rm Cat}$ be  functors.%, which are simply  systems of small categories indexed by $I$ (see~\ref{point systemandGr}).
\begin{enumerate}
\item $N(\text{\rm Gr}_{I}F)$ is weakly equivalent to\/ $\text{\rm hocolim}_{I^{\text{\rm op}}} N(F)$.
\item  Let  $f\colon F\to G$ be a natural transformation. Assume $f_i$ is a weak equivalence
for any object  $i$ in $I$\@.
Then\/ $\text{\rm Gr}_I f$ is  a weak equivalence.
\item Assume  $F({\alpha})\colon F({i_0})\to F({i_1})$ is a weak equivalence for any  $\alpha\colon i_1\to i_0$. Then, for any object $i$ in $I$, the following is a 
homotopy pull-back square:
\[\xymatrix@R=11pt@C=15pt{F_i \rto\dto & \text{\rm Gr}_{I}F\dto^-{\pi}\\
[0]\rto^{i} & I
}\]
where $\pi$ is the projection,  $F(i)\ra \text{\rm Gr}_{I}F$ is the standard inclusion (see~\ref{pt Grothendieck}),
and $i\colon [0]\to I$ is the functor that sends the object\/ $0$ to  $i$.
\item  Let $f\colon J\to I$ be a quasi-fibration of small categories. Then, for any object $i$ in $I$, the following is a
homotopy  pull-back square:
\[\xymatrix@R=11pt@C=15pt{
i\!\uparrow\! f\rrto^{\text{\rm forget}}\dto & &J\dto^-{f}\\
i\!\uparrow\! I\rrto^{\text{\rm forget}} & & I
}\]
\end{enumerate}
\end{prop}
\begin{proof}
Statement (1) is the so called Thomason's theorem~\cite{MR80b:18015}, see also~\cite{MR2002k:55026}. Statement (2) is a  consequence of (1). Statement (3) follows from
(1) and the so called Puppe's theorem~\cite{MR51:1808}, see also~\cite{MR2222504}. Statement (4) is   Quillen's Theorem A; it follows easily from Statement (3).
\end{proof}
Here is a method to verify that a square of small categories is a homotopy pull-back.
It should be compared with~\cite{MR1310749} where the term homotopy pull-back
is used to describe squares that are not homotopy pull-backs as defined in this paper.

\begin{prop}\label{prop hopullback}
Let the following be  a commutative diagram of small categories:
\[\xymatrix@R=11pt@C=15pt{
L\rto^{g}\dto_-{e} & K\dto^-{h}\\
J\rto^{f} & I
}\]
Assume %(compare with the definition of a strong homotopy pull-back in Section~\ref{sec basicdictionary}):
 $f\colon J\to I$ is a quasi-fibration and $(e,h)\colon k\!\uparrow\! g\to h(k)\!\uparrow\! f$ is a weak equivalence
for any object $k$ in $K$.
Then the above square is homotopy pull-back.
\end{prop}
\begin{proof}
The  assumptions imply that $g$ is also a quasi-fibration. Thus by~\ref{prop ThomasonPuppe}.(4)
the induced map on the homotopy fibers of $N(g)$ and $N(f)$ is a weak equivalence.
\end{proof}

%%%%%%%%%%%%%%%%%%%%%%%%%%%%%%%%%%%%%%%
%%%%%%%%%%%%%% Essentiallysmall  %%%%%%%%%%%%%%%%%%%
%%%%%%%%%%%%%%%%%%%%%%%%%%%%%%%%%%%%%%%

\section{Essentially small categories}
The aim of this section
is to explain how certain ``big'' categories can be approximated  by small categories. It is based
on the  following  well known fact:
\begin{lemma}
\label{lemma-telescope}
If $I_0 \subset I_1 \subset \ldots$ is a    sequence of small categories where each inclusion is a weak equivalence, then  $I_0 \subset \text{\rm colim } I_n=\cup_{n\geq 0} I_n$ is 
also a weak equivalence.
\end{lemma}

Here is our key definition:

\begin{Def}\label{def core}
A {\bf core} of a category $\calc$ is a small subcategory  $I\subset \calc$ such that, for any small subcategory $J\subset \calc$ with $I\subset J$,  there is a small subcategory $K\subset \calc$  for which  $J\subset K$ and the inclusion $I\subset K$ is a weak equivalence.
A category is said to be {\bf essentially small} if it has a core.
\end{Def}

For example if $\calc$ has  a small skeleton, then this skeleton is its  core.   
% It is also clear that $I\subset \calc$ is
%a core if and only if  $I^{\text{op}}\subset \calc^{\text{op}}$ is a core. In particular $\calc$ is essentially small if and only if
%$ \calc^{\text{op}}$ is so.

\begin{prop}\label{prop propesssmall} Let $\calc$ be a category.
\begin{enumerate}
\item If $I\subset \calc$ and $J\subset \calc$ are cores, then $I$ and $J$ are weakly equivalent.
%\item Let $S$ be a set and $\{I_s\subset \calc_s\}$ be cores. Then $\prod_{s\in S}I_s\subset \prod_{s\in S} \calc_s$ is a core.
\item A discrete essentially small category is small.
\item If $\calc$ is essentially small, then the components of $\calc$ form a set. If $I\subset \calc$ is a core, then this inclusion induces a bijection between $\pi_0(I)$ and  $\pi_0(\calc)$.
\end{enumerate}
\end{prop}

%To prove this proposition and for future use we need:
\begin{lemma}\label{lem propesssmall}%\hspace{1mm}
\begin{enumerate}
\item Let $I\subset \calc$ be a core and $I'\subset \calc$  a small subcategory containing
$I$. Then $I'\subset \calc$ is a core if and only if  $I\subset I'$  is a weak equivalence.
\item Let $J\subset \calc$ be a small subcategory and $I\subset \calc$ be a core. Then
there is a full subcategory $K\subset \calc$ such that $J\subset K\supset I$  and
 $I\subset K$ is a weak equivalence.
\end{enumerate}
\end{lemma}
\begin{proof}
\noindent (1):\quad 
Assume   $I\subset I'$ is a weak equivalence.
Let $J\subset \calc $ be  a small subcategory such that $I'\subset J$. Since $I\subset \calc$ is a core,
there is  a small subcategory $K\subset \calc$ for which  $J\subset K$ and  $I\subset K$ is a
weak equivalence. By the ``2 out of 3'' property the inclusion $I'\subset K$   is also
a weak equivalence. This shows that $I'\subset \calc$ is  a core.

Assume  $I'\subset \calc$ is  a core.  We   define inductively a sequence of small subcategories
$I_0\subset I'_{0}\subset I_1\subset I'_1\subset\cdots\subset \calc$\@. Set $I_0=I$ and $I'_0=I'$\@.
Assume  $n>0$\@.  Let $I_n\subset \calc$   be a small
subcategory containing $I'_{n-1}$  for which $I_{0}\subset I_n$ is a weak equivalence.
It exists since $I_0$ is a core in $\calc$.
Similarly, let $I'_n\subset \calc$ be a small subcategory containing $I_n$ for which
$I'_0\subset I'_n$ is a weak equivalence. %Such categories exist since $I_0$ and $I'_0$ are cores of $\calc$.
Note that $\bigcup_{n\geq 0}I_n=\bigcup_{n\geq 0}I'_n$. Moreover,
according to~\ref{lemma-telescope}, the inclusions $I_0\subset \bigcup_{n\geq 0}I_n=
\bigcup_{n\geq 0}I'_n\supset I'_0$ are weak equivalences. It follows that $I_0\subset I'_0$ is a weak
equivalence as well.
\smallskip

\noindent (2):\quad 
Define inductively a sequence of small subcategories
$I_0\subset K_0 \subset I_1\subset K_1\subset \cdots \subset \calc$.
Set $I_0=I$ and $K_0$ to  be the full subcategory in $\calc$ on  objects
in $I_0$ and $J$.  Assume  $n>0$\@.
Define $I_n\subset \calc$ to be a small subcategory such that $K_{n-1}\subset I_n$ and  $I_0\subset I_n$ is a weak equivalence.
%Such $I_n$ exists since $I_0$ is  a core of $\calc$. 
Define $K_n$ to  be the full subcategory
of $\calc$ on the set of  objects in $I_n$\@. Set  $K:=\bigcup_{n\geq 0} K_n$\@. Since
for any $n$, $K_n$ is a full subcategory in $\calc$, same is true for $K$\@. As
$K=\bigcup_{n\geq 0} I_n$ and  $I=I_0\subset I_n$ is a weak equivalence, for any $n$, 
% by~\ref{lemma-telescope}  
$I=I_0\subset K$ is also a weak equivalence.
\end{proof}

\begin{proof}[Proof of Proposition~\ref{prop propesssmall}]
\noindent (1):\quad 
According to~\ref{lem propesssmall}.(2) there is a small subcategory $K\subset \calc$ such that $I\subset K\supset J$ and
 $I\subset K$ is a weak equivalence. Since $I\subset \calc $ is a core we can 
use~\ref{lem propesssmall}.(1) to 
conclude that $K\subset \calc$ is also a core.  By assumption $J\subset \calc$ is a core. The inclusion $J\subset K$ is thus  a weak equivalence.
\smallskip

%\noindent (2):\quad  The statement is a consequence of the fact that a product of weak equivalences between small categories is again a weak equivalence.
%\smallskip

\noindent (2):\quad 
Just note that 
weak equivalences of  discrete  categories are isomorphisms.
\smallskip

\noindent (3):\quad 
Follows from the fact that   $\pi_0\colon\calc\to \pi_0(\calc)$ maps a core to a core.
%
% We claim that the image of $I$  via $\pi_0:\calc\ra \pi_0(\calc)$
%is a core of  $\pi_0(\calc)$.  Let $S$ be any subcategory  containing $\pi_0(I)$. Pick a set of representatives for the preimage of $S\setminus \pi_0(I)$ in $\calc$ and let $J$ be the full subcategory
%in  $\calc$ on the set of objects in $I$ and $S$. Then, since $I$ is a core, there is a small subcategory $K \subset \calc$ such that the composition $I \subset J \subset  K$ is a weak equivalence. 
%In particular this inclusion induces an isomorphism between $\pi_0(I)$ and $\pi_0(K)$.
%This shows that $\pi_0(I)$ is a core of $\pi_0(\calc)$. The category  $\pi_0(\calc)$ is therefore  small  by  (3). Bijection between $\pi_0(I)$ and $\pi_0(\calc)$ is then clear.
\end{proof}

By~\ref{prop propesssmall}.(1)  the homotopy type of a core  is a well define invariant. 
%we can associate a unique homotopy type to an essentially small category: the homotopy type of the nerve of its core.
%According to Proposition~\ref{prop propesssmall}.(1), we can associate a unique homotopy type to an essentially small category: the homotopy type of the nerve of its core. 
We can thus use it to introduce various homotopy notions on essentially small categories. For example, we can define homotopy groups of an essentially small category as the homotopy groups of the nerve of  its core. Usefulness of such invariants depend on  how they  behave under functors. 
%Naturality properties of cores are therefore important. 
For that we need to extend Definition~\ref{def core} to:

%To understand this  we need to extend Definition~\ref{def core} to systems of categories (see~\ref{point systemandGr}):

\begin{Def}\label{def essentially small system}
Let $\calf$ be a system of categories indexed by a small category $I$
(see~\ref{point systemandGr}). A {\bf core} of $\calf$ is a subsystem $F\subset \calf$ such that  $F_i\subset \calf_i $ is a core for any $i$\@. A system $\calf$ is called {\bf essentially small} if it has a core.
\end{Def}

\begin{prop}\label{prop funcorialcore} 
Let $\calf$ be a system of categories indexed by a small category $I$\@.
\begin{enumerate}
\item  $\calf$ is essentially small if and only if, for any $i$, $ \calf_i $ is  essentially small.
\item Assume that $F\subset \calf\supset F'$ are cores. Then there is a core $H\subset \calf$ such that
$F\subset H\supset F'$ and $H_i\subset \calf_i$ is a full subcategory for any object $i$ in $I$.
\item If $F\subset \calf$ is a core, then\/ $\text{\rm Gr}_IF\subset \text{\rm Gr}_I\calf$ is a core.
\end{enumerate}
\end{prop}
\begin{lemma}\label{lem keysystemess}
Let  $G_i\subset \calf_i$ be a small subcategory for any object $i$ in $I$. Then  there is a core $H\subset \calf$
such that $G_i\subset H_i$ and $H_i\subset \calf_i$ is a full subcategory for any  $i$.
\end{lemma}
\begin{proof}
We  construct  inductively a sequence of small subcategories, for any object $i$ in $I$,
$(G_i)_0\subset (H_i)_1\subset (G_i)_1\subset (H_i)_2\subset  (G_i)_2\subset\cdots \subset \calf_i$\@. Set $(G_i)_0:=G_i$\@. Assume  $n>0$\@.  Let $(H_i)_n\subset \calf_i$ to be a core such that
$(G_i)_{n-1}\subset (H_i)_n$\@. It exists  by~\ref{lem propesssmall}.(2).
Define $(G_i)_n$ to be the full subcategory of $\calf_i$ on the set of objects
$\bigcup_{\alpha\colon i\to j} \calf_{\alpha} ((H_j)_n)$
where the index $\alpha\colon i\to j$ runs over all possible morphisms in $I$ with domain $i$. 
The purpose of this definition is to ensure that, for any $n>0$,
\begin{enumerate}\renewcommand{\labelenumi}{\alph{enumi}.}
\item $(H_i)_n\subset \calf_i$ is a core for any object $i$ in $I$;
\item $(G_i)_n\subset \calf_i$ is a full subcategory  for any object $i$ in $I$;
\item $\calf_{\alpha}\colon\calf_j\to\calf_i$
takes $(H_j)_n$ to $(G_i)_n$ for any morphism $\alpha\colon i\to j$ in $I$.
\end{enumerate}
Define $H_i:=\cup_{n>0} (H_i)_n$\@.
The above properties  and~\ref{lem propesssmall}.(1) imply:
\begin{enumerate}\renewcommand{\labelenumi}{\alph{enumi}.}
\item $H_i\subset \calf_i$ is a core.
% which follows from~\ref{lem propesssmall}.(1), ~\ref{lemma-telescope},  and  the fact that 
%(%H_i)_n\subset \calf_i$ is a core for any $n>0$.
\item $H_i\subset\calf_i$ is a full subcategory.
%  since 
%$H_i=\cup_{n>0}(G_i)_n$ and $(G_i)_n\subset\calf_i$ is a full subcategory for any $n$.
\item $\calf_{\alpha}\colon\calf_j\to\calf_i$ takes $H_j$ to $H_i$ for any morphism
$\alpha\colon i\to j$ in   $I$. \qedhere
% i.e., $H_i$'s form
%a subsystem of $\calf$.
\end{enumerate}
%These are exactly the properties stated in the lemma.
\end{proof}
\begin{proof}[Proof of Proposition~\ref{prop funcorialcore}]
\noindent
(1) and (2) are direct consequences of~\ref{lem keysystemess}.
To prove (3) choose a small subcategory $J\subset  \text{\rm Gr}_I\calf$  containing  $ \text{\rm Gr}_I F$\@. For any object $i$ in $I$, let $G_i\subset \calf_i$ be the full subcategory on the set of all objects $x$ in $\calf_i$ for which $(i,x)\in J$\@. According to~\ref{lem keysystemess} there is a core $H\subset \calf$ such that $G_i\subset H_i$ for any  $i$\@. Consider the inclusions $\text{Gr}_IF\subset J\subset\text{Gr}_IH\subset \text{Gr}_I\calf$. Since   $F_i\subset H_i$ is a weak equivalence (see~\ref{lem propesssmall}), then so is  $\text{Gr}_IF\subset \text{Gr}_IH$
(see~\ref{prop ThomasonPuppe}.(2)).
\end{proof}

%We finish by proving that essential smallness is preserved by homotopy retracts:
% that a homotopy retract of an essentially small category is essentially small.
\begin{prop}\label{prop esssmallretract}
Let $f\colon \calb\to\cala$ and $r\colon \cala\to\calb$ be functors. Assume  $\cala$ is essentially small and
$rf\colon \calb\to\calb$ is homotopic to $\text{\rm id}_{\calb}$. Then $\calb$ is  essentially small and there are cores 
$A\subset \cala$ and $B\subset\calb$ for which the following diagram commutes:
\[\xymatrix@R=11pt@C=15pt{
B\ar@{^(->}[d]\rto^{f} &A\ar@{^(->}[d]\rto^{r} &B\ar@{^(->}[d]\\
\calb\rto^{f} & \cala\rto^{r} &\calb
}\]
\end{prop}
\begin{proof}
Choose  a sequence of natural transformations $rf= h_0\to \cdots\gets h_m= \text{id}_{\calb}$\@.
For  small subcategories $I\subset\cala$ and $K\subset\calb$,   define inductively a sequence of small subcategories of $\cala$ and $\calb$ that fit into the following commutative diagram:
\[\xymatrix@R=11pt@C=12pt{
I:=I_0\ar@{}[r]|-{\displaystyle \subset}  &
I_1\dto_r\ar@{}[r]|{\displaystyle \subset} &  I_2\dto_r\ar@{}[r]|{\displaystyle \subset}&
\cdots\ar@{}[r]|{\displaystyle \subset}  & \cala\dto^{r}\\
K:=K_0\ar@{}[r]|-{\displaystyle \subset}&
K_1\ar@{}[r]|{\displaystyle \subset} & K_2\ar@{}[r]|{\displaystyle \subset}&
\cdots\ar@{}[r]|{\displaystyle \subset}  & \calb
}\]
Let  $n>0$. Using~~\ref{lem propesssmall}.(2), define $I_n\subset\cala$ to be a core which is a full subcategory and contains the set of objects that either belong to $I_{n-1}$ or are of the form $f(b)$ where $b$  is in $K_{n-1}$\@.  Set $K_n$ to be the full subcategory of $\calb$ on the set of objects that either
belong to $K_{n-1}$, or
 are of the form $r(a)$ where $a$ is  in $I_n$, or
 are of the form $h_k(b)$ where $b$ is in $K_{n-1}$\@. % and $0\leq k\leq m$.
The purpose of this definition is to ensure:
\begin{enumerate}\renewcommand{\labelenumi}{\alph{enumi}.}
\item $f\colon  \calb\to\cala$
takes $K_{n-1}$ to  $I_n$ and
$r\colon\cala\to \calb$ takes  $I_n$ to $K_n$;
\item $I_n\subset \cala$ is a core for any $n>0$;
\item $h_k\colon\calb\to\calb$ takes $K_{n-1}$ to $K_n$ for any $0\leq k\leq m$;
\item $K_n\subset \calb$ is a full subcategory.
\end{enumerate}
Define
$I_\infty:=\cup_{n\geq 0} I_n$ and  $K_\infty:=\cup_{n\geq 0} K_l$.
The above requirements imply:
\begin{enumerate}\renewcommand{\labelenumi}{\alph{enumi}.}
\item there is a commutative  diagram
$\xymatrix@R=11pt@C=15pt{
K_{\infty}\ar@{^(->}[d]\rto^{f} &I_{\infty}\ar@{^(->}[d]\rto^{r} &K_{\infty}\ar@{^(->}[d]\\
\calb\rto^{f} & \cala\rto^{r} &\calb
}$
\item 
$I_{\infty}\subset \cala$ is a core;
\item $h_k\colon \calb\to\calb$ takes $K_{\infty}$ to $K_{\infty}$ for any
$0\leq k\leq m$;
\item $rf\colon K_{\infty}\to K_{\infty}$ is homotopic to the identity functor.
The appropriate ``zig-zag'' is obtained by restriction 
using fulness of $K_{\infty}$ in $\calb$.
 \end{enumerate}

We need to prove that  $K_{\infty}\subset \calb$ is a core.
Let $J\subset \calb$ be a small subcategory containing $K_{\infty}$.
The above construction applied to  $I_{\infty}\subset\cala$ and 
$J\subset \calb$ yields:
\[\xymatrix@R=11pt@C=15pt{
K_{\infty}\ar@{^(->}[d]\rto^-{f} &I_{\infty}\ar@{^(->}[d]\rto^-{r} &K_{\infty}\ar@{^(->}[d]\\
J_{\infty}\ar@{^(->}[d]\rto^-{f} & (I_{\infty})_{\infty} \ar@{^(->}[d]\rto^-{r} &
J_{\infty}\ar@{^(->}[d]\\
\calb\rto^-{f} & \cala\rto^-{r} &\calb
}\]
Since $(I_{\infty})_{\infty} \subset \cala$ is a core,  $I_{\infty}\subset (I_{\infty})_{\infty}$ is a weak equivalence (see~\ref{lem propesssmall}.(1)). Since  $rf\colon J_{\infty}\to J_{\infty}$ and $rf\colon K_{\infty}\to K_{\infty}$ are homotopic to the identity functors, $K_{\infty}\subset J_{\infty}$, as a homotopy retract of a weak equivalence, is a weak equivalence.
\end{proof}
\begin{cor}\label{cor esssmallhomeqinv}
Let $f\colon\calb\to \cala$ be a homotopy equivalence. Then $\cala$ is essentially small if and only if $\calb$ is.
\end{cor}

\section{Weak homotopy notions for essentially small categories}
Our aim is to extend the dictionary from Section~\ref{sec smallcat} to essentially small categories. 
%%%%%%%%%Change %%%%%%%%
%and prove statements analogous to those in~\ref{prop basictranslate}. 
A functor  $f\colon\calf_1\to\calf_0$  is  a system of categories indexed by the poset $[1]$\@. 
%We define a core of $f$ to  be a core of the  system indexed by $[1]$.  
Thus its core 
  consist of cores  $F_1\subset \calf_1$ and $F_0\subset \calf_0$
such that $f$ takes $F_1$ to  $F_0$.  The restricted functor 
 $f\colon F_1\to  F_0$  fits into a  commutative diagram:
 \[\xymatrix@R=11pt@C=15pt{
 F_1\ar@{^(->}[r]\dto_-{f} &  \calf_1 \dto^-{f} \\
F_0\ar@{^(->}[r]  & \calf_0
 }\]
  By the ``2 out of 3'' property of weak equivalences and~\ref{prop funcorialcore}.(2), 
if $f\colon F_1\to  F_0$ and  $f\colon F'_1\to F'_0$ are cores of  $f\colon \calf_1\to\calf_0$, then $f\colon F_1\to  F_0$ is a weak equivalence if and only if $f\colon F'_1\to F'_0$ is so.

Similarly, a commutative square on the left below is  a system of categories indexed by the poset  of all the subsets of $\{0,1\}$\@.  Its  core consists of cores
$F_{\emptyset}\subset \calf_{\emptyset}$, $F_0\subset\calf_0$, $F_1\subset\calf_1$, and 
$F_{0,1}\subset\calf_{0,1}$ making the right cube commutative:
\[\xymatrix@R=10pt@C=12pt{
& \calf_1\ddto^{g_1}\\
\calf_{0,1}\urto^{f_1}\ddto_{f_0} \\
& \calf_{\emptyset}\\
\calf_0\urto^{g_0}
}\ \ \ \ \ \ \ \  \ \ \ \ 
%\xymatrix@R=12pt@C=15pt{\calf_{0,1}\rto^-{f_1}\dto_-{f_0} & \calf_{1}\dto^-{g_1}\\
%\calf_{0}\rto^-{g_0} & \calf_{\emptyset}
%}\ \ \ \ \ \ \ \ 
\xymatrix@R=10pt@C=12pt{
& F_{1}\ar@{^(->}[rr]\ddto|\hole^(.3){g_1} & & \calf_1\ddto^{g_1}\\
F_{0,1}\ar@{^(->}[rr] \urto^{f_1}\ddto_{f_0}& & \calf_{0,1} \urto^(.4){f_1}\ddto_(.3){f_0}\\
& F_{\emptyset}\ar@{^(->}[rr]|\hole& & \calf_{\emptyset}\\
F_{0}\ar@{^(->}[rr] \urto^{g_0}& & \calf_0 \urto^{g_0}
}
\]
Again, by the ``2 out of 3'' property and~\ref{prop funcorialcore}.(2), 
if one   core of a commutative square is homotopy pull-back then so is any other.   This justifies:
\begin{Def}\label{def weqfhpull}\hspace{1mm}
\begin{itemize}
\item 
A functor  
is called 
a {\bf weak equivalence} if  it has a core which  is a weak equivalence.
\item A functor $f\colon\calb\to \cala$ is called a {\bf quasi-fibration} if
$\alpha\!\uparrow\! f\colon a_0\!\uparrow\! f\to\! a_1\uparrow\! f$ is a weak equivalence
 for any morphism $\alpha\colon a_1\to a_0$ in $\cala$.
%\begin{itemize}
%\item for any object $a$ in $\cala$, the category $a\uparrow f$ is essentially small and 
%\item for any morphism $\alpha:a_1\ra a_0$ in $\cala$, the functor
%$\alpha\uparrow f: a_0\uparrow f\ra a_1\uparrow f$ is a weak equivalence.
%\end{itemize}
\item A commutative square of functors  is called  {\bf homotopy
pull-back} if it has a core which is  homotopy pull-back.
\end{itemize}
\end{Def}
Note that a weak equivalence can  only be between  essentially small categories.

%Noe that if  $f\colon\calb\to \cala$ is a weak equivalence, then  $\cala$ and $\calb$ have to be essentially small.
%Analogously, if  $f:\calb\ra \cala$  is  a quasi fibration, then  $a\uparrow f$ is essentially small for any object $a$ in $A$.

\begin{prop}\label{prop esssmallweakequiv}
Let  $f\colon\calb\to \cala$ be a homotopy 
equivalence of essentially small categories. Then $f$ is a weak equivalence.
\end{prop}
\begin{proof}
Let $g$ be  a homotopy inverse to $f$\@.
Consider  a system of categories indexed by the free category on the  graph on the left given by the
diagram on the right:
%
 %These functors  can be regarded as a system of categories indexed by the free category on the  graph on the left.  Such a system is simply given  by the two functors on the right:
\[\xymatrix{
a\ar@/_7pt/[rr]|{\beta}& &b\ar@/_7pt/[ll]|{\alpha}
}\ \ \ \ \ \ \ \ \ \ \  \xymatrix{
\cala\ar@/^7pt/[rr]|{g}&& \calb\ar@/^7pt/[ll]|{f}
}\]
This system is essentially small (see~\ref{prop funcorialcore}.(1)) and  hence has a core. It  consists of small subcategories $A\subset \cala$ and $B\subset\calb$ that fit into a   commutative diagram:
\[\xymatrix@R=11pt@C=15pt{
A\rto^{g}\ar@{^(->}[d] & B\rto^{f}\ar@{^(->}[d] & A\ar@{^(->}[d]\\
\cala\rto^{g} & \calb\rto^{f}& \cala
}\]
To show that $f\colon A\to B$ is a weak equivalence, it is enough to
 prove that the compositions $fg\colon A\to  A$ and $gf\colon B\to B$ are weak equivalences.

%Since $fg:\cala\ra \cala$ is homotopic to the identity, there 
Choose a sequence of 
 natural transformations $fg= h_0\to \cdots\gets h_m= \text{id}_{\cala}$\@. The functors $\{h_k\colon\cala\to \cala\}_{0\leq k\leq m}$ form a system of categories indexed by:
\[\xymatrix{
0 & & &1 \ar@/_20pt/[lll]|{\alpha_0} \ar@/_10pt/[lll]|{\alpha_1}
\ar@{}[lll]|\vdots\ar@/^15pt/[lll]|{\alpha_m} 
}\]
It has a core given by   full subcategories $A_0\subset \cala$ and $A_1\subset \cala$ containing $A$
(see~\ref{lem keysystemess}).
Take the restrictions
$h_k\colon A_0\to A_1$\@. 
Since $A_0$ and $A_1$ are cores of $\cala$,   $h_m\colon A_0\to A_1$ is a weak equivalence as it is the restriction of the identity.  Use fullness of $A_1$ in $\cala$ to get 
a sequence of natural transformations
$h_0\to \cdots\gets h_m$ between these restrictions. 
 Thus  $h_0\colon A_0\to A_1$ is  a weak 
equivalence too and hence,
by the ``2 out of 3'' property, so is $fg\colon A\to A$\@.
By symmetry  $gf$ is also a weak equivalence.
\end{proof}

If $\calf$ is a system of essentially small categories indexed by a small category $I$,
then according to~\ref{prop funcorialcore}.(3), $\text{Gr}_I\calf$ is essentially small.
This is probably not true in general if we just assume that $I$ is essentially small. However, we have 
the following lemma, which will be an important tool for us later in this paper.

\begin{lemma}\label{lemma groveresssmallqf}
Let $\calf$ be a system of essentially small categories indexed by
an essentially small category $\cala$\@.  If, for any morphism
$\alpha\colon a_1\to a_0$ in $\cala$, the functor $\calf_{\alpha}\colon \calf_{a_0}\to\calf_{a_1}$ is a weak
equivalence, then\/ $\text{\rm Gr}_{\cala}\calf$ is essentially small.
\end{lemma}
\begin{proof}
Choose a core  $A\subset\cala$\@.  Let $F$ be a core of the  restriction of $\calf$ to $A$
(see~\ref{prop funcorialcore}.(1)).
 We claim  $\text{Gr}_{A}F\subset  \text{Gr}_{\cala}\calf$ is a core. 
Let  $J\subset \text{Gr}_{\cala}\calf$ be a small subcategory containing $\text{Gr}_{A}F$ and  $A'\subset \cala$ be a core  containing the full subcategory on all the objects of the form $\pi(x)$ where $x$ is in $J$ and $\pi\colon\text{Gr}_{\cala}\calf\ra \cala$ is the projection (see~\ref{pt Grothendieck}). 
 Let $F'$ be a  core
of the restriction of $\calf$ to $A'$ such that 
$\text{Gr}_{A}F\subset J\subset \text{Gr}_{A'}F'$ (see~\ref{lem keysystemess}).
According to~\ref{prop ThomasonPuppe}.(3), the homotopy fibers of the nerves of the projections $\pi:\text{Gr}_{A}F\ra A$ and
 $\pi:\text{Gr}_{A'}F'\ra A'$ over a vertex given by an object $a$ in $A$ are weakly equivalent to the nerves of $F_a$ and $F'_a$. As these categories are the cores of $\calf_a$, they are weakly equivalent and consequently the following square is a homotopy pull-back:
 \[\xymatrix@R=11pt@C=15pt{\text{Gr}_{A}F\dto_{\pi}\ar@{^(->}[r] &  \text{Gr}_{A'}F'\dto^{\pi}\\
A\ar@{^(->}[r] & A'
}\]
Since $A\subset A'$ is  a weak equivalence,  $\text{Gr}_{A}F\subset \text{Gr}_{A'}F'$ is a weak equivalence too.
\end{proof}

\begin{cor}\label{cor esssmallhominv}
Let $f:\calb\ra\cala$ be a quasi-fibration. If $\cala$ is essentially small, then $\calb$ is essentially small.
\end{cor}
\begin{proof}
Consider the system 
$-\!\uparrow\! f$ indexed by $\cala^{\text{op}}$ (see~\ref{pt Grothendieck}).
Recall that we have a homotopy equivalence
$\hat{f}\colon \calb\to  \text{Gr}_{\cala}(-\!\uparrow\! f)$. Thus according to~\ref{cor esssmallhomeqinv},
$\calb$ is essentially small if and only if $ \text{Gr}_{\cala}(-\!\uparrow\! f)$ is so. We can now apply~\ref{lemma groveresssmallqf}.
\end{proof}

\begin{prop}\label{prop stronquasi}
Let $f\colon\calb\to\cala$ be  a strong fibration between essentially small categories. Then $f$ is a quasi-fibration.
\end{prop}
\begin{proof}   We claim: {\em for small subcategories $A'\subset \cala$ and $B'\subset\calb$, there is a core $f\colon B\to A$ of $f$ which is a quasi-fibration and  such that $A'\subset A$ and  $B'\subset B$.}

Assume the claim. To prove the proposition we need to show 
$a\!\uparrow\! f$ is essentially small for any  $a$ in $\cala$. 
% Since homotopy equivalences between essentially small categories are weak equivalences (see~\ref{prop esssmallweakequiv})  it is enough to prove 
% $a\uparrow f$ is essentially small for any  $a$ in $\cala$. 
The under category of the restriction of $f$ to 
 $\calc\subset \calb$ is denoted by $a\!\uparrow\!\calc$\@. 
Use the claim to get a  quasi-fibration core
$f\colon B\to A$ of $f\colon\calb\to\cala$  such that $a$ is in $A$.
 We will show that  $a\!\uparrow\! B\subset a\!\uparrow\! f$ is a core. Let $J\subset a\!\uparrow\! f$ be a small subcategory containing $a\!\uparrow\! B$ and  $B'$  be the full  subcategory  of $\calb$ on the set of objects $b$ for which there is $\alpha\colon a\to f(b))$ with  $(b, \alpha)$   in $J$\@.
  Note  $J\subset a\!\uparrow\! B'$\@. Use the claim again to get a  quasi-fibration  core  $f\colon\hat{B}\to\hat{A}$  of 
  $f\colon\cala\to\calb$ 
   such that $A\subset\hat{A}$ and $B'\subset\hat{B}$.  All this fits into a commutative diagram:
\[\xymatrix@R=11pt@C=10pt{
B\ar@{}[r]|{\displaystyle \subset} \ar@/_5pt/[dr]_-{f}& B'\ar@{}[r]|{\displaystyle \subset} & \hat{B}\ar@{}[r]|{\displaystyle \subset} \dto^{f}&\calb\dto^{f}\\
&A\ar@{}[r]|{\displaystyle \subset}  & \hat{A}\ar@{}[r]|{\displaystyle \subset} &
\cala
}\]
The inclusions 
$A\subset  \hat{A}$ and $B\subset   \hat{B}$ are weak equivalences
(see~\ref{lem propesssmall}.(1)).
Since 
$f\colon B\to  A$ and $f\colon\hat{B}\to \hat{A}$ are quasi-fibrations,  
the homotopy fibers of their nerves over the components containing $a$
are given by 
 $N(a\!\uparrow\! B)$ and $N(a\!\uparrow\! \hat{B})$ (see~\ref{prop ThomasonPuppe}.(4)).
These two observations imply that  $a\!\uparrow\! B\subset a\!\uparrow\! \hat{B}$ is  a weak equivalence.

It remains to show the claim. Since  $f\colon\calb\to \cala$ is a strong fibration, for any 
$\alpha\colon a_1\to a_0$ in $\cala$,  there is $\phi_{\alpha}\colon a_1\!\uparrow\! f\to a_0\!\uparrow\! f$
for which   $\phi_{\alpha} (\alpha\!\uparrow\! f)$  and $(\alpha\!\uparrow\! f)\phi_{\alpha} $
are homotopic to the identity functors.  Choose    
 $\{h_{\alpha,k}\colon a_0\!\uparrow\! f\to a_0\!\uparrow\! f\}_{0\leq k\leq m}$
and $\{g_{\alpha,k}\colon a_1\!\uparrow\! f\to a_1\!\uparrow\! f\}_{0\leq k\leq l}$ and  natural transformations
$\phi_{\alpha} (\alpha\!\uparrow\! f)=h_{\alpha,0}\to\cdots\gets h_{\alpha,m}=\text{id}$ and  
$(\alpha\!\uparrow\! f)\phi_{\alpha}=g_{\alpha,0}\to\cdots\gets g_{\alpha,l}=\text{id}$\@.
%connecting the compositions $\phi_{\alpha} (\alpha\uparrow f)$ and $(\alpha\uparrow f)\phi_{\alpha}$ with the
%identity functors.
By induction   define a sequence of subsystems of $f\colon\calb\to \cala$:
\[\xymatrix@R=11pt@C=10pt{
B_0\dto_f\ar@{}[r]|{\displaystyle \subset} & D_1\dto_f\ar@{}[r]|{\displaystyle \subset} &
B_1\dto_f\ar@{}[r]|{\displaystyle \subset} & D_2\dto_f\ar@{}[r]|{\displaystyle \subset} & B_2\dto_f\ar@{}[r]|{\displaystyle \subset}&
\cdots\ar@{}[r]|{\displaystyle \subset}  & \calb\dto^{f}\\
A_0\ar@{}[r]|{\displaystyle \subset} & C_1\ar@{}[r]|{\displaystyle \subset} &
A_1\ar@{}[r]|{\displaystyle \subset} & C_2\ar@{}[r]|{\displaystyle \subset} & A_2\ar@{}[r]|{\displaystyle \subset}&
\cdots\ar@{}[r]|{\displaystyle \subset}  & \cala
}\]
Set $f\colon B_0\to A_0$ to be a core of $f\colon\calb\to\cala$ such that $A'\subset A_0$ and $B'\subset B_0$.
It  exists by~\ref{lem keysystemess}. Assume  $n>0$ and  the sequence is defined for  indices smaller than $n$. Let  $D_n$ to be the full  subcategory of $\calb$ on the set of objects $b$ such that:
\begin{itemize}
\item $b$  either belongs to $B_{n-1}$ or
\item there is an object $b'$ in $B_{n-1}$ and   morphisms $\alpha\colon a_1\to a_0$ in $A_{n-1}$ and  
$\beta'\colon a_1\to f(b')$ and $\beta\colon a_0\to f(b)$ in $\cala$ such that $\phi_{\alpha}(b',\beta')=(b,\beta)$, or
\item  there is   $b'$ in   $B_{n-1}$ and   $\alpha\colon a_1\to a_0$ in $A_{n-1}$ and 
$\beta'\colon a_0\to f(b')$ and $\beta\colon a_0\to f(b)$ in  $\cala$  such that 
$h_{\alpha,k}(b',\beta')=(b,\beta)$ for some $0\leq k\leq m$, or
\item  there is   $b'$ in   $B_{n-1}$  and   $\alpha\colon a_1\to a_0$ in $A_{n-1}$ and $\beta'\colon a_1\to f(b')$ and $\beta\colon a_1\to f(b)$ in  $\cala$  such that 
$g_{\alpha,k}(b',\beta')=(b,\beta)$ for some $0\leq k\leq l$.
\end{itemize}
Define $C_n$ to  be the full subcategory of $\cala$ on the set of objects that belong either to  $A_{n-1}$ or are of the form $f(b)$ where $b$ is  in $D_n$.
The purpose is to ensure that for any  $\alpha\colon a_1\to a_0$ in $A_{n-1}$:
\begin{enumerate}\renewcommand{\labelenumi}{\alph{enumi}.}
\item   $\phi_{\alpha}\colon a_1\!\uparrow\! f\to a_0\!\uparrow\! f$
takes $a_1\!\uparrow\! B_{n-1}$ to $a_0\!\uparrow\! D_n$,
%for any  $\alpha\colon a_1\to a_0$ in $A_{n-1}$;
\item $h_{\alpha,k}\colon a_0\!\uparrow\! f\to a_0\!\uparrow\! f$ takes  $a_0\!\uparrow\! B_{n-1}$
to $a_0\!\uparrow\! D_n$ for any $0\leq k\leq m$,
%$\alpha\colon a_1\to a_0$ in $A_{n-1}$;
\item  $g_{\alpha,k}\colon a_1\!\uparrow\! f\to a_1\!\uparrow\! f$ takes  $a_1\!\uparrow\! B_{n-1}$
to $a_1\!\uparrow\! D_n$ for any $0\leq k\leq l$. % and any  $\alpha\colon a_1\to a_0$ in $A_{n-1}$.
\end{enumerate}
Let $f\colon B_n\to A_n$ to be a core
of $f\colon\calb\to\cala$ such that $C_{n}\subset A_n$ and $D_{n}\subset B_n$.
Define
$A:=\cup_{n\geq 0} A_n$ and $B:=\cup_{n\geq 0} B_n$. 
According to~\ref{lemma-telescope} and~\ref{lem propesssmall}.(1)
$f\colon \calb\to\cala$ is also a core of $f\colon B\to A$.
We are going to  show  that
$f\colon B\to A$ is a strong fibration.  The requirements  above  imply that,  for any   $\alpha\colon a_1\to a_0$ in $A$:
\begin{enumerate}\renewcommand{\labelenumi}{\alph{enumi}.}
\item   $\phi_\alpha\colon a_1\!\uparrow\! f\to a_0\!\uparrow\! f$
takes  $a_1\!\uparrow\! B$ to  $a_0\!\uparrow\! B$;
 \item  $h_{\alpha,k}\colon a_0\!\uparrow\! f\to a_0\!\uparrow\! f$ takes 
$a_0\!\uparrow\! B$ to $a_0\!\uparrow\! B$ for any $0\leq k\leq m$;
\item  $g_{\alpha,k}\colon a_1\!\uparrow\! f\to  a_1\!\uparrow\! f$ takes 
$a_1\!\uparrow\! B$ to $a_1\!\uparrow\! B$ for any $0\leq k\leq l$.
\end{enumerate}
Since $B$  is a full subcategory in $\calb$, for any  $\alpha\colon a_1\to a_0$ in $A$, by restricting to $a_0\!\uparrow\! B$ and $a_1\!\uparrow\! B$ we have two  sequences of natural transformations
$\phi_{\alpha} (\alpha\!\uparrow\! f)=h_{\alpha,0}\to\cdots\gets h_{\alpha,m}=\text{id}$ and  
$(\alpha\!\uparrow\! f)\phi_{\alpha}=g_{\alpha,0}\to\cdots\gets g_{\alpha,l}=\text{id}$,
showing that $\phi_{\alpha} (\alpha\uparrow f):a_0\uparrow B\ra a_0\uparrow B$ and
$(\alpha\uparrow f)\phi_{\alpha}:a_1\uparrow B\ra a_1\uparrow B$ are both homotopic
to the identity functors. The functor $\alpha\!\uparrow\! B\colon a_0\!\uparrow\! B\to a_1\!\uparrow\! B$ is therefore a homotopy equivalence and hence a weak equivalence. Which shows the claim.
\end{proof}

\begin{cor}\label{cor stronhpullhompull}
Let the following be a  strong homotopy pull-back square:  
\[\xymatrix@R=11pt@C=15pt{
\cald\rto^g\dto_e &\calc\dto^{h}\\
\calb\rto^{f} & \cala
}\]
Assume that $\cala$, $\calb$, and $\calc$ are essentially small categories.
Then $\cald$ is also essentially small and the above square is  homotopy
pull-back.
\end{cor}
\begin{proof}
The strong homotopy pull-back assumption implies
 $(e,h)\colon c\!\uparrow\! g\ra h(c)\!\uparrow\! f$ is a homotopy equivalence 
for any  $c$ in $\calc$, and $f\colon\calb\to\cala$ and  $g\colon\cald\to\calc$  are a strong fibrations.
Since $\cala$ and $\calb$ are   essentially small,  by~\ref{prop stronquasi}, $f\colon\calb\to\cala$  is a quasi-fibration and thus $a\!\uparrow\! f$ is essentially small for any  $a$ in $\cala$.
By~\ref{cor esssmallhomeqinv}, $c\!\uparrow\! g$ is then also essentially small for any  $c$ in $\calc$. The functor $g\colon\cald\to\calc$ is therefore a quasi-fibration. As $\calc$ is essentially small, ~\ref{cor esssmallhominv} implies that so is $\cald$.

%It remains to show that the given square is homotopy pull-back.
By  the claim
in the proof of~\ref{prop stronquasi}, there is a core $g\colon D\to C$ of $g$ which is a quasi-fibration. 
%Since $g:\cald\ra\calc$ is a strong fibration of essentially small categories, by  the claim
%in the proof of~\ref{prop stronquasi}, there is a core $g:D\ra C$ of $g:\cald\ra\calc$ which is a quasi-fibration. 
Let  $B'\subset\calb$ be the full subcategory on the set of all objects 
$e(d)$ where $d$ is in $D$ and the full subcategory $A'\subset \cala$ on the set of all
objects $h(c)$ where $c$ is in $C$. The same
 claim yields a core $f\colon B\to A$ of $f$ which is a quasi-fibration and  such that $A'\subset A$ and $B'\subset B$\@. This leads to a commutative diagram of categories:
\[\xymatrix@R=10pt@C=15pt{
& C\ar@{^(->}[rr]\ddto|\hole^(.27){h} & & \calc\ddto^-{h}\\
D\ar@{^(->}[rr] \urto^{g}\ddto_{e}& & \cald \urto_(.42){g}\ddto^(.28){e}\\
& A\ar@{^(->}[rr]|\hole& & \cala\\
B\ar@{^(->}[rr] \urto^-{f}& & \calb \urto_{f}
}\]
In the proof of~\ref{prop stronquasi} it was also shown that $a\!\uparrow\! B\subset a\!\uparrow\! f$ and
 $c\!\uparrow\! D\subset c\!\uparrow\! g$ are cores for any  $a$ in  $A$ and $c$ in $C$. By the ``2 out of 3'' property  $(e,h)\colon c\!\uparrow\! D\to h(c)\!\uparrow\! B$ is  a weak equivalence.  The following square is a therefore a homotopy-pull-back:
%
%The nerves of these under categories are the homotopy fibers of the nerves of $g:D\ra C$ and $f:B\ra A$ and hence the following square is a homotopy-pull-back:
\[\xymatrix@R=11pt@C=15pt{D\rto^{g}\dto_-{e} & C\dto^-{h}\\
B\rto^{f} & A
}\qedhere\]
\end{proof}

\noindent
{\bf Part II. }
%  Bounded functors and spaces of weak equivalences.} 
The aim of this part is to present  a construction of  the spaces of weak equivalences
and their deloopings  in an arbitrary model category  based on~\cite{MR2443229}. 

%%%%%%%%%%%%%%%%%%%%%%%%%%%%%%%%%%%%%%%
%%%%%%%%%%%%%% Simplexcat  %%%%%%%%%%%%%%%%%%%
%%%%%%%%%%%%%%%%%%%%%%%%%%%%%%%%%%%%%%%

\section{Simplex categories}\label{sec simplex}
%Among all the small categories, the ones we are going to discuss the most in this paper are simplex categories.
%Recall (see~\cite[Section 6]{MR2002k:55026}) that 
Let $A$ be a simplicial set. Its simplex category (see~\cite[Section 6]{MR2002k:55026}), denoted by the same symbol
 $A$, is a category whose objects are simplices of $A$ i.e., maps of the 
form $\sigma\colon\Delta[n]\to A$. The set of morphisms in $A$ between $\tau\colon\Delta[m]\to A$ and 
$\sigma\colon\Delta[n]\to A$ consists of the maps $\alpha\colon\Delta[m]\to\Delta[n]$ for which $\tau=\sigma\alpha$.
A map of spaces $f\colon A\to B$ induces a functor $f\colon A\to B$. It assigns to 
$\sigma\colon\Delta[n]\to A$  the composition $f\sigma\colon\Delta[n]\to B$. This defines  
a functor from $\text{Spaces}$ to $\text{Cat}$ called the simplex category.
Its  composition with the nerve is called the {\bf subdivision}:
\[
\xymatrix@C=20pt{
\text{Spaces}\ar[rrr]^-{\text{ simplex category }}\ar@/_12pt/[rrrr]|-{\text{ subdivision }} &&& \text{Cat}\rto^(.4){N} & \text{Spaces}
}\]
%\[\text{Spaces}\xrightarrow{\text{ simplex category }} \text{Cat}\xrightarrow{N}\text{Spaces}\]
The subdivision has the following properties:
\begin{enumerate}
\item  A space $A$ is contractible if and only if $N(A)$ is contractible.
\item A map $f\colon A\to B$  is  a weak equivalence if and only if $N(f)$ is.
\item For any  $f\colon A\to B$,  the  map $N(f)$ is {\bf reduced}, i.e., it maps
non-degenerate simplices to  non-degenerate simplices (see~\cite[Definition 12.9]{MR2002k:55026}).
\item The subdivision  is a left adjoint and hence it  commutes with colimits. In particular
if the   left square below  is a push-out, then so is the right square:
\[\xymatrix@R=11pt@C=15pt{
A\dto_{g}\rto^{f} & B\dto^{h}\\
C\rto^{k} & D
}\ \ \ \ \ \ \ \ \ \ 
\xymatrix@R=11pt@C=20pt{
N(A)\dto_{N(g)}\rto^{N(f)} & N(B)\dto^{N(h)}\\
N(C)\rto^{N(k)} & N(D)
}\]
\end{enumerate}

The symbols $\text{Fun}(-,\calc)$ and $\text{Fun}(N(-),\calc)$   denote systems of categories indexed by
$\text{Spaces}$ (see~\ref{point systemandGr}) given by the  assignments $A\mapsto  \text{Fun}(A,\calc)$,
$(f:A\ra B)\mapsto f^{\ast}$, $A\mapsto \text{Fun}(N(A),\calc)$, and  $(f:A\ra B)\mapsto  N(f)^{\ast}$.
%\[
%A\mapsto \text{Fun}(A,\calc)\ \ \ \
%(f:A\ra B)\mapsto  (f^{\ast}:\text{Fun}(B,\calc)\ra \text{Fun}(A,\calc))\]
%\[A\mapsto \text{Fun}(N(A),\calc)\ \ \ \  
%(f:A\ra B)\mapsto  (N(f)^{\ast}:\text{Fun}(N(B),\calc)\ra \text{Fun}(N(A),\calc)).
%\]

\subsection{Clutching construction}
\label{sec clutching}
\stepcounter{prop}
To construct and analyze functors indexed   simplex categories one can use  the geometry of the underlying spaces.
For example  the  clutching construction 
can be described  as follows.
An {\bf initial data} consists of a push-out square of spaces where the indicated 
maps are inclusions:
\[\xymatrix@R=11pt@C=18pt{
A\rmono^{i} \dto_{f} & C\dto^{g}\\
B\rmono^{j} & D
}\]
 two functors $F\colon B\to \calm$ and $G\colon C\to \calc$,
and   a natural transformation $\psi\colon f^{\ast}F\to i^{\ast}G$ in $\text{Fun}(A,\calc)$.
Out of this data we are going to construct a functor $H\colon D\to \calc$ and a natural transformation
$\overline{\psi}\colon g^{\ast}H\to G$ in $\text{Fun}(C,\calc)$\@.   This functor is called {\bf the clutching of $F$ and $G$ along $\psi$} and is  denoted by  $H(\psi, F,G)$\@. The  functor $H$ and the natural transformation  $\overline{\psi}$ are supposed to satisfy the following properties:
\begin{enumerate}
\item \label{pt restrictionalongj} $j^{\ast}H=F$;
\item 
\label{eq cluthingreq}
the following   diagram commutes:
\[
\xymatrix@R=11pt@C=6pt{f^{\ast}F\ar@{=}[d]\ar@{->}[rrr]|{\psi} &  & & i^{\ast}G \\
f^{\ast}j^{\ast}H\ar@{=}[r] &
(jf)^{\ast}H\ar@{=}[r] &(gi)^{\ast}H\ar@{=}[r] &i^{\ast}g^{\ast} H\uto_(0.5){i^{\ast}\overline{\psi}}
}
\]
\item 
\label{pt isoonG} the morphism $\overline{\psi}_{\sigma}\colon Hg(\sigma)\to G(\sigma)$ is an isomorphism
for any simplex $\sigma$ in $C$ which is not in the image of $i$.
\end{enumerate}
We  use the following diagrams to depict an initial  data and its clutching:
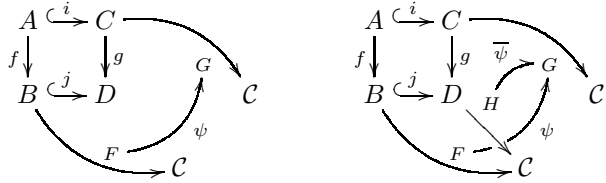
\[\xymatrix@R=15pt@C=15pt{
A\rmono^{i} \dto_{f} & C\dto^{g}\ar@/^1pc/[rrd]_(.7){G}="G" \\
B\rmono^{j}\ar@/_1pc/[rrd]^(.6){F}="F" & D & & \calc \ar@/_.9pc/ @{->}_(.6){\psi} "F"; "G" \\
&  & \calc
}
\ \ \ \ \ \ \ \ \ 
\xymatrix@R=15pt@C=15pt{
A\rmono^{i} \dto_{f} & C\dto^{g}\ar@/^1pc/[rrd]_(.7){G}="G" \\
B\rmono^{j}\ar@/_1pc/[rrd]^(.6){F}="F" & D\drto^(.3){H}="H" & &\calc \ar@/_.9pc/ @{->}_(.6){\psi}|(.28)\hole "F"; "G" 
\ar @/^.5pc/@{->}^{\overline{\psi}} "H" ; "G"\\\
&  & \calc
}
\]

By the push-out assumption in the initial data,  there is a bijective correspondence between the set  $D_n$ and the disjoint union $j(B_n)\coprod g(C_n\setminus i(A_n))$. Furthermore,
$j\colon B_n\to j(B_n)$, $g\colon C_n\setminus i(A_n)\to g(C_n\setminus i(A_n))$, and $i\colon A_n\to i(A_n)$ are bijections.
This justifies the   use of the following notation. If   $\sigma\colon \Delta[n]\to C$ belongs to  $i(A)$, then ${\sigma'}\colon\Delta[n]\to A$  denotes the unique simplex for which $i {\sigma'} =\sigma$\@.  If $\sigma\colon\Delta[n]\to D$ belongs to $j(B)$, then  ${\sigma'}\colon\Delta[n]\to B$  denotes the unique simplex for which $j {\sigma'} =\sigma$\@.  If $\sigma\colon\Delta[n]\to D$ does not belong to $j(B)$,  then $\overline{\sigma}\colon\Delta[n]\to C$  denotes the unique simplex in $C$ for which $g\overline{\sigma}=\sigma$\@. 
We can now define:
\[H(\sigma):=
\begin{cases}
F({\sigma'}) &\text{ if } \sigma\in j(B)\\
G(\overline{\sigma}) &\text{ if } \sigma\not\in j(B)
\end{cases}
\]
Let $\alpha\colon\Delta[m]\to \Delta[n]$ be a morphism in $D$ between
$\tau\colon\Delta[m]\to D$ and $\sigma\colon\Delta[n]\to D$. 
We are going to define $H(\alpha)\colon H(\tau)\ra H(\sigma)$.
Note that if $\sigma$ belongs to $j(B)$, then so does $\tau$.  Thus there are no morphisms
in $D$ between any simplex that does not belong to $j(B)$ and a simplex that belongs to $j(B)$.
Three possibilities remain\smallskip.
\begin{tabular}{c|c}
\hline \\  [-1.5ex]
\hspace{0mm}$\bullet$ If $\tau\in j(B)$ and $\sigma\in j(B)$, then \hspace{6mm} &
\hspace{0mm} $\bullet$ If $\tau\not\in j(B)$ and $\sigma\not \in j(B)$, then
\hspace{6mm} \\
 $\xymatrix@R=6pt@C=20pt{
H(\tau)\ar@{=}[d]\rto^{H(\alpha)}& H(\sigma)\\
F(\tau')\rto^{F(\alpha)} & F(\sigma')\ar@{=}[u] 
}$ 
  & 
 $\xymatrix@R=6pt@C=20pt{H(\tau)\ar@{=}[d]\rto^{H(\alpha)}&H(\sigma) \\
G(\overline{\tau})\rto^(.47){G(\alpha)} & G(\overline{\sigma})
\ar@{=}[u] 
}$ \\
 \hline  \end{tabular}\smallskip
 
 \noindent
$\bullet$ If $\tau\in j(B) $ and $\sigma\not\in j(B)$, then 
we have a commutative diagram of spaces:
\[\xymatrix@R=15pt@C=18pt{
& \Delta[m]\dlto_{(\overline{\sigma}\alpha)'}\rrto^{\alpha}\ar@{=}[dd]|\hole\drto|{\overline{\sigma}\alpha} & & \Delta[n] \dlto_{\overline{\sigma}}  \ar@{=}[dd]\\
A\ddto_{f} \ar@{^(->}[rr]^(.3){i} & & C\ddto^(.3){g}\\
& \Delta[m]\dlto_{\tau'}\rrto^(.3){\alpha}|(.51)\hole \drto|{\tau} & & \Delta[n]\dlto_{\sigma}\\
B\ar@{^(->}[rr]^{j}  & & D
}\]
Define $H(\alpha)\colon H(\tau)\to H(\sigma)$ as the following composition:
\[\xymatrix@R=6pt@C=20pt{
H(\tau)\ar@{=}[d]\ar@{->}[rrrr]|{H(\alpha)} & & & &  H(\sigma)\\
 F(\tau')\ar@{=}[r] & F(f(\overline{\sigma}\alpha)')\rto^{\psi_{(\overline{\sigma}\alpha)'}} &
G(i(\overline{\sigma}\alpha)')\ar@{=}[r] &G(\overline{\sigma}\alpha)\rto^{G(\alpha)} & G(\overline{\sigma})\ar@{=}[u] 
}\]

\noindent\hrulefill
\smallskip

%\begin{itemize}
%\item  both $\tau$ and $\sigma$  belong to $j(B)$. Define $H(\alpha)\colon H(\tau)\to H(\sigma)$ as:
%\[\xymatrix@R=6pt@C=20pt{
%H(\tau)\ar@{=}[d]\rto^{H(\alpha)}& H(\sigma)\\
%F(\tau')\rto^{F(\alpha)} & F(\sigma')\ar@{=}[u] 
%}\]
%\item neither $\tau$ nor $\sigma$   belong to $j(B)$. Define $H(\alpha)\colon H(\tau)\to H(\sigma)$ as:
%\[\xymatrix@R=6pt@C=20pt{H(\tau)\ar@{=}[d]\rto^{H(\alpha)}&H(\sigma) \\
%G(\overline{\tau})\rto^(.47){G(\alpha)} & G(\overline{\sigma})
%\ar@{=}[u] 
%}\]
%\begin{itemize}
%\item  $\tau$ belongs to $j(B)$ and $\sigma$ does not belong to $j(B)$. In this case we have a commutative diagram of spaces:
%\[\xymatrix@R=15pt@C=18pt{
%& \Delta[m]\dlto_{(\overline{\sigma}\alpha)'}\rrto^{\alpha}\ar@{=}[dd]|\hole\drto|{\overline{\sigma}\alpha} & & \Delta[n] \dlto_{\overline{\sigma}}  \ar@{=}[dd]\\
%A\ddto_{f} \ar@{^(->}[rr]^(.3){i} & & C\ddto^(.3){g}\\
%& \Delta[m]\dlto_{\tau'}\rrto^(.3){\alpha}|(.51)\hole \drto|{\tau} & & \Delta[n]\dlto_{\sigma}\\
%B\ar@{^(->}[rr]^{j}  & & D
%}\]
%Define $H(\alpha)\colon H(\tau)\to H(\sigma)$ as the following composition:
%\[\xymatrix@R=6pt@C=20pt{
%H(\tau)\ar@{=}[d]\ar@{->}[rrrr]|{H(\alpha)} & & & &  H(\sigma)\\
 %F(\tau')\ar@{=}[r] & F(f(\overline{\sigma}\alpha)')\rto^{\psi_{(\overline{\sigma}\alpha)'}} &
%G(i(\overline{\sigma}\alpha)')\ar@{=}[r] &G(\overline{\sigma}\alpha)\rto^{G(\alpha)} & G(\overline{\sigma})\ar@{=}[u] 
%}\]
%\end{itemize}
This procedure indeed defines a functor $H\colon D\to \calm$ such that
$Hj=F$, which is the  requirement~(\ref{pt restrictionalongj}). It remains to construct  $\overline{\psi}\colon Hg\to G$.
\begin{itemize}
\item If $\sigma\colon\Delta[n]\to C$ belongs to $i(A)$, then $\overline{\psi}_{\sigma}\colon Hg(\sigma)\to G(\sigma)$
is defined as:
\[\xymatrix@R=6pt@C=15pt{
Hg(\sigma)\ar@{=}[d]\ar@{->}[rrr]|{\overline{\psi}_{\sigma}} & & &G(\sigma) \\ 
Hgi(\sigma')\ar@{=}[r] & Hjf(\sigma')\ar@{=}[r] 
&Ff(\sigma')\rto^{\psi_{\sigma'}} & Gi(\sigma')\ar@{=}[u] 
}\]
\item If $\sigma\colon\Delta[n]\to C$ does not belong to $i(A)$, then $\overline{\psi}_{\sigma}$
is given by the identity  $\text{id}:Hg(\sigma)=G(\sigma)\ra G(\sigma)$.
\end{itemize}
The morphisms $\{\overline{\psi}_{\sigma}\}_{\sigma\in C}$ form a natural transformation between $g^{\ast}H$ and  $G$\@ that fulfills the requirements~(\ref{eq cluthingreq})  and~(\ref{pt isoonG}).
%Since for any  simplex $\sigma$ in $C$  which is not in the image of $i$, the morphism $\overline{\psi}_{\sigma}$ is the identity, the requirement~(\ref{pt isoonG})  is satisfied too.

Assume now that we have
 a push-out square of spaces, where the indicated 
maps are inclusions:
\[\xymatrix@R=11pt@C=20pt{
A\rmono^{i} \dto_{f} & C\dto^{g}\\
B\rmono^{j} & D
}\]
two functors $F\colon B\to \calc$ and $G\colon D\to\calc$, and
 a natural transformation $\psi\colon F\to j^{\ast} G$ in $\text{\rm Fun}(B,\calc)$.
This data induces an initial data for the clutching that consists of the above push-out,
 functors $F\colon A\to\calc$ and  $g^{\ast}G\colon C\to \calc$, and a natural transformation
$f^{\ast}\psi\colon f^{\ast}F\to f^{\ast}j^{\ast} G=i^{\ast}g^{\ast}G$\@. Its clutching  is a
functor $H(f^{\ast}\psi, F, g^{\ast}G)\colon D\to \calc$ and a natural transformation $\overline{f^{\ast}\psi}\colon
g^{\ast}H(f^{\ast}\psi, F, g^{\ast}G)\to g^{\ast} G$.  

\begin{prop}\label{prop univclutching}
There is a unique natural transformation 
$\widehat{\psi}\colon H(f^{\ast}\psi, F, g^{\ast}G)\to G$
 such that $j^{\ast}\widehat{\psi}=\psi$
and $g^{\ast}\widehat{\psi}=\overline{f^{\ast}\psi}$.
\end{prop}
\begin{proof}
Let $\sigma$ be a simplex in $D$\@. If it belongs to $B$, define
$\widehat{\psi}_{\sigma}$
to be $\psi_{\sigma}$. If it does not,  define
$\widehat{\psi}_{\sigma}$ to be
the identity morphism $\text{id}\colon H(f^{\ast}\psi, F, g^{\ast}G)(\sigma)=g^{\ast}G(\overline{\sigma})=
G(\sigma)$. These morphisms define the desired natural transformation $\widehat{\psi}$.
\end{proof}

\section{Bounded functors}\label{sec boundedfunctors}
%In this section we recall how to introduce  homotopy  notions on functors indexed by simplex categories.
In a simplex category $A$ the face and degeneracy morphisms:
%
%In Section~\ref{sec clutching} we illustrated how  the geometry of  simplex categories  can be used to 
%construct functors. In this section we recall how to use it to introduce  homotopy  notions. 
%The aim of this section is to discuss the so called bounded functors. 
%We refer the reader to~\cite{MR2002k:55026} for more detailed information about such functors.
%Let $A$ be a simplicial set.
%The face and degeneracy morphisms in a simplex category $A$:
%Subject to the cosimplicial identities the morphisms in the simplex category $A$ 
%are generated by the face and degeneracy morphisms:
\[\xymatrix@R=8pt@C=16pt{\Delta[n]\rrto^-{d_{i}} \drto_{d_i\sigma} & & \Delta[n+1]\dlto^{\sigma}\\
& A
}\ \ \ \ \
\xymatrix@R=10pt@C=15pt{\Delta[n+1]\rrto^-{s_{i}} \drto_{s_i\sigma} & & \Delta[n]\dlto^{\sigma}\\
& A
}
\]
are subject to the usual cosimplicial identities, and they generate all the morphisms in $A$\@.
A functor  indexed by $A$ is called {\bf  bounded} \cite[Definition 10.1]{MR2002k:55026} if
it assigns an isomorphism to  any degeneracy morphism 
$s_{i}$ in $A$. If $S$ denotes the set of
all these degeneracy morphisms in $A$, then a bounded functor is
a functor indexed by the localized category
$A[S^{-1}]$\@. 
%An important observation is that any bounded functor is naturally
%isomorphic to a functor that assigns identities to morphisms in
%$S$ \cite[Proposition 10.3]{MR2002k:55026}. 
%This can be therefore
%assumed about any considered bounded functor. 
The symbol $\text{Fun}^{b}(A,\mathcal{C})$  denotes the category of 
bounded functors indexed by $A$ with values in  $\calc$\@.  The sets of morphisms in $\text{Fun}^{b}(A,\mathcal{C})$
consists  of all  natural transformations. 
For example,  let $I$ be a small category and  $\epsilon\colon N(I)\to I$  be a functor   defined as follows (see
also~\cite[Definition 6.6]{MR2002k:55026}).
For a simplex $\sigma=(i_{n}\stackrel{\alpha_{n}}{\ra}
\cdots \stackrel{\alpha_{1}}{\ra} i_{0})$ in $N(I)$, let:
\[\epsilon(\sigma):=i_0,\ \ \   \epsilon(s_k\colon s_k\sigma\to\sigma):=\text{id}_{i_0},\ \ \   \epsilon(d_k\colon d_k\sigma\to\sigma):=
\begin{cases}
\text{id}_{i_{0}} & \text{ if }\  k>0\\
\alpha_{1} & \text{ if }\  k=0
\end{cases}
\]
Since $\epsilon$ maps all the degeneracy morphisms to  identities, it is  a bounded functor.
Consequently
$\epsilon^{\ast}\colon\text{Fun}(I,\calc)\to \text{Fun}(N(I),\calc)$ has bounded values. The induced functor is denoted by the same symbol  $\epsilon^{\ast}\colon \text{Fun}(I,\calc)\to \text{Fun}^{b}(N(I),\calc)$.

If $F\colon B\to \mathcal{C}$ is bounded, then, for any map $f\colon A\to B$,  so is  $Ff\colon A\to  \mathcal{C}$.
Thus  the  subcategories $ \text{Fun}^{b}(A,\mathcal{C})\subset  \text{Fun}(A,\mathcal{C})$   form  a subsystem
of  $\text{Fun}(-,\mathcal{C})$. The same is true for the subcategories
 $ \text{Fun}^{b}(N(A),\mathcal{C})\subset  \text{Fun}(N(A),\mathcal{C})$.
 We denote these subsystems by $ \text{Fun}^{b}(-,\mathcal{C})$ and  $ \text{Fun}^{b}(N(-),\mathcal{C})$
 respectively.

If $\calc$ is 
 closed under colimits, then  %(for example if it is a model category),
the left Kan extension
 $f^k\colon \text{Fun}(A,\calc)\to \text{Fun}(B,\calc)$ also preserves the property of being bounded
(\cite[Theorem 10.6]{MR2002k:55026}).  
 Thus  in this case we have  a pair of adjoint functors
  $f^{k}\colon\text{Fun}^{b}(A,\calc)\rightleftarrows \text{Fun}^{b}(B,\calc): f^{\ast}$.

%Similarly to left Kan extensions the clutching  also   preserves boundedness: 
\begin{prop}\label{prop boundclutch}
Notation as in Section~\ref{sec clutching}. If  $F\colon B\to \calc$ and $G\colon C\to \calc$
are bounded, then so is  their clutching $H(\psi,F,G)$.
\end{prop}
\begin{proof}
This is a consequence of properties (1) and (3) given in Section~\ref{sec clutching}.
\end{proof}

Let $\calm$ be a model category and    $A$  a simplicial set. A natural transformation $\phi\colon F\to G$
  in $\text{\rm Fun}^{b}(A,\mathcal{M})$\@ is called a {\bf weak equivalence} ({\bf fibration}) if
 $\phi_{\sigma}\colon F(\sigma)\to
G(\sigma)$ is a weak equivalence (fibration) in $\calm$ for any  $\sigma$ in $A$\@. 
It is called  a {\bf cofibration} if, for any  {\bf non-degenerate}
 $\sigma\colon\Delta[n]\to A$, the morphism:
 \[\text{colim}( \text{colim}_{\partial\Delta[n]} G\sigma\xleftarrow{\text{colim}_{\partial\Delta[n]}\phi}
\text{colim}_{\partial\Delta[n]} F\sigma\ra  F(\sigma))\lra G(\sigma)\]
induced commutativity of the following  square is a cofibration:
\[\xymatrix@R=13pt@C=15pt{
\text{colim}_{\partial\Delta[n]} F\sigma \dto_{\text{colim}_{\partial\Delta[n]}\phi} \rto& \text{colim}_{\Delta[n]} F\sigma 
\dto|{\text{colim}_{\Delta[n]}\phi}
\ar@{=}[r]  &F(\sigma) \dto^{\phi_{\sigma}} \\
\text{colim}_{\partial\Delta[n]} G\sigma\rto & \text{colim}_{\Delta[n]} G\sigma  \ar@{=}[r]  & G(\sigma)
}\]
%
%$\sigma\colon\Delta[n]\to A$  its simplex. For any 
%  $F\colon A\to \calm$,   $\text{colim}_{\Delta[n]}F\sigma$ can be identified  with $F(\sigma)$ as $\text{id}\colon\Delta[n]\to\Delta[n]$ is the terminal object in $\Delta[n]$\@. 
%The morphism $\text{colim}_{\partial\Delta[n]}F\sigma\ra \text{colim}_{\Delta[n]}F\sigma=F(\sigma)$
%is called the  {\bf $\sigma$-cell} of $F$\@.  More generally let $\phi\colon F\to G$ be a natural transformation in $\text{Fun}^{b}(A,\mathcal{M})$\@. 
%Consider a commutative diagram between   $\sigma$-cells:
%\[\xymatrix@R=13pt@C=15pt{
%\text{colim}_{\partial\Delta[n]} F\sigma \dto_{\text{colim}_{\partial\Delta[n]}\phi} \rto& \text{colim}_{\Delta[n]} F\sigma 
%\dto|{\text{colim}_{\Delta[n]}\phi}
%\ar@{=}[r]  &F(\sigma) \dto^{\phi_{\sigma}} \\
%\text{colim}_{\partial\Delta[n]} G\sigma\rto & \text{colim}_{\Delta[n]} G\sigma  \ar@{=}[r]  & G(\sigma)
%}\]
%The induced morphism is called the {\bf $\sigma$-cell} of $\phi$:
%The  morphism induced by the commutativity of the above diagram is called the {\bf $\sigma$-cell} of $\phi$:
%\[\text{colim}( \text{colim}_{\partial\Delta[n]} G\sigma\xleftarrow{\text{colim}_{\partial\Delta[n]}\phi}
%\text{colim}_{\partial\Delta[n]} F\sigma\ra  F(\sigma))\lra G(\sigma)\]

%%%%%%%%%%%%%%%%%%%%%%%%%
\begin{thm}[{\cite[Theorem 21.1]{MR2002k:55026}}]\label{thm modelmain}
%%%%%%%%%%%%%%%%%%%%%%%%%
The above choice of weak equivalences, fibrations, and cofibrations in
$\text{\rm Fun}^{b}(A,\mathcal{M})$ is a  model category structure.
%
%The following describes a model structure on $\text{\rm Fun}^{b}(A,\mathcal{M})$:
%\begin{itemize}
%\item $\phi\colon F\to G$ is a weak equivalence (fibration) if,
%for any simplex $\sigma$ in $A$, $\phi_{\sigma}\colon F(\sigma)\to
%G(\sigma)$ is a weak equivalence (fibration) in $\calm$;
%\item $\phi\colon F\to G$ is a (acyclic) cofibration if, for any non-degenerate
%simplex $\sigma\colon\Delta[n]\to A$, the $\sigma$-cell of $\phi$ is a (acyclic) cofibration in $\mathcal{M}$.
%
%\[\text{\rm colim}( \text{\rm colim}_{\partial\Delta[n]} G\sigma\xleftarrow{\text{\rm colim}_{\partial\Delta[n]}\phi}
%\text{\rm colim}_{\partial\Delta[n]} F\sigma\ra  F(\sigma))\lra G(\sigma)\]
%is a (acyclic) cofibration in $\mathcal{M}$.
%\end{itemize}
\end{thm}

%% According to the above theorem, a bounded functor $F:A\ra \calm$  is cofibrant if, for any non-generate simplex
%% $\sigma:\Delta[n]\ra A$, the $\sigma$-cell $\text{colim}_{\partial\Delta[n]}F\ra F(\sigma)$ is a
%% cofibration. A bounded functor  $F:A\ra \calm$  is fibrant if its values are fibrant.

Let $X$ be an object in $\calm$\@. A bounded functor $F\colon A\to X_{\text{we}}$ is called
cofibrant if its composition with  $X_{\text{we}}\subset \calm$ is cofibrant in
$\text{Fun}^b(A,\calm)$\@. The symbol  $\text{Cof}(A,X_{\text{we}})$ denotes the full subcategory of
$\text{\rm Fun}^{b}(A,X_{\text{we}})$ of  cofibrant functors. 

%How functorial is this  model structure on $\text{\rm Fun}^{b}(A,\calm)$? 
A map of simplicial sets $f\colon A\to B$ can send a non-degenerate simplex in $A$  to a degenerate simplex in $B$\@. This is the reason why %$Ff\colon A\to \calm$  can fail to be cofibrant  for a bounded and cofibrant   $F\colon B\to \calm$. Consequently, 
the subcategories $\text{Cof}(A,X_{\text{we}})\subset \text{Fun}^b(A,X_{\text{we}})$ do not form a subsystem of $\text{Fun}^b(-,X_{\text{we}})$.

%%%%%%%%%%%%%%%%%%%%%%%%%%%%%
\begin{prop}\label{prop basicprophocolim}
%%%%%%%%%%%%%%%%%%%%%%%%%%%%%
Let $f:A\ra B$ be a map of spaces.
\begin{enumerate}
\item If  $\phi$ is a weak equivalence in $\text{\rm Fun}^{b}(A,\mathcal{M})$ between cofibrant objects, 
then  $f^{k}\phi$ is a weak equivalence in
$\text{\rm Fun}^{b}(B,\mathcal{M})$.
\item If $\phi$ is an (acyclic) cofibration  in $\text{\rm Fun}^{b}(A,\mathcal{M})$, then  $f^{k}\phi$ is an (acyclic) cofibration  in $\text{\rm Fun}^{b}(B,\mathcal{M})$.
%In particular $\text{\rm colim}_{A}\phi$ is a (acyclic) cofibration in $\mathcal{M}$. 
\item Assume  $f$ is reduced  (see Section~\ref{sec simplex}). 
If $\phi$ is an (acyclic) cofibration in $\text{\rm Fun}^{b}(B,\calm)$, then so is
$f^{\ast}\phi$ in $\text{\rm Fun}^{b}(A,\calm)$.
\item If $\phi$ is
an (acyclic) cofibration in $\text{\rm Fun}^{b}(N(B),\mathcal{M})$, then so
is the natural transformation  $N(f)^{\ast}\phi$ in $\text{\rm Fun}^{b}(N(A),\mathcal{M})$.
\end{enumerate}
\end{prop}
\begin{proof}
Statement (1) is~\cite[Proposition 13.3(2)]{MR2002k:55026}. Statement (2) is
\cite[Theorem 11.2.]{MR2002k:55026}. Statement (3) follows  from the definition and (4)
is its particular case.
\end{proof}

Since, for  any map $f\colon A\to B$ of spaces,  $N(f)$ is reduced,  the subcategories $\text{Cof}(N(A),X_{\text{we}})\subset \text{\rm Fun}^{b}(N(A),X_{\text{we}})$ form a subsystem of $\text{\rm Fun}^{b}(N(-),X_{\text{we}})$.

%%%%%%%%%%%%%%%%%%%%%%%%%%%%%%%%%%%%%%%
%%%%%%%%%%%%%% Mappingspaces  %%%%%%%%%%%%%%%%%%%
%%%%%%%%%%%%%%%%%%%%%%%%%%%%%%%%%%%%%%%

\section{homotopy colimits and derived left Kan extensions}
\label{sec hocollkex}
%In this section we recall  how to %use  the functor $\epsilon:N(I)\ra I$ to
%construct homotopy colimits and derived left kan extensions.
The following proposition is part of~\cite[Theorem 11.3(1)]{MR2002k:55026} which states that %the functors 
$\epsilon^{k}\colon\text{Fun}^b(N(I),\calm)\rightleftarrows \text{Fun}(I;\calm):\epsilon^{\ast}$
is  a left model approximation (\cite[5.1]{MR2002k:55026}):

\begin{prop}\label{prop modelapprox}
Let $I$ be a small category.
\begin{enumerate}
\item Assume that $\phi\colon F\to G$ is a weak equivalence between cofibrant objects in\/
$\text{\rm Fun}^{b}(N(I),\calm)$. Then  $\epsilon^{k}\phi$ is a weak equivalence in\/ $\text{\rm Fun}(I,\calm)$.
\item  Let $F\colon I\to \calm$ be a functor  and $\psi\colon G\to \epsilon^{\ast}F$ be a weak equivalence 
in\/ $\text{\rm Fun}^{b}(N(I),\calm)$\@. If $G\colon N(I)\to\calm$ is cofibrant, 
then  the morphism $\epsilon^{k}G\to F$ which is adjoint  to $\psi$ is a weak equivalence.
\end{enumerate}
\end{prop}

Let  $P$  be  a functorial cofibrant replacement in   $\text{Fun}^b(N(I),\calm)$
(see~\ref{pt modelcat}).
We  use the same symbol $P$ to denote the following composition:
%induced by this chosen cofibrant replacement and the functor $\epsilon\colon N(I)\to I$:
\[\xymatrix@C=15pt{
\text{Fun}(I,\calm)\rto^-{\epsilon^{\ast}}\ar@/_1.2pc/[rrr]|-P&\text{Fun}^b(N(I),\calm)\rto^-{P}& \text{Fun}^b(N(I),\calm)
\rto^-{\epsilon^k} & \text{Fun}(I,\calm)
}\]
%
%\[\xymatrix@R=11pt@C=15pt{
%\text{Fun}(I,\calm)\dto_(.45){\epsilon^{\ast}}\rto^{P}  &\text{Fun}(I,\calm)
%\\  \text{Fun}^b(N(I),\calm)\rto^{P} 
%& \text{Fun}^b(N(I),\calm)
%\uto_(.55){\epsilon^{k}} 
%}\]
Let $PF=\epsilon^{k}P\epsilon^{\ast}F\to F$ be the natural transformation  adjoint to the cofibrant replacement $P\epsilon^{\ast}F\to \epsilon^{\ast}F$\@.  According to~\cite[Theorem 11.3(1)]{MR2002k:55026} it is a weak equivalence. The functor 
$P\colon\text{Fun}(I,\calm)\to \text{Fun}(I,\calm)$ with this natural  transformation
is called a {\bf cofibrant replacement} in $\text{Fun}(I,\calm)$ and   $\text{colim}_IP(-)\colon\text{Fun}(I,\calm)\to \calm$  is called the {\bf  homotopy colimit}.
There is a natural transformation $\text{colim}_IP(-)\to \text{colim}_I (-)$ given by the colimit
of the cofibrant replacement. 
According to~\cite[Theorem 11.3(2)]{MR2002k:55026},  this natural transformation is   the total  left derived functor of the colimit which justifies
the name homotopy colimit.  
%If $f\colon F\to G$ is a weak equivalence in $\text{Fun}(I,\calm)$,
%then so is $\text{colim}_IPf$ in $\calm$.

Analogous statements hold for arbitrary left Kan extensions. Let $f\colon I\to J$ be a functor
between small categories.
The functor $f^{k}PF\colon J\to \calm$ is called the {\bf derived  left Kan extension} of $F\colon I\to \calm$ along $f$\@.
Let 
$f^{k}PF\to f^kF$ be the natural transformation given by the left Kan extension of the 
  cofibrant replacement.  In this way we obtain a functor
 $ f^{k}P\colon\text{Fun}(I,\calm)\to \text{Fun}(J,\calm)$ and a natural transformation $f^kP\to f^k$\@.
According to~\cite[Theorem 11.3(3)]{MR2002k:55026}, this natural transformation is   the total  left derived functor of the left Kan extension.

Let $j$ be an object in $J$ and  $g\colon f\!\downarrow\! j\to I$ be the forgetful functor. The value of the derived left Kan extension $f^kPF(j)=\text{colim}_{f\downarrow j}(PF)g$
is weakly equivalent to the homotopy colimit  of the composition $Fg\colon f\!\downarrow\! j\to \calm$.
%$\text{hocolim}_{f\downarrow j}Fg$.

\section{Homotopy constant functors and their homotopy colimits}
\label{sec hdkepxwe}
A functor $F\colon I\to \calm$, indexed by a small category $I$, is called {\bf homotopy
constant} if it is isomorphic in $\text{Ho}(\text{Fun}(I,\calm))$ to a constant functor. 
To be homotopy constant it is  necessary for a functor to have  
values in a component $X_{\text{we}}$ for some  $X$ in $\calm$.
In general however this is not enough. Similarly to
the fact that fibrations of spaces over a contractible base are weakly equivalent to 
product fibrations, in the context of functors  according to~\cite[Corollary 29.2]{MR2002k:55026}, we have:

\begin{prop}\label{prop hconst}
Let $I$ be a small  contractible category and  $P$ a cofibrant replacement in\/ $\text{\rm Fun}(I,\calm)$. 
 Then, for any     $F\colon I\to  X_{\text{\rm we}}$,
the morphism
$PF(i)\ra \text{\rm colim}_IPF$, induced by the inclusion of an object  $i$ in $I$,   is
a weak equivalence and 
 $F$ is homotopy constant.
\end{prop}

%According to~\ref{prop hconst}.(1) 
If  $I$ is contractible, then
the homotopy colimit  
$\text{colim}_IP(-)\colon \text{Fun}(I,\calm)\to \calm$ maps the subcategory 
$\text{Fun}(I,X_{\text{we}})\subset \text{Fun}(I,\calm)$ into 
$X_{\text{we}}\subset \calm$ (see~\ref{prop hconst}) inducing a  functor $\text{colim}_IP(-)\colon
\text{Fun}(I,X_{\text{we}})\to X_{\text{we}}$\@. More generally
let $f\colon I\to J$ be a functor of small categories  such that $f\!\downarrow\! j$ is contractible for any $j$ in $J$\@.
Then the derived left Kan extension $f^kP(-)\colon\text{Fun}(I,\calm)\to \text{Fun}(J,\calm)$ 
maps  $\text{\rm Fun}(I,X_{\text{\rm we}})$  %\subset \text{\rm Fun}(I,\calm)$ 
into
$\text{\rm Fun}(J,X_{\text{\rm we}})$ inducing  a functor
$f^kP(-)\colon\text{\rm Fun}(I,X_{\text{\rm we}})\to \text{\rm Fun}(J,X_{\text{\rm we}})$ (see~\ref{prop hconst}).
In particular let $f\colon A\to B$ be a map of simplicial sets. For any  $\sigma\colon\Delta[n]\to B$
the category $f\!\downarrow\! \sigma$ is the simplex category of the pull-back
$df(\sigma)=\text{lim}(\xymatrix@R=2pt@C=16pt{\Delta[n]\rto|(.52){\sigma} & B&A\lto|(.45){f}})$\@.
%\[\text{lim}(\xymatrix@C=15pt{\Delta[n]\rto^{\sigma} & B&A\lto_{f}   })\]
Thus if $df(\sigma)$ is contractible for any $\sigma$, then the derived left Kan extension
induces a functor 
$f^kP\colon \text{Fun}(A,X_{\text{we}})\to \text{Fun}(B,X_{\text{we}})$.

%Let $F:A\ra \calm$ be bounded.
%Let $A$Êbe a simplicial set.
%Assume now that $A$  is a simplicial set and $F:A\ra \calm$ is a functor.
If a bounded functor $F\colon A\to \calm$ 
is weakly equivalent to a constant functor in $\text{Fun}(A,\calm)$, it 
does not necessarily have to be weakly equivalent to a constant functor in 
$\text{Fun}^b(A,\calm)$.
Furthermore even if $A$ is  contractible
and $F\colon A\to X_{\text{we}}$ is  bounded and  cofibrant, 
then  the morphism $F(\sigma)\to\text{colim}_AF$, induced by
the inclusion of a simplex $\sigma$ into $A$, may not be a weak equivalence. This is not
 true even if we assume  that $F$ is weakly equivalent to a constant functor in 
$\text{Fun}^b(A,\calm)$\@. An additional constraint on $A$ is needed. The following is
a consequence of~\ref{prop hconst}:

 \begin{prop}\label{prop hocolimovercontract}
 Let $F\colon N(I)\to \calm$ be  a   bounded and cofibrant
 functor weakly equivalent in\/ $\text{\rm Fun}^{b}(N(I),\calm)$ to a constant functor and $\sigma$ a simplex in $N(I)$. 
If  $I$ is contractible, then  the morphism $F(\sigma)\to\text{\rm colim}_{N(I)}F$ is  a weak equivalence.
\end{prop}
%\begin{proof}
%Proposition   follows  from~\ref{prop hconst}.
% as $\text{\rm colim}_{N(I)}F$  and $\text{colim}_{I} P \epsilon^{k}F$
%are weakly equivalent.
%
%Since $\text{\rm colim}_{N(I)}F$ and $\text{colim}_{I} \epsilon^{k}F$ are isomorphic,
 %By~\ref{prop basicprophocolim}.(1), $\text{\rm colim}_{N(I)}F$  and $\text{colim}_{I} P \epsilon^{k}F$
%are weakly equivalent.
%The statement  follows then from~\ref{prop hconst}.
%\end{proof}

\section{Mapping spaces in model categories}\label{sec mappinginmc}
%Let   $X$ be an object in $\mathcal{M}$\@. 
The homotopy colimits of  constant functors  
are homotopy invariant with respect to the indexing categories and thus 
the functor $\text{Spaces}\ni A\mapsto \text{hocolim}_A X\in\text{Ho}(\mathcal{M})$ is   a composition of the localization $\text{Spaces}\to \text{Ho}(\text{Spaces})$ and a functor denoted by  $X\otimes_l-\colon\text{Ho}(\text{Spaces})\to \text{Ho}(\mathcal{M})$  (\cite[Proposition 7.1]{MR2443229}). A key result in~\cite{MR2443229} states that $X\otimes_l-$ has a right adjoint $\text{map}(X,-)\colon\text{Ho}(\mathcal{M})\to \text{Ho}(\text{Spaces})$\@. Its value  $\text{map}(X,Y)$  is what we take to be the homotopy type of the mapping space between $X$ and $Y$ in $\calm$\@. If $\calm$ is a simplicial model category, then  the mapping space between a cofibrant and a fibrant objects given by the simplicial structure on $\calm$ is weakly equivalent to the value of this  right adjoint.
%In this section we recall from~\cite{MR2443229}  an explicit model for   $\text{map}(X,-)$.
%That was the aim of~\cite{MR2443229}. The aim of this section is to recall this construction.
%We start with some notation:
\begin{notation}\label{pt hcons}
The full subcategory in  $\text{Fun}^{b}(N(\Delta[0]),\calm)$ of   
 cofibrant and fibrant objects which are  weakly equivalent in  $\text{Fun}^{b}(N(\Delta[0]),\calm)$ to constant functors is denoted by 
 $\text{Cons}(N(\Delta[0]),\calm)$.
 
To make formulas more transparent the effect of  $N(p)^{\ast}\colon \text{Fun}^b(N(\Delta[0]),\calm)\to \text{Fun}^b(N(\Delta[n]),\calm)$ is denoted by adding the symbol $[n]$\@.  Thus  if $f\colon F\to G$ is a natural transformation in 
$\text{Fun}^b(N(\Delta[0]),\calm)$, then $f[n]\colon F[n]\to G[n]$ denotes the natural transformation
$N(p)^{\ast}f\colon N(p)^{\ast}F\to N(p)^{\ast}G$ in  $\text{Fun}^b(N(\Delta[n]),\calm)$\@. 
In particular $F=F[0]$ for any $F$ in $\text{Fun}^b(N(\Delta[0]),\calm)$\@.
Note that  $F[n]=N(\alpha)^{\ast} F[m]$  for any map $\alpha\colon\Delta[n]\to\Delta[m]$ and any functor $F$ in $\text{Fun}^b(N(\Delta[0]),\calm)$.
\end{notation}

We  now recall the construction of  mapping spaces in \smallskip a model category from~\cite{MR2443229}.

\noindent
{\bf Step 1.}
Let  $\alpha\colon\Delta[n]\to\Delta[m]$ be  a map. For   objects  
$F$ and $G$   in    $\text{Cons}(N(\Delta[0]),\calm)$,
%(see~\ref{pt hcons}), 
define 
 % sets $\text{map}(F,G)_m$, $\text{map}(F,G)_n$, and 
 a function of sets
 $\text{map}(F,G)_{\alpha}\colon\text{map}(F,G)_m\to \text{map}(F,G)_n$  by the formula:
\[\xymatrix@R=6pt@C=15pt{
  \text{map}(F,G)_m\ar@{*{\cdot\cdot}=}[d]\rrto^{\text{map}(F,G)_{\alpha}} & &  \text{map}(F,G)_n\ar@{*{\cdot\cdot}=}[d]\\
 \text{Nat}(F[m],G[m])\rrto^{N(\alpha)^{\ast}} & &  \text{Nat}(F[n],G[n])
}\]
In this way we obtain a simplicial set  $\text{map}(F,G)$.   
\smallskip

\noindent
{\bf Step 2.}
Let $f\colon E\to F$ and $g\colon G\to H$ be morphisms in  $\text{Cons}(N(\Delta[0]),\calm)$\@. Define a map of simplicial sets $\text{map}(f,g)\colon \text{map}(F,G)\to \text{map}(E,H)$  
by the formula:
\[\xymatrix@R=6pt@C=15pt{
\text{map}(F,G)_n\ar@{=}[d]\ar@{->}[rrr]^{\text{map}(f,g)_n} & &&  \text{map}(E,H)_n\ar@{=}[d]\\
 \text{Nat}(F[n],G[n])\ar@{->}[rrr]^(.47){\text{Nat}(f[n],g[n])} & &&   \text{Nat}_{\calm}(E[n],H[n])
}\]
In this way we have
constructed a functor:
\[\text{map}(-,-)\colon \text{Cons}(N(\Delta[0]),\calm)^{\text{op}}\times \text{Cons}(N(\Delta[0]),\calm\to \text{Spaces}
\smallskip\]

\noindent
{\bf Step 3.}
For any  $F$, $G$, and $H$  in $\text{Cons}(N(\Delta[0]),\calm)$, define the composition
$\circ\colon\text{map}(F,G)\times \text{map}(G,H)\to \text{map}(F,H)$
to be  in each degree the composition of natural transformations:
\[ \xymatrix@R=11pt@C=10pt{
\text{map}(F,G)_n\times \text{map}(G,H)_n\ar@{=}[r]\dto^{\circ_n} &\text{Nat}(F[n],G[n])\times 
 \text{Nat}(G[n],H[n])\dto &(\phi,\psi)\ar@{|->}[d]  \ar@{}[l]|(.23){\ni}\\ 
\text{map}(F,H)_n\ar@{=}[r] & \text{Nat}(F[n],H[n]) & \ar@{}[l]|(.23){\ni}\psi\phi,
}\smallskip\]

\noindent
{\bf Step 4.} For any   $F$  in  $\text{Cons}(N(\Delta[0]),\calm)$,  define $e_F\colon\Delta[0]\to \text{map}(F,F)$ to be given by the $0$-dimensional simplex  $\text{id}_{F[0]}$ in  $\text{map}(F,F)_0=\text{Nat}(F[0],F[0])$.
\smallskip

% The following proposition lists the main properties of the above defined mapping spaces and the composition:
  \begin{prop}
  \label{prop propertiesofmapping}
  Let $E$, $F$, $G$, and $H$ be objects in\/   $\text{\rm Cons}(N(\Delta[0]),\calm)$.
\begin{enumerate}
\item
The following diagrams commute
($\circ$ is associative and has  a unit):
\[\xymatrix@R=11pt@C=18pt{
\text{\rm map}(E,F)\times\text{\rm map}(F,G)\times \text{\rm map}(G,H)\rto^(.58){\circ\times\text{\rm id}}\dto_{\text{\rm id}\times\circ}& 
\text{\rm map}(E,G)\times \text{\rm map}(G,H)\dto^{\circ}\\
\text{\rm map}(E,F)\times\text{\rm map}(F,H)\rto^(.58){\circ} & \text{\rm map}(E,H)
}\]
\[\xymatrix{
\text{\rm map}(E, F)\times \Delta[0]\rto_(.44){\text{\rm id}\times e_F} \ar@/^16pt/[rr]|{\text{\rm pr}}  &  \text{\rm map}(E, F)\times \text{\rm map}(F, F)
\rto_(.63){\circ} &\text{\rm map}(E, F)
}\]
\[\xymatrix{
\Delta[0]\times \text{\rm map}(E, F)\rto^(.43){e_E\times\text{\rm id}} \ar@/ _ 16pt/[rr]|{\text{\rm pr}}  &  \text{\rm map}(E, E)\times \text{\rm map}(E, F)
\rto^(.63){\circ} &\text{\rm map}(E, F)
}\]
\item $\text{\rm map}(E,F)$ is Kan.
\item If $f\colon G\to E$ and $g\colon F\to H$ are weak
equivalences, % between functors in\/  $\text{\rm Cons}(N(\Delta[0]),\calm)$, 
then  the map of spaces $\text{\rm map}(f,g)\colon\text{\rm
map}(E,F)\to \text{\rm map}(G,H)$ is a weak
equivalence.
\item Let $\sigma$ be a simplex in $N(\Delta[0])$\@. The assignment which maps a\/ $0$-dimensional simplex  $f\colon F=F[0]\to G[0]=G$ in\/ $\text{\rm map}(F,G)$ to the  morphism   in $\text{\rm Ho}(\calm)$ 
represented by $f_{\sigma}\colon F(\sigma)\to G(\sigma)$ is   a bijection between  
$\pi_{0}\text{\rm map}(F,G)$ and\/
 $[F(\sigma),G(\sigma)]$\@.
This bijection is natural in $F$ and $G$ and  takes the composition $\circ$ to  the composition of morphisms in $\text{\rm Ho}(\calm)$.
\end{enumerate}
\end{prop}
\begin{proof}
 (1)  follows from analogous properties of  the  composition of natural transformations.\/
 (2)  is a particular case of~\cite[Proposition 8.5]{MR2443229} and (3) 
of~\cite[Corollary 8.4]{MR2443229}. Finally,  (4) is a consequence of~\cite[Corollary 8,6 and 
Proposition 9.2(2)]{MR2443229}.
\end{proof}

Let us  choose a functorial factorization  $P$ in $\text{Fun}^{b}(N(\Delta[0]),\calm)$
as defined in~\ref{pt modelcat} and a functorial fibrant replacement $R$ in $\calm$\@. 
By applying $R$ object-wise we obtain a functorial fibrant replacement in $\text{\rm Fun}^{b}(N(\Delta[0]),\calm)$\@. 
For any bounded functor $F\colon N(\Delta[0])\to\calm$ consider the factorization of $\emptyset \to RF$
given by $P$:
\[\xymatrix@R=11pt@C=15pt{\emptyset \dto \rmono & PRF\ar@{->>}[d]^(.4){\simeq} \\
F\ar@{->}[r]^(.45){\simeq} & RF}\]
Define $QF$ to  be   $PRF$\@.
It is a functorial cofibrant-fibrant replacement in  the category $ \text{\rm Fun}^{b}(N(\Delta[0]),\calm)$.
The natural weak equivalences
$F\xrightarrow{\simeq} RF \xleftarrow{\simeq} PRF=QF$
%$\xymatrix@R=10pt@C=15pt{F\rto^(.4){\simeq} & RF&\ar@{->>}[l]_(.6){\simeq}PRF=QF}$
are called the {\bf standard comparison morphisms} between $QF$ and $F$.

For an object  $X$  in $\calm$, the same symbol $X\colon N(\Delta[0])\to \calm$  is used to denote the constant 
functor with value  $X$. Note that $QX$ is 
an object in  $\text{\rm Cons}(N(\Delta[0]),\calm)$. In this way we obtain a functor
$Q\colon\calm\to \text{\rm Cons}(N(\Delta[0]),\calm)$\@.
Its  composition with the mapping space defined above  is denoted by $\text{map}(-,-)\colon \calm^{\text{op}}\times \calm\to \text{Spaces}$. 
%By definition, for two objects $X$ and $Y$ in $\calm$, $\text{map}(X,Y)=\text{map}(QX,QY)$.

\section{The spaces of weak equivalences}\label{sec we}
In Section~\ref{sec mappinginmc} we described a construction of  mapping spaces. % in a  model category.
    Here  we discuss the spaces of weak equivalences.
Let $F$ and $G$ be objects in  $\text{\rm Cons}(N(\Delta[0]),\calm)$ and $\sigma$ be a simplex 
in $N(\Delta[0])$. The set of components $\pi_0\text{map}(F,G)$ is in bijection with
the set of morphisms $[F(\sigma),G(\sigma)]$ in $\text{Ho}(\calm)$\@.  The component of $\text{map}(F,G)$
corresponding to $\alpha$ in $ [F(\sigma),G(\sigma)]$ is  denoted by $\text{map}(F,G)_{\alpha}$.  
Define:
%We define the space of weak equivalences 
%$\text{we}(F,G)$ to be:
\[\text{we}(F,G):=\coprod_{\alpha\in [F(\sigma),G(\sigma)] \text{ is an isomorphism}}\text{map}(F,G)_{\alpha}\]
Since $N(\Delta[0])$ is connected and  $F$ and $G$ are weakly equivalent to constant functors, 
$\text{we}(F,G)$ does not depend on the choice of the simplex $\sigma$.

The spaces of weak equivalences $\text{we}(F,G)$
do  not  form a subfactor in  
$\text{map}(-,-)$, since 
in general $\text{map}(f,\text{id})\colon\text{map}(F,G)\to\text{map}(E,G)$ does not take
the component corresponding to  an isomorphism in $[F(\sigma),G(\sigma)]$ to a component  corresponding to
an isomorphism in $[E(\sigma),G(\sigma)]$\@. However, if both $f\colon E\to F$ and $g\colon G\to H$ are weak equivalences, then $\text{map}(f,g)$ takes $\text{we}(F,G)$ into $\text{we}(E,H)$ and in this case we use the symbol 
$\text{we}(f,g)\colon\text{we}(F,G)\to \text{we}(E,H)$ to denote the induced map. Thus $\text{we}(f,g)$ is defined only if $f$ and $g$
are weak equivalences.

Let $F$ and $G$ be objects in $\text{\rm Cons}(N(\Delta[0]),\calm)$\@. Define $\text{Natwe}(F[n],G[n])$  
to be the subset of $\text{map}(F,G)_n=\text{Nat}(F[n],G[n])$ which consists of the natural transformations
$f\colon F[n]\to G[n]$ in $\text{Fun}^{b}(N(\Delta[n]),\calm)$ which are weak equivalences.

\begin{prop} 
$\text{\rm we}(F,G)_n=\text{\rm Natwe}(F[n],G[n])$.
\end{prop}
\begin{proof}
Let $f$ be an element in $\text{map}(F,G)_n=\text{Nat}(F[n],G[n])$\@.
As $F[n]$ and $G[n]$ are  weakly equivalent to  constant functors,  $f$ is a weak equivalence if and only if  for some  $\tau$ in $N(\Delta[n])$,
$f_{\tau}\colon F[n](\tau)\to G[n](\tau)$ is a weak equivalence.
Choose  a map $v\colon\Delta[0]\to \Delta[n]$ and a simplex $\sigma$ in $N(\Delta[0])$\@.
Let $\tau=N(v)(\sigma)$\@. Note that  $F[n](\tau)=F(\sigma)$ and $G[n](\tau)=G(\sigma)$\@.
According to~\ref{prop propertiesofmapping}.(4)  $f$ is  an $n$-dimensional simplex in 
$\text{map}(F,G)$ that belongs to the component  determined by the  morphism
in $[F(\sigma),G(\sigma)] $ represented by
$f_{\tau}\colon F[n](\tau)=F(\sigma)\to G(\sigma)=G[n](\tau)$\@.
This morphism is  an isomorphism in $\text{Ho}(\calm)$ if and only if $f_{\tau}$ is a weak equivalence.
\end{proof}

For example since $\text{id}\colon F[0]\to F[0]$ belongs to $\text{Natwe}(F[0],F[0])$, the map $e_F\colon\Delta[0]\to \text{map}(F,F)$ factors as a composition of a map which  we denote by the same symbol $ e_F\colon\Delta[0]\to \text{we}(F,F)$ and the inclusion $\text{we}(F,F)\subset  \text{map}(F,F)$.

Compositions of weak equivalences are weak equivalences and hence 
the restriction of the composition operation  $\circ$  fits into a commutative diagram:
 \[\xymatrix@R=11pt@C=15pt{
\text{we}(F,G)\times  \text{we}(G,H)\ar@{^(->}[d]\rto^(.6){\circ} & \text{we}(F,H)\ar@{^(->}[d]\\
\text{map}(F,G)\times  \text{map}(G,H)\rto^(.63){\circ} &  \text{map}(F,H)
}\]
This together with~\ref{prop propertiesofmapping} gives:
\begin{cor}\label{cor propertiesofwe}
Let $E$, $F$, $G$, and $H$ be objects in\/ $\text{\rm Cons}(N(\Delta[0]),\calm)$.
\begin{enumerate}
\item
The following diagrams commute ($\circ$ is associative and has a unit):
\[\xymatrix@R=11pt@C=18pt{
\text{\rm we}(E,F)\times\text{\rm we}(F,G)\times \text{\rm we}(G,H)\rto^(.58){\circ\times\text{\rm id}}\dto_{\text{\rm id}\times\circ}& 
\text{\rm we}(E,G)\times \text{\rm we}(G,H)\dto^{\circ}\\
\text{\rm we}(E,F)\times\text{\rm we}(F,H)\rto^(.58){\circ} & \text{\rm we}(E,H)
}\]
\[\xymatrix{
\text{\rm we}(E, F)\times \Delta[0]\rto_-{\text{\rm id}\times e_F} \ar@/^16pt/[rr]|{\text{\rm pr}}  &  \text{\rm we}(E, F)\times \text{\rm we}(F, F)
\rto_-{\circ} &\text{\rm we}(E, F)
}\]
\[\xymatrix{
\Delta[0]\times \text{\rm we}(E, F)\rto^(.43){e_E\times\text{\rm id}} \ar@/ _ 16pt/[rr]|{\text{\rm pr}}  &  \text{\rm we}(E, E)\times \text{\rm we}(E, F)
\rto^(.63){\circ} &\text{\rm we}(E, F)
}\]

\item  $\text{\rm we}(F,G)$ is Kan.
\item If $f\colon G\to E$ and $g\colon F\to H$ are weak
equivalences, then so is the map of spaces\/  $\text{\rm we}(f,g)\colon\text{\rm
we}(E,F)\to \text{\rm we}(G,H)$.
%\item Let $\sigma$ be a simplex in $N(\Delta[0])$. The assignment which maps a $0$-dimensional simplex  $f:F=F[0]\ra G[0]=G$ in $\text{\rm we}(F,G)$ to the morphism  in $\text{\rm Ho}(\calm)$
%represented  by
%$f_{\sigma}:F(\sigma)\ra G(\sigma)$ 
%induces  a bijection  between 
%$\pi_{0}\text{\rm we}(F,G)$ and  
%$[F(\sigma),G(\sigma)]$.
%This bijection is natural in $F$ and $G$ and  takes the composition $\circ$ to  the composition of morphisms in $\text{\rm Ho}(\calm)$.
\end{enumerate}
\end{cor}

We finish this section with:
\begin{prop}
\label{prop wehgoup}
Let $F$, $G$, and $H$ be objects in\/ $\text{\rm Cons}(N(\Delta[0]),\calm)$\@. Then the following square is a homotopy pull-back:
\[\xymatrix@R=11pt@C=15pt{
\text{\rm we}(F,G)\times\text{\rm we}(G,H)\rto^(.62){\circ} \dto_{\text{\rm pr}} & \text{\rm we}(F,H)\dto\\
\text{\rm we}(F,G)\rto & \Delta[0]
}\]
\end{prop}
\begin{proof}
We leave to the reader the proof of  the proposition in the case when one of the spaces 
$\text{\rm we}(F,G)$, or $\text{\rm we}(G,H)$, or 
$\text{\rm we}(F,H)$ is empty.  Assume then that all these spaces are  not empty.
We need to  show that $\circ$ induces a weak equivalence between the homotopy fibers 
of  $\text{pr}$ and $\text{we}(F,H)$\@.
Choose a vertex in  $\text{we}(F,G)$ represented by
a weak equivalence $f\colon F[0]\to G[0]$\@.
The homotopy fiber of  the  map $\text{pr}$ over the component of $f$ is given by
$\text{we}(G,H)$\@. The restriction of the composition  $\circ$ 
to this homotopy fiber is given by $\text{we}(f,\text{id})$ which is a weak equivalence
by~\ref{cor propertiesofwe}.(3).
\end{proof}

Let $X$ and $Y$ be objects in $\calm$\@. Define $\text{we}(X,Y)$ as  $\text{we}(QX,QY)$,
where $Q$ is the functorial cofibrant-fibrant replacement  in $\text{Fun}^b(N(\Delta[0]),\calm)$
(see Section~\ref{sec mappinginmc}).

\section{Deloopings of the spaces of weak equivalences}
In this section we describe the standard  delooping of  the spaces of weak equivalences with the monoid structure given by the composition.  The proof of the  main result and needed techniques are places in the appendix (Section~\ref{appendix}).

\begin{notation}\label{not conswesystem}
Let $X$ be an object  in $\calm$\@. The symbol $\text{\rm Cons}(N(\Delta[0]),X_{\text{we}})$
denotes the full subcategory of $\text{Fun}^b(N(\Delta[0]), X_{\text{we}})$ whose objects
are functors  whose composition with
the inclusion $X_{\text{we}}\subset \calm$ 
is cofibrant, fibrant, and weakly equivalent  to the  constant functor $X$
in $\text{Fun}^b(N(\Delta[0]), \calm)$
(compare with~\ref{pt hcons}). The set of morphisms   between two such functors $F$ and $G$ is given by $\text{Natwe}(F,G)$.

Let   $\cals$ be a collection of objects in $\text{\rm Cons}(N(\Delta[0]),X_{\text{we}})$\@. Define
$\cals_n$ to be the full subcategory of $\text{Fun}^b(N(\Delta[n]), X_{\text{we}})$
 of objects of the form $F[n]$ where
$F$ is in $\cals$  (see~\ref{pt hcons}). 
The set of morphisms between  $F[n]$ and $G[n]$ in  $\cals_n$ is given by 
$\text{Natwe}(F[n],G[n])$\@.
For any  $\alpha\colon\Delta[n]\to\Delta[m]$, 
$F[n]=N(\alpha)^{\ast} F[m]$. The restriction of $ N(\alpha)^{\ast}\colon
\text{Fun}^b(N(\Delta[m]), X_{\text{we}})\to \text{Fun}^b(N(\Delta[n]), X_{\text{we}})$
induces therefore a functor $N(\alpha)^{\ast}\colon\cals_m\to \cals_n$  which we denote by $\cals_{\alpha}\colon\cals_m\to \cals_n$. In this way we obtain a system of categories $\cals_{-}$
indexed by $\Delta$.   
\end{notation}

Assume  $\cals$ is a {\bf set} of objects in $\text{\rm Cons}(N(\Delta[0]),X_{\text{we}})$\@. For  $n\geq 0$, 
the category $\cals_n$ is then   small. By taking the nerves we obtain a functor $N(\cals_{-})\colon\Delta^{\text{op}}\to \text{Spaces}$.

\begin{Def}\label{def Bwe}
$B\text{\rm we}(\mathcal{S})$  is defined to  be 
the diagonal of $N(\cals_{-})$.
\end{Def}

By definition, the set of $0$-dimensional simplices $B\text{we}(\mathcal{S})_0$ is given by the set of objects in
$\cals_0$ which is  the set $\cals$\@.
%The set of $1$-dimensional simplices  $B\text{we}(\mathcal{S})_1$ is given by the
%morphisms in $\cals_1$:
%\[B\text{we}(\mathcal{S})_1=\coprod_{(X_1,X_0)\in\mathcal{S}^2}\text{Natwe}(X_1[1],X_0[1])\]
For $n>0$, the set of $n$-dimensional simplices 
$B\text{we}(\mathcal{S})_n$ is the set of $n$-composable morphisms in $\cals_n$:
\[B\text{we}(\mathcal{S})_n=\coprod_{(X_n,\ldots, X_0)\in\mathcal{S}^{n+1}}\prod_{k=n}^{k=1}\text{Natwe}(X_{k}[n],X_{k-1}[n])\]

If $\mathcal{S}'\subset \mathcal{S}$, then ${\cals'}_n$ is a full  subcategory of ${\cals}_n$\@.
%The map $N({\cals'}_{-})\subset N({\cals}_{-})$ induced by these inclusions 
We call the induced map $N({\cals'}_{-})\subset N({\cals}_{-})$ 
the 
standard inclusion. The following is a consequence of~\ref{prop appendix deloopingwe}:
\begin{prop}\label{prop deloopingwe}
Let  $\mathcal{S}'\subset \mathcal{S}$ be non-empty  sets of objects in\/ $\text{\rm Cons}(N(\Delta[0]),X_{\text{\rm we}})$.
\begin{enumerate} 
\item  $B\text{\rm we}(\mathcal{S})$ is a connected space.
\item  The  loop space\/  $\Omega B\text{\rm we}(\mathcal{S})$ is weakly equivalent to\/ $\text{\rm we}(X,X)$.
\item The map $B\text{\rm we}(\mathcal{S}')\to B\text{\rm we}(\mathcal{S})$, induced by the standard inclusion,
is a weak equivalence.
\end{enumerate}
\end{prop}
%\begin{proof}
%See the appendix, Section~\ref{appendix}.
%\end{proof}
\bigskip

\noindent
{\bf Part III.}
%   The categories of functors with values in  $X_{\text{we}}$.}   
In this part we show that $\text{Fun}(I,X_{\text{we}})$ is essentially small and the nerve of its core is weakly equivalent to  $\text{map}(N(I), B\text{we}(X,X))$. This will be applied to  prove Theorem A. 
% from the introduction.

%%%%%%%%%%%%%%%%%%%%%%%%%%%%%%%%%%%%%%%
%%%%%%%%%%%%%% Catofweakequiv  %%%%%%%%%%%%%%%%%%%
%%%%%%%%%%%%%%%%%%%%%%%%%%%%%%%%%%%%%%%
\section{The category of weak equivalences}
\begin{thm}\label{thm bweXX}
Let $X$ be an object in $\calm$\@.
The category $X_{\text{\rm we}}$ is essentially small and the nerve of its core is weakly equivalent to
$B\text{\rm we}(X, X)$.
\end{thm} 

\begin{proof}
Let $\calt$ be the collection of all the objects in $\text{Cons}(N(\Delta[0]),X_{\text{we}})$\@.
Consider the system $\calt_{-}$ indexed by $\Delta$ (see~\ref{not conswesystem}).
% and its  Grothendieck construction
%$\text{Gr}_{\Delta} \calt_{-}$. 
The objects in  its  Grothendieck construction $\text{Gr}_{\Delta} \calt_{-}$
are given by   functors $F[n]$ where $n\geq 0$ and $F$  is an object in  $\text{Cons}(N(\Delta[0]),X_{\text{we}})$\@. A morphism in $\text{Gr}_{\Delta} \calt_{-}$ between two such functors $F[n]$ and $G[m]$ is 
 a pair 
$(\alpha\colon\Delta[n]\to \Delta[m], \phi\colon F[n]\to G[n])$ where  $\alpha$ is a map and  $\phi$
is a natural weak 
equivalence. Define a functor $\Psi\colon\text{Gr}_{\Delta}\calt_{-}\to\calm$ as
follows:
\[\xymatrix@R=6pt@C=15pt{
\Psi(F[n])\ar@{*{\cdot\cdot}=}[d]\ar[rrrr]|{\Psi(\alpha,\phi)} & & & &  \Psi(G[m])\ar@{*{\cdot\cdot}=}[d]\\
\text{colim}_{N(\Delta[n])}F[n]\rrto^{\text{colim}(\phi)}    & & \text{colim}_{N(\Delta[n])}G[n]
 \rrto^{\text{colim}_{N(\alpha)}} & &  \text{colim}_{N(\Delta[m])}G[m]
}\]
Since $\Delta[n]$ is contractible, by~\ref{prop hocolimovercontract},
$\Psi$ has values in $X_{\text{we}}$\@. We claim  the induced functor
 $\Psi\colon\text{Gr}_{\Delta}\calt_{-}
\to X_{\text{we}}$
is a homotopy equivalence. 
To prove the claim  choose a functorial   factorization $P$  in $\text{Fun}^b(N(\Delta[0]),\calm)$
as defined  in~\ref{pt modelcat} and
a functorial fibrant replacement $R$ in $\calm$\@. 
By applying  $P$ to  morphisms of the form $\emptyset \to F$ we obtain a functorial cofibrant replacement in 
$\text{Fun}^b(N(\Delta[0]),\calm)$ (see~\ref{pt modelcat}).
By applying $R$ object-wise we obtain a functorial fibrant replacement in $\text{Fun}^b(N(\Delta[0]),\calm)$\@. Set  $QF=PRF$
(see the end of Section~\ref{sec mappinginmc}).
Let us denote by the same symbol $Q\colon X_{\text{we}}\to \text{Cons}(N(\Delta[0]),X_{\text{we}})$ the 
restriction of $Q$ to the constant functors.
Define $\Phi\colon X_{\text{we}}\to
\text{Gr}_{\Delta} \calt_{-}$ to be  the following composition:
\[\xymatrix@R=11pt@C=17pt{
X_{\text{we}}\rto^-{Q} \ar@/ _ 14pt/[rrrr]|{\Phi}  &  \text{Cons}(N(\Delta[0]),X_{\text{we}})=\calt_0\ar@{->}[rrr]^-{\text{standard inclusion}}  && &  \text{Gr}_{\Delta}\calt_{-}
}\]
%\[\xymatrix@R=11pt@C=15pt{
%X_{\text{we}}\dto_{Q}\rto|{\Phi} & \text{Gr}_{\Delta}\calt_{-}\\
% \text{Cons}(N(\Delta[0]),X_{\text{we}})\ar@{=}[r] &\calt_0\uto_{\text{standard inclusion}}   }\]

We will show  that   $\Psi\Phi$ and $\Phi \Psi$ are homotopic to the identity
functors. 
 The  composition  $\Psi\Phi\colon X_{\text{we}}\to X_{\text{we}}$ assigns to  $Y$
the object $\text{colim}_{N(\Delta[0])} QY$\@. 
Think about the simplex ${\sigma}\colon\Delta[0]\to \Delta[0]$ as a vertex in $N(\Delta[0])$\@. 
According to~\ref{prop hocolimovercontract} the morphism
$QY(\sigma)\to \text{colim}_{N(\Delta[0])} QY$, induced by the  inclusion of   $\sigma$ in $N(\Delta[0])$,
is a weak equivalence. Consequently the following morphisms form
a ``zig-zag'' of natural weak equivalences between  $\text{id}_{X_{\text{we}}}$ and
$\Psi\Phi$:
\[
\underbrace{Y\to RY\gets QY(\sigma) }_{\text{standard comparison}}\to  \text{colim}_{N(\Delta[0])} QY=\Psi\Phi(Y)
\]

The composition $\Phi\Psi\colon\text{Gr}_{\Delta} \calt_{-}\to \text{Gr}_{\Delta}\calt_{-}$
assigns to $F[n]\colon N(\Delta[n])\to X_{\text{we}}$ the functor $Q(\text{colim}_{N(\Delta[n])} F[n])\colon N(\Delta[0])\to X_{\text{we}}$\@.
Consider the left Kan extension  $N(p)^k(F[n])\colon N(\Delta[0])\to \calm$\@.
Denote by $N(p)^k(F[n]) \to \text{colim}_{N(\Delta[n])}F[n]$ the canonical natural transformation
 into the constant functor $ \text{colim}_{N(\Delta[n])}F[n]$\@.
 We denote the composition of this natural transformation with the 
fibrant replacement $\text{colim}_{N(\Delta[n])}F[n]\to R(\text{colim}_{N(\Delta[n])}F[n])$ by $\pi$.
Finally, take the following factorizations in $\text{Fun}^b(N(\Delta[0]),\calm)$ given in~\ref{pt modelcat}:
\[\xymatrix@R=9pt@C=15pt{
\emptyset\dto  \rmono & PR(\text{colim}_{N(\Delta[n])}F[n])\ar@{->>}[r]^{\simeq}\dto^{\simeq} &R(\text{colim}_{N(\Delta[n])}F[n]) \ar@{=}[d]\\
N(p)^{k}(F[n]) \ar@/_14pt/[rr]|{\pi}  \rmono& P(\pi)\ar@{->>}[r]^(.37){\simeq}& R(\text{colim}_{N(\Delta[n])}F[n])
}\]
Observe that $ N(p)^{k}(F[n])$ is cofibrant  in $\text{Fun}^b(N(\Delta[0]),\calm)$
(see~\ref{prop basicprophocolim}.(2)). Thus  the functors
$P(\pi)$ and $PR(\text{colim}_{N(\Delta[n])}F[n])$ are objects in  $\text{Cons}(N(\Delta[0]),X_{\text{we}})$.
Consider next the following sequence of morphisms in $\text{Gr}_{\Delta}\calt_{-}$:
\[
\xymatrix@1@R=10pt@C=13pt{
F[n]\rto^{(p,\phi)} & P(\pi) & PR(\text{colim}_{N(\Delta[n])}F[n])=Q(\text{colim}_{N(\Delta[n])}F[n])=\Phi\Psi(F[n])\lto
}\]
 where $p\colon\Delta[n]\to \Delta[0]$ is the unique map,  $\phi\colon F[n]\to P(\pi)[n]$ is adjoint to 
$N(p)^k(F[n])\hookrightarrow P(\pi)$, and $P(\pi)\gets PR(\text{colim}_{N(\Delta[n])}F[n])$ is the vertical morphism in the above
commutative diagram.
%\begin{itemize}
%\item  $p:\Delta[n]\ra \Delta[0]$ is the unique map;
%\item $\phi:F[n]\ra P(\pi)[n]$ is adjoint to 
%the morphism $N(p)^k(F[n])\hookrightarrow P(\pi)$;
%\item $P(\pi)\xleftarrow{\simeq} PR(\text{colim}_{N(\Delta[n])}F[n])$ is the vertical morphism in the above
%commutative diagram.
%\end{itemize}
These morphisms  give a  ``zig-zag'' of natural weak equivalences between
% the identity functor
 $\text{id}_{\text{Gr}_{\Delta} \calt_{-}}$ and $\Phi\Psi$.

 Since $X_{\text{we}}$ and $\text{Gr}_{\Delta}\calt_{-}$ are homotopy equivalent, 
 one of them is essentially small if and only if the other one is (see~\ref{cor esssmallhomeqinv}).  Consider %the constant functor $X:N(\Delta[0])\ra \calm$  and 
 the one element set $\{QX\}$. % in $\text{Cons}(N(\Delta[0]),X_{\text{we}})$\@. 
We claim that $\text{Gr}_{\Delta}\{QX\}_{-}\subset \text{Gr}_{\Delta}\calt_{-}$ is a core.
Let  $J\subset \text{Gr}_{\Delta}\calt_{-}$
be a small subcategory  containing
$\text{Gr}_{\Delta}\{QX\}_{-}$ and  $\cals$ be the  set of all  $F$ in $\text{Cons}(N(\Delta[0]),X_{\text{we}})$  such
that, for some $n\geq 0$,  $F[n]$ is  in $J$\@. We  have a sequence of inclusions
$\text{Gr}_{\Delta}\{QX\}_{-}\subset  J\subset  \text{Gr}_{\Delta}\cals$. Its composition
 is a weak equivalence by~\ref{prop deloopingwe}.(3), which shows the claim.

By  Thomason's theorem~\ref{prop ThomasonPuppe}.(1)) and the fact that homotopy colimit of a simplicial space is weakly equivalent to its diagonal we get that the following spaces are weakly equivalent  to each other:
the   nerve $N(\text{Gr}_{\Delta}\{QX\}_{-})$, the homotopy colimit $\text{hocolim}_{\Delta^{\text{op}}} N(\{QX\}_{-})$
and the diagonal  $B\text{we}(X,X)$ of  $N(\{QX\}_{-})$ (see~\ref{def Bwe}).
%The   nerve $N(\text{Gr}_{\Delta}\{QX\}_{-})$, according to 
% Thomason's theorem~\ref{prop ThomasonPuppe}.(1)),   is weakly equivalent to $\text{hocolim}_{\Delta^{\text{op}}} N(\{QX\}_{-})$\@. The homotopy colimit of a simplicial space is weakly equivalent to its diagonal
% and thus $N(\text{Gr}_{\Delta}(\{QX\}_{-}))$ is weakly equivalent to the diagonal of $N(\{QX\}_{-})$ which 
%by~\ref{def Bwe} is $B\text{we}(X,X)$.
%
 %which shows the claim.
%According to~\ref{prop deloopingwe}.(3) their composition  is a weak equivalence showing the claim.
%The subcategory $\text{Gr}_{\Delta}\{QX\}_{-}\subset \text{Gr}_{\Delta}\calt_{-}$ is therefore a core.
 \end{proof}

\section{Cofinality}
%In this section we start discussing categories of functors with values in the component $X_{\text{we}}$.
%The main result   is what we will refer to as cofinality:
\begin{prop}\label{prop cofinalwecat}
Let $f\colon I\to J$ be a functor between small categories. Assume that, for any object $j$ in $J$,
the over category $f\!\downarrow\! j$ is contractible. Then 
the functor $f^{\ast}\colon\text{\rm Fun}(J,X_{\text{\rm we}})\to
\text{\rm Fun}(I,X_{\text{\rm we}})$ is a homotopy equivalence.
\end{prop}
\begin{proof}
Let $P$ be a cofibrant replacement in $ \text{\rm Fun}(I,\calm)$ (see Section~\ref{sec hocollkex}).
Recall  that  the derived left Kan extension $f^{k}P\colon\text{\rm Fun}(I,\calm)\to
\text{\rm Fun}(J,\calm)$ maps the subcategory $\text{\rm Fun}(I,X_{\text{\rm we}})\subset \text{\rm Fun}(I,\calm)$ 
into $\text{\rm Fun}(J,X_{\text{\rm we}})$ (see Section~\ref{sec hdkepxwe}).
We claimed that the  functor 
$f^{k}P\colon\text{\rm Fun}(I,X_{\text{\rm we}})\to \text{\rm Fun}(J,X_{\text{\rm we}})$ is a homotopy inverse to $f^{\ast}$.
% i.e.  the compositions $f^{k}Pf^{\ast}$ and $f^{\ast}f^{k}P$ are homotopic to the  identity functors between 
%$\text{\rm Fun}(J,X_{\text{\rm we}})$ and $\text{\rm Fun}(I,X_{\text{\rm we}})$.

Let $F\colon J\to X_{\text{we}}$ be a functor, $Pf^{\ast} F\to f^{\ast} F$ 
 the cofibrant replacement, and 
$\phi_F\colon f^kPf^{\ast} F\to F$  its adjoint.
Choose an object $a=(i, \alpha\colon f(i)\to j)$ in   $f\!\downarrow\! j$ and consider the following commutative diagram: 
%The morphism $(\phi_F)_j:(f^kPf^{\ast} F)(j)\ra F(j)$ fits into the following   commutative diagram:
\[\xymatrix@R=10pt@C=15pt{
(Pf^{\ast}F)(i)\dto \ar@{->>}[r]^{\simeq} & F(f(i))\rto^{F(\alpha)} & F(j)\\
\text{colim}_{f\downarrow j}Pf^{\ast}F\ar@{=}[r] & (f^{k}Pf^{\ast}F)(j)\ar@/_15pt/[ur]_(.75){(\phi_F)_j}
}\]
where the left vertical morphism is induced by the inclusion of the object $a$ in  $f\!\downarrow\! j$.
According to~\ref{prop hconst} this morphism is a weak equivalence and hence so is
$(\phi_F)_j$\@. The morphisms $\{\phi_F\colon f^kPf^{\ast} F\to F\}_{F}$ 
form therefore a natural transformation between  $f^kPf^{\ast}\colon\text{\rm Fun}(J,X_{\text{\rm we}})\to 
\text{\rm Fun}(J,X_{\text{\rm we}})$ and the identity functor. 

By the same argument the morphism  $\psi_F\colon PF\to f^{\ast}f^{k}PF$, which is adjoint to
 $\text{id}_{f^{k}PF}$, is also a natural  weak equivalence.
Thus the natural transformations $F\gets PF$ and  $\psi_F$
form a ``zig-zag''  connecting $f^{\ast}f^kP$ with 
$\text{id}_{\text{\rm Fun}(I,X_{\text{\rm we}})}$.
%\[\xymatrix@R=10pt@C=15pt{F & PF\ar@{->>}[l]_{\simeq} \rto^(.4){\psi_F} & f^{\ast}f^{k}PF}\]
%form a ``zig-zag''  connecting $f^{\ast}f^kP$ with the identity functor 
%$\text{id}_{\text{\rm Fun}(I,X_{\text{\rm we}})}$.
\end{proof}

Using the above  proposition we can extend Theorem~\ref{thm bweXX}  to:
\begin{cor}\label{cor funoutofcotract}
If $I$ is a small contractible category, then\/ $\text{\rm Fun}(I,X_{\text{\rm we}})$ is essentially
small and the nerve of its core is weakly equivalent to $B\text{\rm we}(X,X)$. 
\end{cor}
%\begin{proof}
%Since $I$ is  contractible, it follows from Proposition~\ref{prop cofinalwecat} that
%the constant functor inclusion $X_{\text{we}}\subset \text{\rm Fun}(I,X_{\text{\rm we}})$ is a homotopy equivalence. Thus, since $X_{\text{we}}$ is essentially small, then  so is
%$\text{\rm Fun}(I,X_{\text{\rm we}})$ (see Corollary~\ref{cor esssmallhomeqinv}) and the inclusion $X_{\text{we}}\subset \text{\rm Fun}(I,X_{\text{\rm we}})$ is a weak equivalence (see Proposition~\ref{prop esssmallweakequiv}).
%The corollary now follows from Theorem~\ref{thm bweXX}.
%\end{proof}

The cofinality result~\ref{prop cofinalwecat} 
can also be used to  translate between categories of functors indexed by  arbitrary small categories and  simplex categories:
%categories of functors indexed by simplex categories:
\begin{cor}\label{cor redtonerve}
Let $I$ be a small category and  $\epsilon\colon N(I)\to  I$ be the functor defined in Section~\ref{sec boundedfunctors}.
Then both functors $\epsilon^{\ast}\colon\text{\rm Fun}(I,X_{\text{\rm we}})\to \text{\rm Fun}(N(I),X_{\text{\rm we}})$ and
 $\epsilon^{\ast}\colon\text{\rm Fun}(I,X_{\text{\rm we}})\to \text{\rm Fun}^{b}(N(I),X_{\text{\rm we}})$ 
 are homotopy equivalences.
%\begin{enumerate}
%\item  $\epsilon^{\ast}:\text{\rm Fun}(I,X_{\text{\rm we}})\ra \text{\rm Fun}(N(I),X_{\text{\rm we}})$
%is a homotopy equivalence.
%\item $\epsilon^{\ast}:\text{\rm Fun}(I,X_{\text{\rm we}})\ra \text{\rm Fun}^{b}(N(I),X_{\text{\rm we}})$
%is a homotopy equivalence.
%\end{enumerate}
\end{cor}
\begin{proof}According to the proof of~\ref{prop cofinalwecat}, 
$\epsilon^{k}P\colon
\text{\rm Fun}(N(I),X_{\text{\rm we}})\to \text{\rm Fun}(I,X_{\text{\rm we}})$ is a homotopy inverse to the first functor
in the statement of the corollary. Its restriction to $\text{\rm Fun}^{b}(N(I),X_{\text{\rm we}})$ is  a homotopy inverse to the second functor.
\end{proof}

The following special cases of~\ref{prop cofinalwecat} are of particular interest to us:
%We will be particularly interested in the following special cases of~\ref{prop cofinalwecat}:
\begin{cor}\label{cor cofinalitycons}
\begin{enumerate}
\item   If 
$df(\sigma)=
\text{\rm lim}(\Delta[n]\xrightarrow{\sigma} B\xleftarrow{f} A)$
is contractible  for any   $\sigma$ Êin $B$, then 
 $f^{\ast}\colon\text{\rm Fun}(B,X_{\text{\rm we}})\to \text{\rm Fun}(A,X_{\text{\rm we}})$
is a homotopy equivalence.
\item Let $K$ be a contractible simplicial set and\/ $\text{\rm pr}:B\times K\ra B$ be the projection.
Then $\text{\rm pr}^{\ast}\colon\text{\rm Fun}(B,X_{\text{\rm we}})\to \text{\rm Fun}(B\times K,X_{\text{\rm we}})$
is a homotopy equivalence.
\item  Let 
$A_0\subset A_1\subset A_2\subset \cdots  A$ be an increasing sequence of  subspaces  of $A$ such  that  $A=\cup_{i\geq 0} A_i$\@. Define $f\colon\text{\rm Tel}\to A$ to be the following map:
\[\xymatrix@R=12pt@C=22pt{
\text{\rm Tel}\dto_{f}\ar @{}[r]|(0.46){:=} &
*{\text{\rm colim}\hspace{1mm}\big(\hspace{-18pt}} & \coprod_{i\geq 0} A_i\times \Delta[1] \dto_{ \coprod_{i\geq 0}\text{\rm pr}}
& (\coprod_{i\geq 0} A_i) \coprod(\coprod_{i\geq 0}  A_i) \lto_(.55){g_0\coprod g_1}
\rto^(.65){\text{\rm id}\coprod \text{\rm id}}\ar@{=}[d]  &   \coprod_{i\geq 0} A_i\ar@{=}[d] & *{\hspace{-18pt}\big)}\\
A \ar @{}[r]|(0.46)=& *{\text{\rm colim}\hspace{1mm}\big(\hspace{-18pt}} &
\coprod_{i\geq 0} A_i &  (\coprod_{i\geq 0} A_i) \coprod(\coprod_{i\geq 0}  A_i)\lto\rto^(.65){\text{\rm id}\coprod \text{\rm id}} &   \coprod_{i\geq 0} A_i & *{\hspace{-18pt}\big)} }
\]
where on each component:
\begin{itemize}
\item  $g_0$ is the inclusion  $A_i= A_i\times\Delta[0]
\xrightarrow{\text{\rm id}\times d_0}A_i\times\Delta[1]$;
\item $g_1$ is the inclusion $A_i\subset A_{i+1}= A_{i+1}\times\Delta[0]
\xrightarrow{\text{\rm id}\times d_1}A_{i+1}\times\Delta[1]$.
\end{itemize}
Then $f^{\ast}\colon\text{\rm Fun}(A,X_{\text{\rm we}})\to \text{\rm Fun}(\text{\rm Tel},X_{\text{\rm we}})$
is a homotopy equivalence.
\end{enumerate}
\end{cor}
%\begin{proof}
%Statement (1) is a particular case of Proposition~\ref{prop cofinalwecat} since the simplex category of 
%$df(\sigma)$ is isomorphic to $f\downarrow \sigma$. Statements (2) and (3) are particular cases of (1).
%\end{proof}

\section{Clutching}
%In this section we continue discussing categories of functors with values in  $X_{\text{we}}$.
%We focus on simplex categories as indexing categories. We will be particularly interested in how 
%functor categories change when we change the indexing space.
Recall that  $\text{\rm Cof}(A,X_{\text{\rm we}})$ denotes the full subcategory
of $\text{\rm Fun}^{b}(A,X_{\text{\rm we}})$ whose objects are  bounded functors 
$F\colon A\to X_{\text{\rm we}}$ whose  composition with $X_{\text{we}}\subset\calm$ is  cofibrant  in
$\text{\rm Fun}^{b}(A,\calm)$ (see Section~\ref{sec boundedfunctors}).

\begin{prop}\label{prop cofinallbund}
$\text{\rm Cof}(A,X_{\text{\rm we}})\subset  \text{\rm Fun}^{b}(A,X_{\text{\rm we}})$ is a homotopy equivalence.
% for any simplicial set $A$.
\end{prop}
\begin{proof}
A homotopy inverse  is given by
a cofibrant replacement.
\end{proof}

Recall  that  if $f\colon A\to B$ is reduced (see Section~\ref{sec simplex}), then
$f^{\ast}\colon\text{Fun}^{b}(B,X_{\text{\rm we}})\to \text{Fun}^{b}(A,X_{\text{\rm we}})$ maps
the subcategory $\text{\rm Cof}(B,X_{\text{\rm we}})$
into $\text{\rm Cof}(A,X_{\text{\rm we}})$.

\begin{prop}\label{prop cofstronfibration}
If  $f:A\ra B$ is reduced, then
$f^{\ast}\colon\text{\rm Cof}(B,X_{\text{\rm we}})\to \text{\rm Cof}(A,X_{\text{\rm we}})$
is a strong fibration (see Section~\ref{sec basicdictionary}). 
\end{prop}
\begin{proof}
 We need to show that $\psi\!\uparrow\! f^{\ast}\colon
F_0\!\uparrow\! f^{\ast}\to F_1\!\uparrow\! f^{\ast}$ is a homotopy equivalence for any 
 $\psi\colon F_1\to F_0$   in
$\text{\rm Cof}(A,X_{\text{\rm we}})$. 
Let  $(G,\alpha_1)$ be an object in  $F_1\!\uparrow\! f^{\ast}$\@. It consists of a functor $G\colon B\to X_{\text{we}}$ in 
$\text{\rm Cof}(B,X_{\text{\rm we}})$ and
a natural weak equivalence $\alpha_1\colon F_1\to f^{\ast}G$\@.  Consider the   following commutative diagram in $\text{Fun}^b(B,\calm)$:
%For any such object consider the following commutative diagram in $\text{Fun}^b(B,\calm)$:
\[\xymatrix@R=11pt@C=20pt{
f^{k}F_1 \ar@/^15pt/[rr]|{\widetilde{\alpha_1}}\rmono\dto_(.4){f^k\psi}^(.4){\simeq}  & P(\widetilde{\alpha_1})\ar@{->>}[r]^{\simeq}\dto^(.4){\simeq} & G\\
f^{k}F_0\rmono^(.45){\widetilde{\alpha_0}} & \Phi(G,\alpha_1)
}\]
where  $\widetilde{\alpha_1}\colon f^kF_1\to G$ is the adjoint to $\alpha_1$,
$f^{k}F_1\hookrightarrow P(\widetilde{\alpha_1})\stackrel{\simeq}{\twoheadrightarrow}G$
%$\xymatrix@R=15pt@C=15pt{f^{k}F_1\rmono & P(\widetilde{\alpha_1})\ar@{->>}[r]^{\simeq}& G}$
is the functorial factorization given in~\ref{pt modelcat}, and 
 the left hand square is a {\bf push-out}.
%\item the square:
%\[\xymatrix@R=15pt@C=20pt{
%f^{k}F_1 \rmono\dto_(.4){f^k\psi}^(.4){\simeq}  & P(\widetilde{\alpha_1})\dto^(.4){\simeq} \\
%f^{k}F_0\rmono^(.45){\widetilde{\alpha_0}} & \Phi(G,\alpha_1)
%}\]
%is a push-out square.
%\end{itemize}
Since $\psi$ is a  week equivalence between cofibrant 
objects, its left Kan extension $f^{k}\psi$ is also  a weak equivalence between cofibrant objects (see~\ref{prop basicprophocolim}). Thus 
%the above square being a push-out has two  consequences: 
the natural transformation $P(\widetilde{\alpha_1})\to  \Phi(G,\alpha_1)$ is a weak equivalence
%(again as suggested by the above diagram) 
and $\widetilde{\alpha_0}$ is a cofibration.
The functor $\Phi(G,\alpha_1)$ is therefore cofibrant and has values in $X_{\text{we}}$.
Define $\alpha_0\colon F_0\to f^{\ast}\Phi(G,\alpha_1)$ to be the adjoint to $\widetilde{\alpha_0}$.
In this way out of an object $(G,\alpha_1)$ in  $F_1\!\uparrow\! f^{\ast}$ we have constructed an 
object $(\Phi(G,\alpha_1), \alpha_0)$ in $F_0\!\uparrow\! f^{\ast}$\@. This whole procedure is  functorial. The induced functor is denoted by $\Phi\colon F_1\!\uparrow\! f^{\ast}\to F_0\!\uparrow\! f^{\ast}$\@. We claim that $\Phi$ is a homotopy
inverse to $\psi\!\uparrow\! f^{\ast}\colon F_0\!\uparrow\! f^{\ast}\to F_1\!\uparrow\! f^{\ast}$\@.
The weak equivalences
$\Phi(G,\alpha_1) \xleftarrow{\simeq} P(\widetilde{\alpha_1})\xrightarrow{\simeq} G$
give a ``zig-zag'' of natural transformations between the composition $(\psi\!\uparrow\! f^{\ast})\Phi$ and 
$\text{id}_{F_1\uparrow f^{\ast}}$\@.
Let $(G\colon B\to X_{\text{we}},\lambda\colon F_0\to f^{\ast}G)$ be an object in $F_0\!\uparrow\! f^{\ast}$ and consider the following commutative diagram in $\text{Fun}^b(C,\calm)$:
\[\xymatrix@R=11pt@C=15pt{
f^{k}F_1 \ar@/^15pt/[rr]|{\widetilde{\lambda}f^k\psi}\rmono\dto_-{f^k\psi}^{\simeq}  & P(\widetilde{\lambda}f^k\psi)\ar@{->>}[r]^{\simeq}\dto^-{\simeq} & G\ar@{=}[d]\\
f^{k}F_0\rmono \ar@/_13pt/[rr]|{\widetilde{\lambda}} & \Phi(G,\widetilde{\lambda}f^k\psi)\rto^(.65){\simeq} & G
}\]
where $ \Phi(G,\widetilde{\lambda}f^k\psi)\xrightarrow{\simeq} G$ is given by the universal property
of a push-out.
These  week equivalences
form a natural transformation  between  $\Phi(\psi\!\uparrow\! f^{\ast})$ to $\text{id}_{F_0\uparrow f^{\ast}}$.
\end{proof}

\begin{prop}\label{prop cluthcingcof}
Assume that the  square on the left below is a push-out with $i$ an inclusion and $f$ reduced (see Section~\ref{sec simplex}). Then
$g$ is also reduced and the square on the right is   a strong homotopy pull-back
(see Section~\ref{sec basicdictionary}):
%Let the following be a push-out square of spaces where the indicated maps are inclusions:
\[\xymatrix@R=11pt@C=15pt{
A\rmono^{i} \dto_-{f} & C\dto^-{g}\\
B\rmono^{j} & D
}\ \ \ \ \ \ \ \  \ \ \ \ 
\xymatrix@R=11pt@C=15pt{
\text{\rm Cof}(D,X_{\text{\rm we}})\rto^{j^{\ast}} \dto_-{g^{\ast}} & \text{\rm Cof}(B,X_{\text{\rm we}})\dto^-{f^{\ast}}\\
\text{\rm Cof}(C,X_{\text{\rm we}})\rto^{i^{\ast}} & \text{\rm Cof}(A,X_{\text{\rm we}})
}\]
\end{prop}
\begin{proof}
The proof of the fact that $g$ is reduced is left for the reader.

By definition, the claimed square is a strong homotopy pull-back if
two requirements are satisfied: 
$i^{\ast}$ is a strong fibration and,
for any object $F$ in  $\text{\rm Cof}(B,X_{\text{\rm we}})$, the functor
$(g^{\ast},f^{\ast})\colon F\!\uparrow\!  j^{\ast}\rightarrow  f^{\ast}F\!\uparrow\! i^{\ast}$ is a homotopy
equivalence (see Section~\ref{sec basicdictionary}).
The first requirement is the content of~\ref{prop cofstronfibration}. It remains  to
prove the second one.
A short argument for $(g^{\ast},f^{\ast})$ being a homotopy equivalence is that
the clutching construction is its homotopy inverse.  Here is a more detailed
explanation  of why this is so.
An object  $(G,\psi)$ in $f^{\ast}F\!\uparrow\! i^{\ast}$  consists of a functor $G$ in
$\text{\rm Cof}(C,X_{\text{\rm we}})$ and a natural weak equivalence $\psi\colon f^{\ast}F\to i^{\ast}G$.
 This is an example of a clutching data (see Section~\ref{sec clutching}).
 Its  clutching   is  a functor  $H(\psi, F, G)\colon D\to \calm$ and 
a natural transformation $\overline{\psi}\colon g^{\ast}H(\psi, F, G)\to G$\@.  
The functor $H(\psi, F, G)$ is bounded (see~\ref{prop boundclutch}) and
 $\overline{\psi}$ is a weak equivalence. 
%The properties
%(1) and (3) of the clutching construction (see Section~\ref{sec clutching}) imply 
%$H(\psi, F, G)$ is bounded (see~\ref{prop boundclutch}) and  $\overline{\psi}$ is a weak equivalence. 
However   $H(\psi, F, G)$ may not be  cofibrant  in $\text{Fun}^b(D,\calm)$\@.
%In particular $H(\psi, F, G)$ has values in $X_{\text{\rm we}}$.
%The functor $H(\psi, F, G)$ may not be a cofibrant object in $\text{Fun}^b(D,\calm)$ however.
Let $\alpha\colon j^{k}F\to H(\psi, F, G)$ be the adjoint to the identity functor
$F=j^{\ast}H(\psi, F, G)$\@. Define $\overline{H}(\psi, F, G)$ to be the functor that fits into the following functorial factorization in
$\text{Fun}^b(D,\calm)$ (see~\ref{pt modelcat}):
\[\xymatrix@R=12pt@C=15pt{
j^{k}F \ar@/_12pt/[rr]|{\alpha}\rmono & \overline{H}(\psi, F, G)\ar@{->>}[r]^(.45){\simeq}& H(\psi, F, G)
}\]
By applying $g^{\ast}$ to the acyclic fibration on the right  and composing it with $\overline{\psi}$
we get a natural transformation denoted by the same symbol 
$\overline{\psi}\colon g^{\ast}\overline{H}(\psi, F, G)\to G$:
\[\xymatrix@R=12pt@C=15pt{
g^{\ast}\overline{H}(\psi, F, G)\ar@{->>}[r]^(.45){\simeq}  \ar@/_12pt/[rr]|{\overline{\psi}}& 
g^{\ast}H(\psi, F, G)\rto^-{\overline{\psi}} & G}\]
Let $\overline{\alpha}\colon F\to j^{\ast}\overline{H}(\psi, F, G)$  be the adjoint to
 $j^{k}F\hookrightarrow \overline{H}(\psi, F, G)$\@.
Since $F$ is cofibrant, then so are $j^{k}F$ and  $\overline{H}(\psi, F, G)$\@.
Thus $\overline{H}(\psi, F, G)$ is an object in $\text{\rm Cof}(D,X_{\text{\rm we}})$.
This data can be arranged into a commutative diagram:
\[
\xymatrix@R=13pt@C=13pt{
 & C\drto^-{g}\rrto_(.6){G}="G" & & X_{\text{we}}\\
A\urmono^{i}\drto_{f} & & D\rrto|(.6){\overline{H}}="H" &  &X_{\text{we}}\\
 & B\urmono^-{j}\rrto^(.6){F}="F" & & X_{\text{we}}
 \ar@/_.2pc/ @{->}|-\hole_(.3){\psi} "F"; "G" 
 \ar @/_.5pc/@{->}|(.6){\overline{\alpha}} "F" ; "H"
 \ar @/_.5pc/@{->}|(.4){\overline{\psi}} "H" ; "G"
}
\]
%\[
%\xymatrix@R=20pt@C=20pt{
%A\rmono^{i} \dto_{f} & C\dto^{g}\ar@/^1pc/[rrd]_(.7){G}="G" \\
%B\rmono^{j}\ar@/_1pc/[rdd]^(.65){F}="F" & D\drto|(.3){\overline{H}}="H" & & X_{\text{we}}
%\ar@/_.9pc/ @{->}_(.3){\psi}|(.51)\hole "F"; "G" 
%\ar @/^.5pc/@{->}^{\overline{\psi}} "H" ; "G"
%\ar @/^.5pc/@{->}^{\overline{\alpha}} "F" ; "H"
%\\
%&  &  X_{\text{we}}\\
%& X_{\text{we}}
%}
%\]
Out of an object $(G,\psi)$ in $f^{\ast}F\!\uparrow\! i^{\ast}$,
we have constructed an object $(\overline{H}(\psi, F, G), \overline{\alpha})$
in $F\!\uparrow\!  j^{\ast}$. All the steps in this construction are functorial. The obtained functor is denoted by
$\Phi\colon f^{\ast}F\!\uparrow\! i^{\ast}\to F\!\uparrow\!  j^{\ast}$.
We claim that $(g^{\ast},f^{\ast})\Phi$ and $\Phi (g^{\ast},f^{\ast})$ are homotopic to
the identity functors. The morphisms $\overline{\psi}\colon g^{\ast}\overline{H}(\Psi,F,G)\to G$  in $f^{\ast}F\!\uparrow\! i^{\ast}$ form a natural transformation between 
$(g^{\ast},f^{\ast})\Phi$ and $\text{id}_{f^{\ast}F\uparrow i^{\ast}}$.

Let $(G, \psi)$ be an object in $F\!\uparrow\!  j^{\ast}$ consisting of a functor $G$ in
$\text{Cof}(D, X_{\text{we}})$ and a  weak equivalence $\psi\colon F\to j^{\ast}G$.
The composition of $ \widehat{\psi}\colon H(f^{\ast}\psi, F, g^{\ast} G)\to  G$ given in~\ref{prop univclutching}
with the cofibrant replacement $\overline{H}(f^{\ast}\psi, F, g^{\ast} G)\to  H(f^{\ast}\psi, F, g^{\ast} G)$ is   a natural transformation between  $\Phi (g^{\ast},f^{\ast})$ and 
$\text{id}_{F\uparrow j^{\ast}}$.
\end{proof}

%%%%%%%%%%%%%%%%%%%%%%%%%%%%%%%%%%%%%%%
%%%%%%%%%%%%%% proof  %%%%%%%%%%%%%%%%%%%
%%%%%%%%%%%%%%%%%%%%%%%%%%%%%%%%%%%%%%%
\section{The category of functors with values in $X_{\text{we}}$}
\begin{thm}\label{thm funintoxwe}
Let %$X$ be an object in  $\calm$ and 
 $I$ be  a small category. 
Then\/  $\text{\rm Fun}(I,X_{\text{\rm we}})$ is essentially small and
the nerve of its core is weakly equivalent to\/ $\text{\rm map}(N(I),B\text{\rm we}(X,X))$.
%\begin{enumerate}
%\item Let $I$ be a small category. Then  $\text{\rm Fun}(I,X_{\text{\rm we}})$ is essentially small and
%the nerve of its core is weakly equivalent to $\text{\rm map}(N(I),B\text{\rm we}(X,X))$.
%\item Let $f:I\ra J$ be a weak equivalence of small categories. Then the functor 
%$f^{\ast}:\text{\rm Fun}(J,X_{\text{\rm we}})\ra \text{\rm Fun}(I,X_{\text{\rm we}})$ is a weak
%equivalence.
%\end{enumerate}
\end{thm}

Our strategy to prove~\ref{thm funintoxwe} is  to show that the nerve of a 
core of  $\text{\rm Fun}(I,X_{\text{\rm we}})$  is weakly equivalent to the 
homotopy limit of the constant functor indexed by $I$ with value $B\text{\rm we}(X,X)$\@. 
%induction on the dimension of $N(I)$.
%We show that
%Note that this is another
%model for the mapping space $\text{\rm map}(N(I),B\text{\rm we}(X,X))$. 
For this strategy to work we need to choose these cores
% in $\text{\rm Fun}(I,X_{\text{\rm we}})$ 
in a certain functorial way 
with respect to $I$\@. We set this functoriality first.
%For that we need a weak equivalence 
%between the nerve of a core of $\text{\rm Fun}(I,X_{\text{\rm we}})$ and 
%$\text{\rm map}(N(I),B\text{\rm we}(X,X))$ which is functorial with respect to $I$. We  set up 
%this functoriality first.   

Let   $\phi\colon I\to \text{Spaces}$ be a functor.
For any morphism $\lambda\colon i\to j$ in $I$ and any commutative diagram on the left  below we have the following  commutative diagram of  functors on the right:
\[\xymatrix@R=11pt@C=17pt{\Delta[n]\rto^{\alpha}\dto_{\sigma} & \Delta[m]\dto^{\tau}\\
\phi(i)\rto^{\phi(\lambda)}\dto_{p} & \phi(j)\dto^{p}\\
\Delta[0]\ar@{=}[r] & \Delta[0]}\ \ \ \ \ \ \ \ \ \ \  \ \ \ \ \ \ \  \xymatrix@R=11pt@C=17pt{
\text{Fun}(\Delta[n], X_{\text{we}}) &  \text{Fun}(\Delta[m], X_{\text{we}})\lto_{\alpha^{\ast}} \\
\text{Fun}(\phi(i), X_{\text{we}})\uto^{\sigma^{\ast}}    &\text{Fun}(\phi(j), X_{\text{we}})\lto_{\phi(\lambda)^{\ast}}\uto_{\tau^{\ast}} \\
\text{Fun}(\Delta[0], X_{\text{we}})\uto^{p^{\ast}}\ar@{=}[r] & \text{Fun}(\Delta[0], X_{\text{we}})\uto_{p^{\ast}}
}\]
The symbol $\calf(\phi,X)$  denotes the system of categories given by all the functors in the right 
diagram above for all the morphisms $\lambda$ in $I$\@. To describe this system we use the following 
notation. Let  $i$ be an object in $I$ and  $\sigma\colon\Delta[n]\to \phi(i)$ a simplex:
\[\xymatrix@R=6pt@C=15pt{
\calf(\phi,X)_{e_{\phi(i)}}\ar@{*{\cdot\cdot}=}[d]\rrto^{\calf(\phi,X)_p}
 \ar@/^22pt/[rrrr]|{\calf(\phi,X)_p} & &\calf(\phi,X)_{t_{\phi(i)}} \ar@{*{\cdot\cdot}=}[d] 
\rrto^{\calf(\phi,X)_{\pi}}& &\calf(\phi,X)_{\sigma}\ar@{*{\cdot\cdot}=}[d]
\\ 
\text{Fun}(\Delta[0],X_{\text{we}}) \rrto^{p^{\ast}}
 \ar@/_15pt/[rrrr]|{p^{\ast}} & & \text{Fun}(\phi(i),X_{\text{we}})\rrto^{\sigma^{\ast}} &&
\text{Fun}(\Delta[n],X_{\text{we}})
}\]
%where $p$'s are the unique maps. 
For example let $A$ be a space and $A\colon [0]\to \text{Spaces}$ be the constant functor with value 
$A$\@.
The corresponding system
of categories  $\calf(A,X)$ is given by the following functors indexed by
morphism $\alpha\colon\sigma\to \tau$ in $A$:
\[\xymatrix@R=40pt@C=18pt{
\calf(A,X)_{e_A}\rto^(.55){\calf_{p}}\ar@{=}[d]\ar@/^17pt/[rr]|{\calf_p} \ar@/^30pt/[rrr]|(.52){\calf_p}&  \calf(A,X)_{t_A}\rto^(.45){\calf_{\pi}}  \ar@/_15pt/[rr]|(.6){\calf_{\pi}} \ar@{=}[d]&   \calf(A,X)_{\tau}\rto^{\calf_{\alpha}}\ar@{=}[d]|(.265)\hole
&  \calf(A,X)_{\sigma}\ar@{=}[d]\\
\text{Fun}(\Delta[0],X_{\text{we}})\rto^(.58){p^{\ast}} 
\ar@/^17pt/[rr] |(.49)\hole ^(.68){p^{\ast}} 
\ar@/^30pt/[rrr]|(.32)\hole |(.52){p^{\ast}} |(.65)\hole&  \text{Fun}(A,X_{\text{we}})\rto^{\tau^{\ast}}  \ar@/_15pt/[rr]|{\sigma^{\ast}}& \text{Fun}(\Delta[m],X_{\text{we}})\rto^{\alpha^{\ast}}
&  \text{Fun}(\Delta[n],X_{\text{we}})  }\]

\begin{prop}\label{prop keypropessentiallysmall}
%For any simplicial set   $A$ and any object $X$  in a model category $\calm$, 
The system $\calf(A,X)$ is essentially small 
and,  for any of its cores
$F\subset \calf(A,X)$, %(see~\ref{def essentially small system}), 
 the maps
$N(F_{t_A})\to \text{\rm holim}_{\sigma\in A}N(F_{\sigma})\gets \text{\rm holim}_{\sigma\in A} N(F_{e_A})$,
induced by $F_{\pi}\colon F_{t_A}\to F_{\sigma}$  and  $F_{\sigma}\gets F_{e_A} : F_p $, are  weak equivalences.
\end{prop}

\begin{proof}
If $K$ is a contractible space, then $\text{Fun}(K,X_{\text{we}})$ is essentially small (see~\ref{cor funoutofcotract}). Thus to prove that $\calf(A;X)$ is an essentially small system, we need to show that $\text{Fun}(A,X_{\text{we}})$ is an essentially small category (see~\ref{prop funcorialcore}.(1)).

Let $\calt$ to be the collection of all  spaces  for which the proposition is true. 
To show that any space belongs  to $\calt$, we  prove that $\calt$ satisfies the following properties:
\begin{enumerate}
\item If in the following push square $A$, $B$, and $C$ belong to $\calt$, then so does $D$. 
\[\xymatrix@R=11pt@C=15pt{
A\rmono^{i} \dto_{f} & C\dto^{g}\\
B\rmono^{j} & D
}\]
%If $A$, $B$, and $C$ belong to $\calt$, then so does $D$.
\item Let $f\colon A\to B$ be a map  such that
$df(\sigma)=
\text{\rm lim}(\Delta[n]\xrightarrow{\sigma} B\xleftarrow{f} A)$
is contractible for any $\sigma$ Êin $B$. Then Ê$A$ belongs to $\calt$ if and only if $B$ does.
\item  If $A$ is contractible, then $A$ belongs to $\calt$.
\item Let $S$ be a set. If $A_s\in \calt$  for any $s\in S$,  then    $\coprod_{s\in S}A_s\in\calt$.
\item Let $A_0\subset A_1\subset \cdots \subset A$ be a filtration of $A$ such that $A=\cup_{i\geq 0} A_i$.
If $A_{i}\in\calt$  for any $i$, then to $A\in \calt$. 
\end{enumerate}
In each of the following steps we prove the corresponding  property. 
\smallskip

\noindent 
{\bf Step (1).} 
By~\ref{cor redtonerve} and~\ref{prop cofinallbund} the  functors 
$ \text{Cof}(N(I),X_{\text{we}})\subset \text{Fun}^b(N(I),X_{\text{we}})$ and
$\epsilon^{\ast}\colon\text{Fun}(I,X_{\text{we}})\to  \text{Fun}^b(N(I),X_{\text{we}})$
are homotopy equivalences.
% for any small category $I$:
%\[\text{Cof}(N(I),X_{\text{we}})\subset \text{Fun}^b(N(I),X_{\text{we}})\xleftarrow{\epsilon^{\ast}} \text{Fun}(I,X_{\text{we}})\]
Thus if one of those categories is essentially small, then so is 
any other (see~\ref{cor esssmallhomeqinv})  in which case these functors are weak equivalences (see~\ref{prop esssmallweakequiv}).

Since the following square on the left is a push-out with all the maps being reduced, according to~\ref{prop cluthcingcof} we have a strong homotopy pull-back on the right:
\[\xymatrix@R=11pt@C=25pt{
N(A)\rmono^{N(i)} \dto_{N(f)} & N(C)\dto^{N(g)}\\
N(B)\rmono^{N(j)} & N(D)
}\ \ \ \ \ \ \ \ 
\xymatrix@R=11pt@C=25pt{
\text{\rm Cof}(N(D),X_{\text{\rm we}})\rto^{N(j)^{\ast}} \dto_{N(g)^{\ast}} & \text{\rm Cof}(N(B),X_{\text{\rm we}})\dto^{N(f)^{\ast}}\\
\text{\rm Cof}(N(C),X_{\text{\rm we}})\rto^{N(i)^{\ast}} & \text{\rm Cof}(N(A),X_{\text{\rm we}})
}
\]
By the assumption 
$\text{Fun}(A,X_{\text{\rm we}})$, $\text{Fun}(B,X_{\text{\rm we}})$, and $\text{Fun}(C,X_{\text{\rm we}})$
are essentially small  and hence so are 
$ \text{\rm Cof}(N(A),X_{\text{\rm we}})$, $\text{\rm Cof}(N(B),X_{\text{\rm we}})$, and $\text{\rm Cof}(N(C),X_{\text{\rm we}})$. We can  use~\ref{cor stronhpullhompull} to conclude that
$\text{\rm Cof}(N(D),X_{\text{\rm we}})$ is also essentially small. The   right square above is thus  a homotopy 
pull-back and consequently  so is: % proves that the following is also a homotopy pull-back of essentially small categories:
\[\xymatrix@R=11pt@C=15pt{
\text{Fun}(D,X_{\text{\rm we}})\rto^{j^{\ast}} \dto_{g^{\ast}} &\text{Fun}(B,X_{\text{\rm we}})\dto^{f^{\ast}}\\
\text{Fun}(C,X_{\text{\rm we}})\rto^{i^{\ast}} & \text{Fun}(A,X_{\text{\rm we}})
}\]

Let $\phi$ be the  functor indexed by  the poset category of all subsets of $\{0,1\}$
given by the  commutative square  in  the statement  (1).
%\[\xymatrix@R=12pt@C=15pt{A\rmono^{i}\dto_{f} & C\dto^{g}\\
%B\rmono^{j} & D}
%\]
The system  $\calf(\phi,X_{\text{we}})$ is essentially small as its values are so
(see~\ref{prop funcorialcore}.(1)).
Let $F$ be its core.
If  $K$ is among $\{A,B,C\}$, then the maps
$N(F_{t_K})\to \text{holim}_{\sigma\in K} N(F_{\sigma})\gets \text{holim}_{\sigma\in K} N(F_{e_K})$
are weak equivalences. These maps fit into a commutative diagram:
\[\xymatrix@C=-14pt@R=10pt{
N(F_{t_D})\ddto\drto \rrto& &  \text{holim}_{\sigma\in D} N(F_{\sigma})\ddto|\hole\drto
& &  \text{holim}_{\sigma\in D} N(F_{e_D})\llto\ddto|\hole\drto\\
& N(F_{t_B})\ddto\rrto & &  \text{holim}_{\sigma\in B} N(F_{\sigma})\ddto 
& &  \text{holim}_{\sigma\in B} N(F_{e_B})\llto\ddto\\
N(F_{t_C})\drto\rrto|(.38)\hole & &  \text{holim}_{\sigma\in C} N(F_{\sigma})\drto 
& &  \text{holim}_{\sigma\in C} N(F_{e_C})\llto|(.51)\hole\drto\\
&  N(F_{t_A})\rrto & &  \text{holim}_{\sigma\in A} N(F_{\sigma}) & &  \text{holim}_{\sigma\in A} N(F_{e_A})\llto
}\]
The right side, the left, and the middle  squares are  homotopy pull-backs. The first one is because   the functors involved are constant, the second because of the discussion before, and the middle because 
the right horizontal maps are weak equivalences.  By the  inductive assumption
the horizontal left bottom and  front  maps are  weak equivalences. It follows that so is the fourth one proving step 1.
\smallskip

\noindent 
{\bf Step (2).}  Let  $\phi\colon [1]\to \text{Spaces}$ be a functor given by the map $f\colon A\to B$. Consider the system $\calf(\phi,X)$. By~\ref{cor cofinalitycons}.(1),
 $f^{\ast}\colon\text{\rm Fun}(B,X_{\text{\rm we}})\to \text{\rm Fun}(A,X_{\text{\rm we}})$ is a homotopy
equivalence. Thus if one of these categories is essentially small, then   $\calf(\phi,X)$ is essentially  small and   it has a core $F$.
Consider the following commutative diagram induced by the functors in such a core:
\[\xymatrix@R=11pt@C=15pt{
N(F_{t_B})\rto\dto & \text{holim}_{\sigma\in B} N(F_{\sigma}) \dto& \text{holim}_{\sigma\in B} N(F_{e_B})\lto\dto\\
N(F_{t_A})\rto & \text{holim}_{\sigma\in A} N(F_{\sigma}) & \text{holim}_{\sigma\in A} N(F_{e_A})\lto
}\]
The right vertical map is a weak equivalence since the functors involved are constant. 
The right horizontal maps are weak equivalences as the  homotopy limit preserves  weak equivalences.  
It follows that so is the middle vertical map. The left vertical map is a weak equivalence by~\ref{cor cofinalitycons}.(1).
 We can then conclude that
if one of the left horizontal  maps is a weak equivalence then so is the other. This means that
$A$ belongs to $\calt$ if and only if $B$ does which proves step 2.
\smallskip

\noindent 
{\bf Step (3).} 
According to~\ref{cor funoutofcotract} the category $\text{Fun}(A,X_{\text{we}})$ is essentially small
and consequently $\calf(A,X)$ is an essentially small system (see~\ref{prop funcorialcore}.(1)). 
Let $F\subset \calf(A,X)$ be its core and consider  the following commutative diagram:
\[\xymatrix@R=11pt@C=15pt{
N(F_{e_A})\rto^(.32){a} \dto_{N(F_p)} & \text{holim}_{\sigma\in A} N(F_{e_A})\dto\\
N(F_{t_A})\rto &  \text{holim}_{\sigma\in A} N(F_\sigma)
}\]
where $a$ is the diagonal  map from $N(F_{e_A})$ to the homotopy limit of the constant functor
$N(F_{e_A})$ and  the rest of the maps are induced by functors in the system $F$\@. Since $A$ is contractible, 
the vertical maps and $a$ are weak equivalences (see~\ref{cor funoutofcotract}). 
%Contractibility of $A$ implies also that $a$ is a weak equivalence. We can conclude that the bottom map is  a weak equivalence and consequently $A$ belongs to $\calc$.
Thus so is the bottom map and  $A$ belongs to $\calt$.
\smallskip

\noindent 
{\bf Step (4).}  Consequence of the fact that  products preserve weak equivalences.
\smallskip

\noindent
{\bf Step (5).}  
The map $f\colon\text{Tel}\to A$ (see~\ref{cor cofinalitycons}.(3)) satisfies the requirements in step (2).
Thus  we need to prove that $\text{Tel}$ belongs to $\calt$\@. Again by step (2) the products
$A_i\times\Delta[1]$ belong to $\calt$ and so by step (4)  the components of the push-out describing $\text{Tel}$ also belong to $\calt$\@. We can now use step (1).
\end{proof}

\begin{proof}[Proof of Theorem~\ref{thm funintoxwe}]
Apply~\ref{prop keypropessentiallysmall} to  $\calf(N(I), X)$
to see that the nerve of a core of  $\calf(N(I), X)_{t_{N(I)}}=\text{\rm Fun}(N(I),X_{\text{\rm we}})$
is weakly equivalent to the homotopy limit of the constant functor $\text{holim}_{\sigma\in N(I)}N(F_{e_{N(I)}})
\simeq \text{map}(N(I),N(F_{e_{N(I)}}))$.
  By~\ref{cor funoutofcotract} the nerve of $F_{e_{N(I)}}$, which
is a core of $\calf(N(I),X)_{e_{N(I)}}=\text{Fun}(\Delta[0],X_{\text{we}})$, is weakly equivalent to 
 $B\text{we}(X,X)$.  To finish the proof recall that 
according to~\ref{cor redtonerve}, the functor $\epsilon^{\ast}\colon\text{\rm Fun}(I,X_{\text{\rm we}})\to \text{\rm Fun}(N(I),X_{\text{\rm we}})$
is a homotopy equivalence.
%
%
%To prove statement (2) choose a weak equivalence  $f:I\ra J$. Let  $\phi$ be the functor indexed by the category $[1]$ with values in $\text{Spaces}$
%given by the map $N(f):N(I)\ra N(J)$. Consider the system $\calf(\phi,X)$. Let  $F$  be its core. 
%Form the following commutative diagram induced by functors in $F$:
%\[\xymatrix@R=15pt@C=15pt{
%N(F_{t_{N(J)}})\rto\dto & \text{holim}_{\sigma\in N(J)} N(F_{\sigma}) \dto& \text{holim}_{\sigma\in N(J)} N(F_{e_N(J)})\lto\dto\\
%N(F_{t_{N(I)}})\rto & \text{holim}_{\sigma\in N(I)} N(F_{\sigma}) & \text{holim}_{\sigma\in N(I)} N(F_{e_{N(I)}})\lto
%}\]
%All the horizontal maps are weak equivalences by Proposition~\ref{prop keypropessentiallysmall}.
%The right vertical map is a weak equivalence since the functors involved are constant and $f$ is a weak equivalence. The left vertical map is therefore also a weak equivalence. This shows that 
%$N(f)^{\ast}:\text{\rm Fun}(N(J),X_{\text{\rm we}})\ra \text{\rm Fun}(N(I),X_{\text{\rm we}})$ is a weak
%equivalence. Same is therefore true for $f^{\ast}:\text{\rm Fun}(J,X_{\text{\rm we}})\ra \text{\rm Fun}(I,X_{\text{\rm we}})$.
\end{proof}

As a direct consequence of~\ref{prop keypropessentiallysmall} we also get:
\begin{cor}
If $f\colon I\to J$ be a weak equivalence of small categories, then 
$f^{\ast}\colon\text{\rm Fun}(J,X_{\text{\rm we}})\to \text{\rm Fun}(I,X_{\text{\rm we}})$ is a weak
equivalence.
\end{cor}

%%%%%%%%%%%%%%%%%%%%%%%%%%%%%%%%%%%%%%%%%%
%%%%%%%%%%%%%%%%%%%%%%%%%%%%%%%%%%%%%%%%%%
%%%%%%%%%%%%%%%%%%%%%%%%%%%%%%%%%%%%%%%%%%
%%%%%%%%%%%%%% Theorem A %%%%%%%%%%%%%%%%%
%%%%%%%%%%%%%%%%%%%%%%%%%%%%%%%%%%%%%%%%%%
%%%%%%%%%%%%%%%%%%%%%%%%%%%%%%%%%%%%%%%%%%
%%%%%%%%%%%%%%%%%%%%%%%%%%%%%%%%%%%%%%%%%%
\section{Theorem A}
The aim of this final section is to prove Theorem A from the introduction. 
%Let us fix two spaces $X$ and $F$ and define: 
\begin{Def} Let $X$ and $F$ be spaces.
 $\text{Ext}(X,F)$ is a subcategory of $\text{\rm Fib}(X,F)$ whose objects  are maps 
$f\colon E\to X$ with a fixed codomain $X$\@.
The set of  morphisms  in $\text{Ext}(X,F)$ between $f\colon E\to X$ and $f'\colon E'\to X$ consists of  weak equivalences $g\colon E\to E'$ for which $f'g=f$.
%\begin{enumerate}
%\item ${\mathcal Fib}(X,F)$ is a full subcategory of $\text{\rm Fib}(X,F)$ whose objects are    maps
%$f\colon A\to B$ which are {\bf fibrations}.
%  (and not just the maps whose homotopy fibers are weakly equivalent to $F$).
%\item $\text{Ext}(X,F)$ is a subcategory of $\text{\rm Fib}(X,F)$ whose objects  are maps 
%$f\colon E\to X$ with a fixed codomain $X$\@.
%The set of  morphisms  in $\text{Ext}(X,F)$ between $f\colon E\to X$ and $f'\colon E'\to X$ consists of  weak equivalences $g\colon E\to E'$ for which $f'g=f$.
%\item ${\mathcal Ext}(X,F)$ is a full subcategory of $\text{Ext}(X,F)$ whose objects are these maps 
%$f\colon E\to X$ with codomain $X$ which are {\bf fibrations}.
%\end{enumerate}
\end{Def}
%All the categories defined above are called fibration categories and they fit into:
%\[\xymatrix@C=11pt@R=15pt{
%{\mathcal Ext}(X,F) \rmono\dmono & \text{\rm Ext}(X,F)\dmono\\
%{\mathcal Fib}(X,F) \rmono & \text{\rm Fib}(X,F)
%}\]

To prove Theorem A we are going to use the following constructions.
%\begin{prop}\label{prop inclusionfiball}
% ${\mathcal Ext}(X,F)\subset \text{\rm Ext}(X,F)$ and 
%${\mathcal Fib}(X,F)\subset \text{\rm Fib}(X,F)$ are homotopy equivalences.
%\end{prop}
%
%To prove this proposition we  are going to use:
\begin{point}{\bf Fibrant replacement.}\label{pt fibreplac}
In addition to the factorization given in~\ref{pt modelcat}, the category $\text{\rm Spaces}$ has  the following  factorization: any commutative square on the left below can be extended {\bf functorially } to a commutative diagram on the right with the indicated morphisms being acyclic cofibrations and fibrations:
\[\xymatrix@R=13pt@C=15pt{
X\rto^{f}\dto_{\alpha_1} &Y\dto^{\alpha_2}\\
X'\rto^{f'} & Y'
} \ \ \ \ \  \ \ \ \ \ \ \ \
\xymatrix@R=13pt@C=22pt{
X\dto_{\alpha_1}\ar@/^18pt/[rr]|{f}\rmono_(.45){\simeq}^{\gamma_f}  & R(f)\ar@{->>}[r]^{\mu_f} \dto|(.45){R(\alpha_1,\alpha_2)}& Y\dto^{\alpha_2}\\
X'\ar@/_15pt/[rr]|{f'}\rmono_(.45){\simeq}^{\gamma_{f'}}  &R(f') \ar@{->>}[r]^{\mu_{f'}} &Y'
} \]
%Functoriality means that  $f\mapsto R(f)$ and  $(\alpha_1,\alpha_2)\mapsto R(\alpha_1,\alpha_2):P(f)\ra R(f')$  is a functor out of the arrow category of  $\text{\rm Spaces}$ to $\text{\rm Spaces}$. 
The commutativity of the right diagram  means that  the morphisms
$\gamma_f\colon X\to R(f)$ and $\mu_f\colon R(f)\to Y$  form natural transformations.
%\[\xymatrix{X\rmono^(.45){\simeq} & R(f)}\ \ \ \ \ \ \xymatrix{R(f) \ar@{->>}[r] &Y} \] 
%The restriction of $R$ to the full subcategory of arrows in $\text{Spaces}$ of the form
%$X\to \Delta[0] $ is denoted by the same symbol $R$\@. This restriction is  a functorial fibrant replacement. 
%$\gamma_X:X\to R(X)$.   
Define $\mu\colon \text{\rm Ext}(X,F)\to \text{\rm Ext}(X,F)$ to be the functor   assigning to an object
 $f\colon E\to X$ the fibration $\mu_f\colon R(f)\to X$ and to a morphism $g\colon E\to E'$, between $f\colon E\to X$ and
 $f'\colon E'\to X$ the map $R(f',\text{id}_X)\colon R(f)\to R(f')$.
 \end{point}

%In the proofs of the classification statements for fibrations of  spaces (see Theorem A from the 
%introduction and~\ref{thmA Ext is Map}), in addition to~\ref{pt fibreplac}, we 
%will use:

\begin{point}{\bf Decomposition.}\label{pt decompositon} Let  $f\colon E\to X$ be a map.  
Define $df\colon X\to \text{Spaces}$ to be the functor that assigns to 
a morphism  $\alpha\colon \Delta[n]\to \Delta[m]$ in $X$ between $\sigma\colon\Delta[n]\to X$ and
 $\tau\colon\Delta[m]\to X$
the map: 
\[\xymatrix@C=15pt@R=11pt{
df(\sigma)\dto_{df(\alpha)}\ar @{}[r]|(0.60){:=} & *{\text{\rm lim}\hspace{1mm}\big(\hspace{-18pt}} &
\Delta[n]\dto_{\alpha}\rto^{\sigma} & X\ar@{=}[d] & E\ar@{=}[d]\lto_{f} & *{\hspace{-18pt}\big)}\\
df(\tau)\ar @{}[r]|(0.60){:=} & *{\text{\rm lim}\hspace{1mm}\big(\hspace{-18pt}} &
\Delta[m]\rto^{\tau} & X & E\lto_{f} & *{\hspace{-18pt}\big)}
}\]
 $d(\text{id}_X)\colon X\to \text{\rm Spaces}$ is  also denoted by $\Delta_X$\@. Its value on
 $\sigma\colon\Delta[n]\to X$ is  the space $\Delta[n]$.  For any such $\sigma$, define
 $df_{\sigma}\colon df(\sigma)\to \Delta_X(\sigma)$  to be the map that fits into   the pull-back square on the left below. These maps form a natural transformation that we denote by  $df\colon df\to \Delta_X$.
The horizontal maps in  this square satisfy the universal property of the colimit inducing 
{\bf canonical} isomorphisms on the right:
\[\xymatrix@C=15pt@R=11pt{
df(\sigma)\rrto\dto_{df_{\sigma}} & & E\dto^{f}\\
\Delta_X(\sigma)\ar@{=}[r] & \Delta[n]\rto^{\sigma} & X
} \ \ \ \ \ \ \ \ 
\xymatrix@C=15pt@R=11pt{\text{colim}_X df \rto^-{\cong}\dto_{\text{colim}_{\sigma\in X}df_{\sigma} }&  E\dto^{f}\\
\text{colim}_X \Delta_X\rto^-{\cong} & X
}
\]
% These maps form a natural transformation that we denote by  $df\colon df\to \Delta_X$.
%The horizontal maps in the above square satisfy the universal properties of colimits inducing 
%{\bf canonical} isomorphisms:
%\[\xymatrix@C=15pt@R=11pt{\text{colim}_X df \rto^-{\cong}\dto_{\text{colim}_{\sigma\in X}df_{\sigma} }&  E\dto^{f}\\
%\text{colim}_X \Delta_X\rto^-{\cong} & X
%}\]
For any morphism $\psi\colon f\to f'$ in $\text{Spaces}\!\downarrow\! X$,  define  $d\psi\colon df\to df'$ as:
\[\xymatrix@C=15pt@R=11pt{
df(\sigma)\dto_{d\psi_\sigma}\ar @{}[r]|(0.60){:=} & *{\text{\rm lim}\hspace{1mm}\big(\hspace{-18pt}} &
\Delta[n]\ar@{=}[d]\rto^{\sigma} & X\ar@{=}[d] & E\dto^{\psi}\lto_{f} & *{\hspace{-18pt}\big)}\\
df'(\sigma)\ar @{}[r]|(0.60){:=} & *{\text{\rm lim}\hspace{1mm}\big(\hspace{-18pt}} &
\Delta[n]\rto^{\tau} & X & E'\lto_{f'} & *{\hspace{-18pt}\big)}
}\]
In  this way we obtain a functor $d\colon\text{Spaces}\!\downarrow\! X\ra \text{Fun}(X,\text{Spaces})\!\downarrow\! \Delta_X$
which we call the  {\bf decomposition}.    Its composition with the fibrant replacement is called 
 the {\bf derived decomposition} and is denoted by $\partial$:
\[\xymatrix{ \text{\rm Ext}(X,F)\rto^{\mu}\ar@/_13pt/[rr]|{\partial}& \text{\rm Ext}(X,F)\rto^(.4)d &\text{\rm Fun}(X,F_{\text{\rm we}})\!\downarrow\! \Delta_X} \]
\end{point}

\begin{point}{\bf Assembly.} Let   
$\text{colim}\colon\text{Fun}(X,\text{Spaces})\!\downarrow\! \Delta_X\to \text{Spaces}\!\downarrow\! X$ be the   functor whose value 
on  an object $g\colon G\to \Delta_X$ is the composition of the colimit
$\text{colim}_X g$ with  the canonical isomorphism
$ \text{colim}_X \Delta_X\cong X$. It maps a morphism $\psi\colon g\to g'$  to $\text{colim}_X \psi$\@.
Let us choose a cofibrant replacement $P$  in $\text{Fun}(X,\text{Spaces})$\@.
% (see after~\ref{prop modelapprox}). 
%Note that $P$ takes the subcategory $\text{\rm Fun}(X,F_{\text{\rm we}})$ into itself.
For an object $g\colon G\to \Delta_X$  in $\text{\rm Fun}(X,F)\!\downarrow\! \Delta_X$, take
the cofibrant replacement $PG\to G$ and define $\int g$ to be the  composition:
\[
\xymatrix@C=15pt@R=11pt{\text{hocolim}_XG\ar@{=}[r]\ar@/_16pt/[rrrrr]|{\int g} &\text{colim}_X PG\rto & \text{colim}_X G\rrto^-{\text{colim}_Xg} & &
\text{colim}_X \Delta_X\rto^-{\cong}& X
}\]
Note that if $g$ belongs to $\text{\rm Fun}(X,F_{\text{\rm we}})\!\downarrow\! \Delta_X$, then
by~\ref{prop ThomasonPuppe},
the homotopy fiber of $\int g$ is weakly equivalent to $F$ and hence  $\int g$ is an object in 
 $ \text{\rm Ext}(X,F)$\@. We call this composition of $P$ and  the colimit the {\bf assembly} functor  and  denote it by
$\int\colon\text{\rm Fun}(X,F_{\text{\rm we}})\!\downarrow\! \Delta_X\to \text{\rm Ext}(X,F)$.
%and the following  induced commutative diagram:
%\[\xymatrix@C=18pt@R=18pt{
%\text{hocolim}_XG\ar@{=}[r] \dto|{\text{hocolim}_Xg}& \text{colim}_XPG\dto|{\text{colim}_XPg}\rto & \text{colim}_X G\dto|{\text{colim}_Xg}\\
%\text{hocolim}_X\Delta_X\ar@{=}[r] &  \text{colim}_XP\Delta_X \rto&  \text{colim}_X\Delta_X\rto^(.6){\cong} & X
%}\]
%The homotopy fibers of $\text{colim}_XPg$ are weakly equivalent to $F$. Since 
%$\text{colim}_XP\Delta_X\ra  \text{colim}_X\Delta_X$
%is a weak equivalence,  the following composition is an object in $ \text{\rm Ext}(X,F)$ which is denoted by $\int g$:
%\[
%\xymatrix@C=15pt@R=15pt{\text{hocolim}_XG\ar@{=}[r]\ar@/_18pt/[rrrrr]|{\int g} &\text{colim}_X PG\rto & \text{colim}_X G\rrto^-{\text{colim}_Xg} & &
%\text{colim}_X \Delta_X\rto^-{\cong}& X
%}
%\text{hocolim}_XG=\text{colim}_X PG\ra \text{colim}_X G\xrightarrow{\text{colim}_Xg}\text{colim}_X \Delta_X\cong X
%\]
\end{point}

\begin{prop}\label{prop Decompint}
$\partial\colon\text{\rm Ext}(X,F)\to \text{\rm Fun}(X,F_{\text{\rm we}})\!\downarrow\! \Delta_X $ is a homotopy equivalence.
\end{prop}
\begin{proof}
We are going to show that  the  assembly functor  is a homotopy inverse to $\partial$.
%\begin{prop}\label{prop Decompint}
%$\int \partial \colon\text{\rm Ext}(X,F)\to \text{\rm Ext}(X,F)$ and $ \partial\int \colon \text{\rm Fun}(X,F_{\text{\rm we}})\!\downarrow\! \Delta_X\to \text{\rm Fun}(X,F_{\text{\rm we}})\!\downarrow\! \Delta_X$
%are homotopic to the identity functors.
%\end{prop}
%\begin{proof}
For  an object $f\colon E\to X$ in   $\text{\rm Ext}(X,F)$ form a commutative diagram:
\[\xymatrix@C=15pt@R=11pt{
\text{colim}_XP\partial f\rto \dto \drrto|{\int\partial f}
&  \text{colim}_X\partial f = \text{colim}_Xd\mu_f\rto^(.7){ \cong} & R(f)\ar@{->>}[d]_{\mu_f} & 
E\lmono_(.45){\gamma_f}\ar@/^5pt/[dl]|{f}\\
\text{colim}_XP\Delta_X\rto & \text{colim}_X\Delta_X\rto_-{ \cong} &  X
}\]
The top horizontal maps form a ``zig-zag'' connecting  $\int \partial$ with the identity functor.
%$\text{id}:\text{\rm Ext}(X,F)\ra \text{\rm Ext}(X,F)$ and $\int \partial :\text{\rm Ext}(X,F)\ra \text{\rm Ext}(X,F)$.

Similarly, for an  object $g\colon G\to \Delta_X$  in $\text{\rm Fun}(X,F_{\text{\rm we}})\!\downarrow\! \Delta_X$ and a simplex  $\sigma\colon \Delta[n]\to X$, consider the following commutative diagram
where the  maps $G(\sigma)\to  \text{colim}_X G$ and $PG(\sigma)\to  \text{colim}_X PG$ 
are induced by the inclusion of $\sigma$ into $X$,
and $PG(\sigma)\ra (d\mu_{\int g})(\sigma)$ is induced the universal property of a
pull-back. 

\[\xymatrix@C=15pt@R=10pt{
 & & G(\sigma)\rrto \ar@/^7pt/[dddl]|(.29)\hole_{g_\sigma}|(.62)\hole& & \text{colim}_X G\ddto|{\text{colim}_Xg}\\
 & PG(\sigma)\rrto\urto \dlto\ddto|\hole& & \text{colim}_XPG\ddto^{\int g}\urto\ar@{_(->}[dl]_{\simeq}\\
(d\mu_{\int g})(\sigma)\rrto\ar@/_5pt/[dr]|{(\partial\int)_{\sigma}} & & R(\int g)\ar@/_5pt/@{->>}[dr]|{\mu_{\int g}} &
& \text{colim}_X\Delta_X\dlto^{\cong}\\
& \Delta[n]\rrto^(.45){\sigma} & & X
}\]
The maps $(d\mu_{\int g})(\sigma)\gets PG(\sigma)\to G(\sigma)$ form natural transformations $d\mu_{\int g}\gets PG\to G$ connecting
$ \partial\int$ and the identity functor.
%$\text{id}:\text{\rm Fun}(X,F_{\text{\rm we}})\!\downarrow\! \Delta_X\ra \text{\rm Fun}(X,F_{\text{\rm we}})\!\downarrow\! \Delta_X$ and $ \partial\int :\text{\rm Fun}(X,F_{\text{\rm we}})\!\downarrow\! \Delta_X\ra \text{\rm Fun}(X,F_{\text{\rm we}})\!\downarrow\! \Delta_X$.
\end{proof}
%="H"

We can now prove our  first classification theorem for fibrations of  spaces:
\begin{thm}\label{thmA Ext is Map}
The category\/ $ \text{\rm Ext}(X,F)$ is essentially small and the nerve of its core is weakly equivalent to\/
$\text{\rm map}(X,B\text{\rm we}(F,F))$.
\end{thm}
\begin{proof}
According to~\ref{prop Decompint} and ~\ref{thm funintoxwe},  we need to show  
$\text{\rm Fun}(X,F_{\text{\rm we}})\!\downarrow\! \Delta_X$   is homotopy equivalent to
$\text{\rm Fun}(X,F_{\text{\rm we}})$.
Let  $\Phi\colon \text{\rm Fun}(X,F_{\text{\rm we}})\!\downarrow\! \Delta_X
\to \text{\rm Fun}(X,F_{\text{\rm we}})$  be the functor that forgets the augmentation. Its right adjoint
 $-\times \Delta_X\colon\text{\rm Fun}(X,F_{\text{\rm we}})\to\text{\rm Fun}(X,F_{\text{\rm we}})\!\downarrow\! \Delta_X$  assigns to  $G\colon X\to F_{\text{\rm we}}$  the projection
 $\text{pr}\colon G\times \Delta_X\to \Delta_X$. This  adjointness implies that $-\times \Delta_X$ is a homotopy inverse of
 $\Phi$.
% One can now use this adjointness to get natural transformations
 %$\text{fr}(-\times\Delta_X)\to \text{id}$ and $\text{id}\to (-\times \Delta_X)\text{fr}$.
  \end{proof}

To understand  $\text{Fib}(X,F)$ we  will use the range functor
$\calr\colon\text{Fib}(X,F)\to X_{\text{we}}$ that  assigns to a morphism $(\phi,\psi)$  in $\text{Fib}(X,F)$
(see~\ref{def intro FibXF}) the map $\psi$. 

\begin{prop}\label{prop rangeqf}
$\calr\colon\text{\rm Fib}(X,F)\to X_{\text{we}}$ is a quasi-fibration (see~\ref{def weqfhpull}).
\end{prop}
\begin{proof}
Let $B$ be in  $X_{\text{we}}$.
%Let $B$ be an object inÊ $X_{\text{we}}$\@.
%We first  show  $B\!\uparrow\!\calr$ is homotopy
%equivalent to $\text{\rm Ext}(B, F)$ 
%for any  $B$Ê in  $X_{\text{we}}$\@.
An  object in $B\!\uparrow\!\calr$ is a pair of maps 
$E\stackrel{p}{\to} B'\stackrel{\phi}{\gets} B$
%$\xymatrix@C=15pt@R=6pt{
%E \ar[r]|(.43)p & B' & \ar[l]|(.43)\phi B}$ 
 where $\phi$ is a weak equivalence and $p$ is  in $\text{\rm Fib}(X,F)$\@. A morphism 
 %in $B\!\uparrow\!\calr$  
 between  such objects is a pair of weak equivalences $(g,h)$ making  the  following diagram commutative: 
\begin{equation}\label{morphisminover}
\xymatrix@C=15pt@R=11pt{
E_1 \ar[r]^{p_1}\ar[d]_g & B_1 \ar[d]^h & B \ar[l]_{\phi_1}\ar@{=}[d] \\
E_0 \ar[r]^{p_0} & B_0 & B \ar[l]_{\phi_1}
}\tag{*}
\end{equation}

Define $\calf_B\colon\text{\rm Ext}(B,F) \to B\!\uparrow\!\calr$ to  be the functor given by the assignment:
\[\vcenter{\xymatrix@C=15pt@R=11pt{
E_1\rto^{p_1} \dto_{g}& B\ar@{=}[d]\\
E_0\rto^{p_0} & B
}}\ \ \stackrel{\calf_B}{\longmapsto}\ \ 
\vcenter{\xymatrix@C=15pt@R=15pt{
E_1\rto^{p_1} \dto_{g}& B\ar@{=}[d] &B\ar@{=}[d]\ar@{=}[l]\\
E_0\rto^{p_0} & B&B\ar@{=}[l]
}}\]
%We will show that $\calf_B$ is a homotopy equivalence.

Let $\calg_B\colon B\!\uparrow\!\calr\to \text{\rm Ext}(B,F)$  be the functor   that maps  a morphism
~\eqref{morphisminover} 
in $B\!\uparrow\!\calr$  to a morphism 
in $\text{\rm Ext}(B,F)$ described as follows.
Use the fibrant replacement $R$  to extend~\eqref{morphisminover}  to
a diagram on the left below (see~\ref{pt fibreplac}) and set  $\calg_B(g,h)\colon\calg_B(p_1,\phi_1)\to \calg_B(p_0,\phi_0)$  to be the 
unique morphism in $\text{\rm Ext}(B,F)$ that fits into the following commutative cube on the right where  the front and back
squares are pull-backs: 
\[
\vcenter{\vbox{\xymatrix@R=13pt@C=20pt{
E_1\dto_{g}\ar@/^14pt/[rr]|{p_1}\rmono^-{\simeq}  & R(p_1)\ar@{->>}[r]^{\mu_{p_1}} \dto|(.45){R(g,h)} &B_1 \dto^{h} & B \ar[l]_{\phi_1} \ar@{=}[d]\\
E_0\ar@/_13pt/[rr]|{p_0}\rmono^-{\simeq}  &R(p_0) \ar@{->>}[r]^{\mu_{p_0}} &B_0 & B \ar[l]_{\phi_0}
}}}\ \ \  \ \ \ 
\vcenter{\vbox{\xymatrix@!@C=10pt@R=-18pt{
 &\phi_1^{\ast}R(p_1)\dlto_{\calg_B(g,h)\ \ })\ar@{->>}[rr]^{\calg_B(p_1,\phi_1)}\ddto|\hole& & B 
 \ddto^{\phi_1}\ar@{=}[dl]\\
\phi_0^{\ast}R(p_0))\ar@{->>}[rr]^(.72){\calg_B(p_0,\phi_0)}\ddto  & & B\ddto^(.25){\phi_0}\\
& R(p_1)\ar@{->>}[rr]^(.32){\mu_{p_1}}|(.5)\hole\dlto_{R(g,h)} & & B_1 \dlto_{h}\\
R(p_0)\ar@{->>}[rr]^{\mu_{p_0}} & & B_0 
}}}\qedhere\]

For any object  $E\stackrel{p}{\to} B'\stackrel{\phi}{\gets} B$  in
 $B\!\uparrow\!\calr$ the following morphisms give a ``zig-zag'' of natural transformations between  $\calf_B \calg_B$ and the identity functor:
 \[\xymatrix@R=11pt@C=15pt{
 \calf_B\calg_B(p,\phi)\dto &  \phi^{\ast}R(p)\ar@{->>}[rr]^-{\calg_B(p,\phi)}\dto && B\dto^{\phi} &B\ar@{=}[l]\ar@{=}[d]\\
 (\mu_p,\phi) &R(p) \ar@{->>}[rr]^-{\mu_p} && B' &B\lto_-{\phi}  \\
 (p,\phi)\uto & E_0\rrto^p\ar@{^(->}[u]^{\simeq} & & B' \ar@{=}[u] & B\lto_-{\phi}\ar@{=}[u]
 }\]
 
 Instead of showing directly that the other composition $\calg_B\calf_B$ is homotopic to the identity, we  prove slightly more. Let $f\colon C\to B$ be a morphism in $X_{\text{we}}$ and $f_{\ast}\colon\text{\rm Ext}(C,F) \to \text{\rm Ext}(B,F)$ by the composition with $f$ functor:
\[\vcenter{\xymatrix@R=11pt@C=15pt{
E_1\rto^{p_1} \dto_{g} & C\ar@{=}[d]\\
E_0\rto^{p_0} & C}}\ \ \ \stackrel{f_{\ast}}{\longmapsto}\ \ \ 
\vcenter{\xymatrix@R=11pt@C=15pt{
E_1\rto^{fp_1} \dto_{g} & B\ar@{=}[d]\\
E_0\rto^{fp_0} & B
}}
\]
We claim that  $\xymatrix@1@R=15pt@C=18pt{ 
\text{\rm Ext}(B,F) \ar[r]^-{\calf_B} &  B\!\uparrow\!\calr \ar[r]^{f\uparrow\calr} & C\!\uparrow\!\calr\ar[r]^-{\calg_C} & \text{\rm Ext}(C,F)
}
$ is a homotopy inverse to $f_{\ast}$\@. The action of $f_{\ast} \calg_C(f\!\uparrow\!\calr) \calf_B$ on an object $p\colon E \to B$ in $\text{\rm Ext}(B,F)$  can be understood through the
following commutative diagram:
\[\xymatrix@R=18pt@C=40pt{
E \ar@/_5pt/[dr]_{p}\ar@{^(->}[r]^-{\simeq} &  R(p)\ar@{->>}[d]^{\mu_p} & & f^{\ast}R(p)
\dllto|{f_{\ast}\calg_C(f\uparrow\calr) \calf_B(p)}\llto\ar@{->>}[d]^{\calg_C(p,f)=\calg_C(f\uparrow\calr) \calf_B(p)}\\
 & B & &C\llto^{f}
}\]
The maps $E\hookrightarrow
 R(p)\leftarrow f^{\ast}R(p)$ give a ``zig-zag'' of natural transformations
between  $f_{\ast} \calg_C(f\uparrow\pi) \calf_B$ and  $\text{id}_{\text{\rm Ext}(B,F)}$.
 Conversely, by applying $\calg_C(f\uparrow\pi) \calf_B f_{\ast}$ to 
 an object  $q\colon E\to C$ in $\text{\rm Ext}(C, F)$,  we get a commutative diagram:
 \[\xymatrix@R=15pt{ 
 f^{\ast}R(fq)\ar@/_5pt/[dr]_(.3){\calg_C(f\uparrow\calr) \calf_Bf_{\ast}(q)=\calg_C(fq, f)\ }
 \ar@/^14pt/[rr] & E\dto^-{q}  \drto|{fq}\lto
 \ar@{^(->}[r]^-{\simeq}&
 R(fq)\ar@{->>}[d]^{\mu_{fq}}\\
  & C\rto^(.35){f} & B
 }\]
 where  $E\ra f^{\ast}R(fq)$ is the unique map into the pull-back 
 given by the commutativity of the inner square. These maps give a natural transformation
 between  $\calg_C(f\!\uparrow\!\calr) \calf_B f_{\ast}$ and  $\text{id}_{\text{\rm Ext}(C,F)}$.
For example if  $f=\text{\rm id}_B$,  then we see that $\calg_B\calf_B$ is homotopic to $ \text{\rm id}_{\text{\rm Ext}(B,F)}$.  Thus according to~\ref{thmA Ext is Map} and~\ref{prop esssmallweakequiv}, $f\!\uparrow\!\calr\colon B\!\uparrow\!\calr\to
 C\!\uparrow\!\calr$ is a weak equivalence of essentially small categories,  which shows the proposition.
\end{proof}

With~\ref{thmA Ext is Map} and~\ref{prop rangeqf}
we have the tools necessary to prove: 
%
%\begin{thmA}
%The category\/ $\text{\rm Fib}(X,F)$ has a core that admits a map to $B\text{\rm we}(X,X)$. This map has a section and its homotopy fiber is weakly equivalent to the mapping space $\text{\rm map}(X,B\text{\rm we}(F,F))$.
%\end{thmA}
\begin{proof}[Proof of Theorem A]
Consider the range functor $\calr\colon\text{\rm Fib}(X,F) \to X_{\text{\rm we}}$.  The categories
$\text{\rm Fib}(X,F)$ and $\text{\rm Gr}_{X_{\text{\rm we}}} (-\!\uparrow\!\calr)$ are homotopy equivalent.
(see~\ref{pt undercat}). According to~\ref{thm bweXX}, the category $X_{\text{\rm we}}$ is essentially small
and by~\ref{prop rangeqf} the system of categories $-\!\uparrow\!\calr$ satisfies the requirement of~\ref{lemma groveresssmallqf}.
The category $\text{\rm Gr}_{X_{\text{\rm we}}} (-\!\uparrow\!\calr)$ is therefore essentially small. Moreover,
there is a core $\Xi$ of    $X_{\text{\rm we}}$ and a core $\Psi$ 
of the restriction of the system  $(-\!\uparrow\!\calr)$ to $\Xi$ for which $\text{\rm Gr}_\Xi \Psi$ is a core of 
 $\text{\rm Gr}_{X_{\text{\rm we}}} (-\!\uparrow\!\calr)$. 
 We can now use Thomason's theorem~\ref{prop ThomasonPuppe}.(3)  to get that   the homotopy
 fiber of the the nerve of the projection $\text{\rm Gr}_\Xi \Psi \ra \Xi$ is weakly equivalent to
 the nerve of $\Psi_i$, which  
 is weakly equivalent to the mapping space $\text{map}(X,B\text{we}(F,F))$
 (see the proof of~\ref{prop rangeqf} and~\ref{thmA Ext is Map}). 
 The theorem follows as the projection $\text{\rm Gr}_\Xi \Psi \ra \Xi$ has a section
 and the nerve of $\Xi$ is weakly equivalent to $B\text{we}(X,X)$ (see~\ref{thm bweXX}).
\end{proof}

%%%%%%%%%%%%%%%%%%%%%%%%%%%%%%%%%%%%%%
%%%%%%%%%%%%%%%%%%%%%%%%%%%%%%%%%%%%%%
%%%%%%%%%%%%%%%%%%%%%%%%%%%%%%%%%%%%%%
%%%%%%%%%%%%%%%%%%%%%%%%%%%%%%%%%%%%%%
%%%%%%%%%%%%%%%%%%%%%%%%%%%%%%%%%%%%%%
%%%%%%%%%%%%%   APPENDIX  %%%%%%%%%%%%%%%%%%
%%%%%%%%%%%%%%%%%%%%%%%%%%%%%%%%%%%%%%
%%%%%%%%%%%%%%%%%%%%%%%%%%%%%%%%%%%%%%
%%%%%%%%%%%%%%%%%%%%%%%%%%%%%%%%%%%%%%
%%%%%%%%%%%%%%%%%%%%%%%%%%%%%%%%%%%%%%

 \section{Appendix: delooping of homotopy groupoids}\label{appendix}
  In this appendix we recall the    standard delooping machinery and prove~\ref{prop deloopingwe}. %Recall that $[n]$  denotes the poset $\{0<1<\ldots< n\}$.
  
 \begin{Def} \label{def hgroupoid}
 Let $\mathcal{S}$ be a set. 
 A homotopy groupoid  indexed by  $\mathcal{S}$ consists of:
 %$\mathcal{S}$.
 %$(r,s,t)$ of elements in $\mathcal{S}$;
 \begin{itemize}
  \item   fibrant-cofibrant spaces 
 $G(r,t)$ indexed by   pairs $(r,t)$ of elements in $\mathcal{S}$;
 \item  maps $\circ\colon G(r,s)\times G(s,t)\to G(r,t)$ %called compositions
  indexed by triples
 $(r,s,t)$  in $\mathcal{S}$;
 \item maps $e_r\colon\Delta[0]\to  G(r,r)$ indexed by elements $r$ in 
 $\mathcal{S}$.
 \end{itemize}
 These sequences are required to  satisfy the following properties:
 \begin{enumerate}
 \item (associativity) for any  $r,s,t,v$ in $\mathcal{S}$ 
 the following diagram commutes:
 \[\xymatrix@R=11pt@C=18pt{
 G(r,s)\times G(s,t)\times G(t,v)\rto^(.58){\circ\times\text{id}}\dto_{\text{id}\times\circ} & G(r,t)\times G(t,v)\dto^{\circ}\\
 G(r,s)\times G(s,v)\rto^(.6){\circ} & G(r,v)
 }\]
 \item (identity) for any  $r,s$ in $\mathcal{S}$ the following
 diagrams commute:
 \[\xymatrix@R=11pt@C=18pt{
\Delta[0]\times G(r,s)\rto^(.45){e_r\times\text{id}}\drto_{\text{pr}} &  G(r,r)\times G(r,s)\dto^{\circ}\\
& G(r,s)
 }\ \ \ 
 \xymatrix@R=12pt@C=18pt{
 G(r,s)\times G(s,s)\dto_{\circ} & G(r,s)\times\Delta[0]\lto_(.45){\text{id}\times e_s}\dlto^{\text{pr}}\\
 G(r,s)
  }
 \]
 \item  for any  $r,s, t$  in $\mathcal{S}$ the following square is a homotopy pull-back:
 \[\xymatrix@R=11pt@C=15pt{
 G(r,s)\times G(s,t)\rto^-{\circ}\dto_{\text{pr}} & G(r,t)\dto\\
 G(r,s)\rto & \Delta[0]
 }\]
 \end{enumerate}
 \end{Def}
 A homotopy groupoid is simply a small category %($\cals$ is the set of its objects) 
 enriched
 over $\text{Spaces}$ %(requirements (1) and (2) above) 
 with an additional assumption given by requirement (3). %This last requirement  is  modelled after the observation that a set $G$ endowed with an associative composition, $\circ$, 
%and a two-sided identity is a group if and only if the following diagram is a pull-back: 
%\[
%\xymatrix@R=12pt@C=15pt{
%G\times G \ar[d]_{\text{\rm projection}}\ar[r]^(.58)\circ & G \ar[d] \\
%G \ar[r] & 1 
%}
%\]
A homotopy groupoid $G$ indexed by $\cals$ is also denoted by $G_\cals$.
  If all the spaces $G(s,t)$ are not empty, $G$ is called {\bf connected}.
In this case the requirement (3) of~\ref{def hgroupoid} 
implies that   all these spaces  are weakly equivalent to each other.
% for any  $r$, $s$, $t$ and $v$ in $S$. 
This unique homotopy type is  called the homotopy type of the connected 
groupoid $G$.
For example
let    $\mathcal{S}$  be a set of objects in $\text{\rm Cons}(N(\Delta[0]),\calm)$ (see~\ref{pt hcons}).  The  spaces $\text{\rm we}(F,G)$ with the composition operations $\circ$ as defined in Section~\ref{sec we} and  the identities
 given  by the maps   $e_F\colon\Delta[0]\to \text{\rm we}(F,F)$ 
 form a homotopy groupoid denoted by $\text{we}(\cals)$. The requirements (1) and (2) are the content of~\ref{cor propertiesofwe}.(1), and (3) of~\ref{prop wehgoup}.
%The requirement (3) is the content of~\ref{prop wehgoup}. 
%If  objects in $\cals$ are weakly equivalent to each other, then $\text{we}(\cals)$ is connected.
If all objects in $\cals$ are weakly equivalent to each other,   $\text{we}(\cals)$ is connected.

The definition of a homotopy groupoid $G_{\cals}$ is designed so  we can form a  so called  {bar construction}. It is
a map of simplicial spaces $\pi\colon {\mathcal E}G\to  {\mathcal B}G$ and here is how to construct it.
 Let  $t_{-1}\in\mathcal{S}$ be an arbitrary but fixed element and %two simplicial spaces ${\mathcal B}G$, ${\mathcal E}G$ and 
  $n\geq 0$.  For any tuple $(t_n,\ldots, t_{0})\in\mathcal{S}^{n+1}$ set:
\[\xymatrix@R=8pt@C=40pt{
{\mathcal E}G_{t_n,\ldots, t_0}\ar@{*{\cdot\cdot}=}[d]\rto^{\pi_{t_n,\ldots, t_0}} & 
 {\mathcal B}G_{t_n,\ldots, t_0}\ar@{*{\cdot\cdot}=}[d]\\
\displaystyle{\prod_{k=n}^{k=0} G(t_{k},t_{k-1}}) \rto^(.5){\text{projection}} &\displaystyle{\prod_{k=n}^{k=1} G(t_{k},t_{k-1}}) 
}\]
Thus ${\mathcal E}G_{t_n,\ldots, t_0}= {\mathcal B}G_{t_n,\ldots, t_0}\times G(t_0,t_{-1})$
 and $\pi_{{t_n,\ldots, t_0}}$ is the projection.
For example, in the case $n=0$, since  the product of an empty set of spaces is  $\Delta[0]$, 
$\pi_t\colon {\mathcal E}G_{t}=G(t,t_{-1})\to\Delta[0]={\mathcal B}G_{t}$ is the unique map.
By assembling all these spaces we define:
\[
\xymatrix@R=8pt@C=95pt{
{\mathcal E}G_n\ar@{*{\cdot\cdot}=}[d]\ar@{->}[r]^{\pi_n} & {\mathcal B}G_n\ar@{*{\cdot\cdot}=}[d]\\
\displaystyle{\coprod_{(t_n,\ldots, t_0)\in\mathcal{S}^{n+1}}}{\mathcal E}G_{t_n,\ldots, t_0}
\ar@{->}[r]^{{\coprod_{(t_n,\ldots, t_0)\in\mathcal{S}^{n+1}}}\pi_{t_n,\ldots, t_0}} & 
\displaystyle{\coprod_{(t_n,\ldots, t_0)\in\mathcal{S}^{n+1}}}{\mathcal B}G_{t_n,\ldots, t_0}
}\]

Let $\phi\colon [m]\to [n]$ be a morphism in $\Delta$. Set $\phi(-1)=-1$. 
%For any
%$(t_n,\ldots, t_0)\in  S^{n+1}$, we call $(t_{\phi(m)},\ldots, t_{\phi(0)})\in  S^{m+1}$ the tuple  induced by $\phi$.
 Define: 
\[ {\prod_{k=n}^{k=1} G(t_{k},t_{k-1}})={\mathcal B}G_{t_n,\ldots, t_0}\xrightarrow{{\mathcal B}G_{t_n,\ldots, t_0,\phi}} {\mathcal B}G_{t_{\phi(m)},\ldots, t_{\phi(0)}}= {\prod_{k=m}^{k=1} G(t_{\phi(k)},t_{\phi(k-1)}}) \]
\[ {\prod_{k=n}^{k=0} G(t_{k},t_{k-1}})={\mathcal E}G_{t_n,\ldots, t_0}\xrightarrow{{\mathcal E}G_{t_n,\ldots, t_0,\phi}} {\mathcal E}G_{t_{\phi(m)},\ldots, t_{\phi(0)}}= {\prod_{k=m}^{k=0} G(t_{\phi(k)},t_{\phi(k-1)}}) \]
to  be  the maps whose projections onto  $G(t_{\phi(i)},t_{\phi(i-1)})$ are given by
 the following compositions in the case $\phi(i)=\phi(i-1)$:
\[{\mathcal B}G_{t_n,\ldots, t_0}\ra \Delta[0]\xrightarrow{e_{\phi(i)}}G(t_{\phi(i)},t_{\phi(i-1)})\ \ \ \ \ {\mathcal E}G_{t_n,\ldots, t_0}\ra \Delta[0]\xrightarrow{e_{\phi(i)}}G(t_{\phi(i)},t_{\phi(i-1)})\]
and  the following compositions in the case   $\phi(i)>\phi(i-1)$:
\[{\mathcal B}G_{t_n,\ldots, t_0}\xrightarrow{\text{projection}}{\prod_{k=\phi(i)}^{k=\phi(i-1)+1} G(t_{k},t_{k-1}})\xrightarrow{\text{composition}} G(t_{\phi(i)},t_{\phi(i-1)})\]
\[{\mathcal E}G_{t_n,\ldots, t_0}\xrightarrow{\text{projection}}{\prod_{k=\phi(i)}^{k=\phi(i-1)+1} G(t_{k},t_{k-1}})\xrightarrow{\text{composition}} G(t_{\phi(i)},t_{\phi(i-1)})\]

For any  $\phi\colon [m]\to [n]$ in $\Delta$, define
${\mathcal B}G_{\phi}\colon {\mathcal B}G_n\to {\mathcal B}G_m$ and
${\mathcal E}G_{\phi}\colon{\mathcal E}G_n\to {\mathcal E}G_m$ to be the maps
which on components are given by ${\mathcal B}G_{t_n,\ldots, t_0,\phi}$ and ${\mathcal E}G_{t_n,\ldots, t_0,\phi}$.
The requirements (1) and (2) of~\ref{def hgroupoid} are exactly what is needed for  ${\mathcal B}G$ and ${\mathcal E}G$  to be simplicial spaces (see for example~\cite{MR0370579, MR0232393}). From the above definition it is also clear that 
 $\pi\colon {\mathcal E}G\to {\mathcal B}G$ is a  map of simplicial spaces.
Note further that if $S'\subset S$, then ${\mathcal E}G_{S'}$ and ${\mathcal B}G_{S'}$ are  simplicial subspaces of
 ${\mathcal E}G_{S}$ and  ${\mathcal B}G_{S}$.
 \begin{prop}\label{prop appendix deloopingwe}
Let $\mathcal{S}'\subset \mathcal{S}$ be non-empty sets and $G_{\mathcal{S}}$ be a connected homotopy groupoid indexed by $ \mathcal{S}$.  Let $G_{\mathcal{S}'}$ be the restriction of $G_{\mathcal{S}}$ to  $\mathcal{S}'$.
Then:
\begin{enumerate} 
\item  $\text{\rm hocolim}_{\Delta^{\text{\rm op}}} \mathcal{B}G_{\mathcal{S}}$ is connected 
and\/  $\text{\rm hocolim}_{\Delta^{\text{\rm op}}} \mathcal{E}G_{\mathcal{S}}$ is contractible.
%\item the space  $\text{\rm hocolim}_{\Delta^{\text{\rm op}}} \mathcal{E}G_{\mathcal{S}}$ is contractible;
\item $G_{\mathcal{S}}$ has the homotopy tope of the  the loop space  $\Omega \text{\rm hocolim}_{\Delta^{\text{\rm op}}} \mathcal{B}G_{\mathcal{S}}$.
% has the  homotopy type of the homotopy groupoid $G_{\mathcal{S}}$;
\item the map\/ $\text{\rm hocolim}_{\Delta^{\text{\rm op}}} \mathcal{B}G_{\mathcal{S}'}\to \text{\rm hocolim}_{\Delta^{\text{\rm op}}} \mathcal{B}G_{\mathcal{S}}$, induced by the inclusion of simplicial spaces $\mathcal{B}G_{\mathcal{S}'}\subset \mathcal{B}G_{\mathcal{S}}$,
is a weak equivalence.
\end{enumerate}
\end{prop}
In the rest of this section we  prove~\ref{prop appendix deloopingwe}.  We start with  two classical lemmas
whose proves can be found for example in~\cite{MR2064057} and~\cite{MR2222504}.

 \begin{lemma}\label{lemma extra degeneracy}
Let $X$ be a simplicial space. Set $d_0\colon X_{0}\to \Delta[0]=:X_{-1}$ to be the
unique map. 
Assume that there are maps $s\colon X_n\to X_{n+1}$ for $n\geq -1$ such that  $d_0s = \text{\rm id}$
and  $d_{i}s= sd_{i-1}$ for $i>0$.
Then\/ $\text{\rm hocolim}_{\Delta^{\text{op}}} X$  is contractible. 
\end{lemma}
%\begin{proof}
%See, for example, \cite{MR2064057}.
%\end{proof}
%The hypothesis in Lemma~\ref{lemma extra degeneracy} is sometimes known as an {\bf extra degeneracy}.
\begin{lemma}\label{prop puppe homotopy square}
Let $\psi\colon F\to G$ be a natural transformation in\/ $\text{\rm Fun}(I,\text{\rm Spaces})$.
Assume the commutative diagram on the left below is a 
homotopy pull-back for any  $\alpha\colon i\to j$ in $I$. Then, for any  $k$ in $I$, the  diagram on the right  is  a 
homotopy pull-back, where the horizontal maps are induced by the inclusion of  $k$ in $I$.
\[
\xymatrix@R=11pt@C=18pt{
F(i) \ar[d]_{\psi_i}\ar[r]^{F(\alpha)} & F(j) \ar[d]^{\psi_j} \\
G(i) \ar[r]^{G(\alpha)} & G(j)
}\ \ \ \ \ \ \ \ \ \ \ 
\xymatrix@R=11pt@C=15pt{
F(k) \ar[d]_{\phi_k} \ar[r] & \text{\rm hocolim}_I F \ar[d]^{ \text{\rm hocolim}_I \psi} \\
G(k) \ar[r] & \text{\rm hocolim}_I G
}
\]
\end{lemma}
%\begin{proof}
%See, for example, \cite{MR2222504}.
%\end{proof}
\begin{proof}[Proof of Proposition~\ref{prop appendix deloopingwe}]
(1):\quad  $\text{\rm hocolim}_{\Delta^{\text{op}}} \mathcal{B}G_{\mathcal{S}}$ is weakly equivalent to
the diagonal of  $\mathcal{B}G_{\mathcal{S}}$. 
For two vertices  $s,t\in \coprod_{t_0\in\cals} \Delta[0]_0=\text{\rm diag}(BG_\cals)_0$,
pick en element  $f$ in $ G(t,s)_0$ and set  $v=s_0f\in \text{\rm diag}(BG_\cals)_1$. Note that
$d_0v=s$ and $d_1v=t$. The space $\text{\rm diag}(BG_\cals)$ is therefore connected
 and hence so is  $\text{\rm hocolim}_{\Delta^{\text{op}}} \mathcal{B}G_{\mathcal{S}}$.

Let  $t_n$, \ldots, $t_0$ be in $\cals$. 
Define   $s_{t_n,\ldots,t_0}\colon\cale G_{t_n,\ldots,t_0}\to \cale G_{t_n,\ldots,t_0, t_{-1}}$  as:
\[\xymatrix@R=11pt@C=15pt{\cale G_{t_n,\ldots,t_0}\ar@{=}[r]\drto_{s_{t_n,\ldots,t_0}} &\displaystyle{\prod_{k=n}^{k=0}}G(t_k,t_{k-1})\ar@{=}[r] &\displaystyle{\prod_{k=n}^{k=0}}G(t_k,t_{k-1})\times\Delta[0]\dto^{\text{id}\times e_{t_{-1}}}\\
&\cale G_{t_n,\ldots,t_0,t_{-1}}\ar@{=}[r] &\displaystyle{\prod_{k=n}^{k=0}}G(t_k,t_{k-1})\times G(t_{-1},t_{-1}) 
}\]
Let  $s\colon\cale G_{n}\to \cale G_{n+1}$  be  given on components  by  $s_{t_n,\ldots,t_0}$\@.
Set  $\cale G_{-1}=\Delta[0]$ and $s:\cale G_{-1}=\Delta[0]\ra \cale G_{0} $ to be  the composition of $e_{t_{-1}}\colon\Delta[0]\to  G(t_{-1},t_{-1})$ and the inclusion $G(t_{-1},t_{-1})\subset \coprod_{t\in S} G(t,t_{-1})=
\cale G_{0}$. Since these maps satisfy the assumptions of~\ref{lemma extra degeneracy},
we can conclude  $\text{\rm hocolim}_{\Delta^{\text{\rm op}}} \mathcal{E}G_{\mathcal{S}}$ is contractible.
\smallskip

\noindent
(2):\quad 
The requirement (3) in~\ref{def hgroupoid} implies 
the square on the left below is  a homotopy pull-back. Since being a homotopy pull-back can be checked component-wise, the square  on the right is  also a homotopy pull-back:
\[\xymatrix@R=11pt@C=60pt{
{\mathcal E}G_{t_n,\ldots, t_0}\rto^(.45){{\mathcal E}G_{t_n,\ldots, t_0,\phi}}\dto_{\pi_{t_n,\ldots, t_0}} &
{\mathcal E}G_{t_{\phi(m)},\ldots, t_{\phi(0)}}\dto^{\pi_{t_{\phi(m)},\ldots, t_{\phi(0)}}}
\\
{\mathcal B}G_{t_n,\ldots, t_0}\rto^(.45){{\mathcal B}G_{t_n,\ldots, t_0,\phi}} &
{\mathcal B}G_{t_{\phi(m)},\ldots, t_{\phi(0)}}
}\ \ \ \  \ \ \ \ 
\xymatrix@R=11pt@C=18pt{
{\mathcal E}G_m\rto^{{\mathcal E}G_{\phi}}\dto_{\pi_{m}} & 
{\mathcal E}G_n\dto^{\pi_{n}}\\
{\mathcal B}G_m\rto^{{\mathcal B}G_{\phi}} & {\mathcal B}G_n
}
\]
We can then use~\ref{prop puppe homotopy square} to conclude that, for any $t$ in $S$, the  square:
\[\xymatrix@R=11pt@C=15pt{
G(t,t_{-1})\ar@{=}[r] & {\mathcal E}G_{t}\rto\dto_{\pi_t} & \text{\rm hocolim}_{\Delta^{\text{\rm op}}} {\mathcal E}G
\dto^{\text{\rm hocolim}_{\Delta^{\text{\rm op}}}\pi}\\
\Delta[0]\ar@{=}[r] &  {\mathcal B}G_{t}\rto &\text{\rm hocolim}_{\Delta^{\text{\rm op}}} {\mathcal B}G
}
\]
  is a homotopy pull-back which  proves statement (2).
\smallskip

  \noindent
(3):\quad  Let $t$ be a object in $\cals'$. Consider the following commutative diagram:
\[\xymatrix@R=13pt@C=15pt{
G(t,t_{-1})\ar@{=}[r] & {\mathcal E}G_{t}\rto\dto_{\pi_t} & \text{\rm hocolim}_{\Delta^{\text{op}}} {\mathcal E}G_{\cals '}
\dto|{\text{\rm hocolim}_{\Delta^{\text{op}}}\pi}\rto &\text{\rm hocolim}_{\Delta^{\text{op}}} {\mathcal E}G_{\cals}
\dto|{\text{\rm hocolim}_{\Delta^{\text{op}}}\pi} \\
\Delta[0]\ar@{=}[r] &  {\mathcal B}G_{t}\rto &\text{\rm hocolim}_{\Delta^{\text{op}}} {\mathcal B}G_{\cals '}
\rto 
&\text{\rm hocolim}_{\Delta^{\text{op}}} {\mathcal B}G_{\cals }
}
\]
The squares in this diagram are homotopy pull-backs,
the spaces $\text{\rm hocolim}_{\Delta^{\text{op}}} {\mathcal B}G_{\cals '}$ and 
$\text{\rm hocolim}_{\Delta^{\text{op}}} {\mathcal B}G_{\cals }$ are connected, and 
$\text{\rm hocolim}_{\Delta^{\text{op}}} {\mathcal E}G_{\cals '}$ and  
$\text{\rm hocolim}_{\Delta^{\text{op}}} {\mathcal E}G_{\cals }$  are contractible.
Thus  $\text{\rm hocolim}_{\Delta^{\text{op}}} {\mathcal B}G_{\cals '}\to \text{\rm hocolim}_{\Delta^{\text{op}}} {\mathcal B}G_{\cals }$  is   a weak equivalence.
\end{proof}

%\begin{proof}[Proof of Proposition~\ref{prop deloopingwe}]
%Note that for any two objects $F$ and $G$ in the category $\text{Cons}(N(\Delta[0]),X_{\text{we}})$,
%the space $\text{we}(F,G)$ is not empty as its set of components is in bijection with
%the set of isomorphisms between $F(\sigma)$ and $G(\sigma)$  in the homotopy category $\text{Ho}(\calm)$.  It follows that, for any   set $\cals$  of objects in $\text{Cons}(N(\Delta[0]),X_{\text{we}})$, 
%the groupoid $\text{we}(\cals)$ (see Example~\ref{example weS}) is connected. 
%We can then apply Proposition~\ref{prop appendix deloopingwe}.
%\end{proof}

\bibliographystyle{plain}\label{biblography}
\bibliography{bibho}

\end{document}